\documentclass{article}
\usepackage[utf8]{inputenc}
\usepackage[utf8]{inputenc}
\usepackage[utf8]{inputenc}
\usepackage{dirtytalk}
\usepackage{tikz}
\usepackage{geometry}
\usepackage{graphicx}
\usepackage{enumitem}
\usepackage{parskip}
\usepackage{amsfonts}
\usepackage{amssymb}
\usepackage{amsmath}
\usepackage{amsthm}
\usepackage{xcolor}
\usepackage{hyperref}
\newlist{myitemize}{itemize}{1}
\setlist[myitemize,1]{leftmargin = 0.5in}
\hypersetup{colorlinks = true, linkcolor = red,  linktocpage}

\theoremstyle{plain}
\newtheorem{thm}{Theorem}[section]
\newtheorem*{thm*}{Theorem}

\newtheorem{cor}[thm]{Corollary}

\theoremstyle{definition}

\newtheorem{conj}[thm]{Conjecture}

\newtheorem{rem}[thm]{Remark}

\title{\textbf{\small{COUNTING THE NUMBER OF $m$-PERIODIC $\mathcal{O}_{K}$-POINTS OF A DISCRETE DYNAMICAL SYSTEM WITH APPLICATIONS FROM ARITHMETIC STATISTICS, V}}}
\author{\footnotesize{BRIAN KINTU}}
\date{\small{\textit{February 22, 2026}}}

\begin{document}
\maketitle
\begin{abstract}
\footnotesize{In this follow-up paper, we again inspect a surprising relationship between the set of $m$-periodic points of a polynomial map $\varphi_{d, c}$ defined by $\varphi_{d, c}(z) = z^d + c$ for all $c, z \in \mathcal{O}_{K}$ and the coefficient $c$, where $K$ is any number field of degree $n\geq 2$, $d>2$ is an integer and $m\in \mathbb{Z}_{\geq 2}$ is any fixed (period). As in \cite{BK11} we wish to study counting problems which are inspired by advances on $m$-torsion point-counting in arithmetic statistics and $m$-periodic point-counting in arithmetic dynamics. In doing so, we then first prove that for any prime $p\geq 3$ and for any fixed $\ell\in \mathbb{Z}_{ \geq 1}$ and (period) $m\in \mathbb{Z}_{\geq 2}$, the average number of distinct $m$-periodic integral points of any $\varphi_{p^{\ell}, c}$ modulo prime ideal $p\mathcal{O}_{K}$ is unbounded or zero as $c$ tends to infinity; and so the average behavior here coincide with the average behavior of the number of distinct fixed points in \cite{BK3}. Motivated further by $K$-rational periodic point-counting work of Benedetto along with  conjectural work of Hutz on $m$-periodic points of any $\varphi_{(p-1)^{\ell}, c}$ for any prime $p\geq 5$ and any fixed $\ell \in \mathbb{Z}_{\geq 1}$ in arithmetic dynamics, we then also prove that for any fixed (period) $m\in \mathbb{Z}_{\geq 2}$, the average number of distinct $m$-periodic integral points of any $\varphi_{(p-1)^{\ell}, c}$ modulo prime $p\mathcal{O}_{K}$ is $1$ or $2$ or $0$ as $c\to \infty$; and so the average behavior here also coincide with the average behavior of the number of distinct fixed points in \cite{BK1}. Finally, we then apply here density, polynomial-counting, field-counting, and Sato-Tate equidistribution results from arithmetic statistics, and thereby obtaining further counting and statistical results on the irreducible polynomials, Artin-Mazur zeta functions, algebraic number fields, and lastly on Artin $L$-functions arising naturally in our polynomial discrete dynamical settings.}
\end{abstract}

\begin{center}
\tableofcontents
\end{center}

\begin{center}
    \section{Introduction}\label{sec1}
\end{center}
\noindent
Given any morphism $\varphi: {\mathbb{P}^N(K)} \rightarrow {\mathbb{P}^N(K)} $ of degree $d \geq 2$ defined on a projective space ${\mathbb{P}^N(K)}$ of dimension $N$, where $K$ is a number field. Then for any $n\in\mathbb{Z}$ and $\alpha\in\mathbb{P}^N(K)$, we then call $\varphi^n = \underbrace{\varphi \circ \varphi \circ \cdots \circ \varphi}_\text{$n$ times}$ the $n^{th}$ \textit{iterate of $\varphi$} and call $\varphi^n(\alpha)$ the \textit{$n^{th}$ iteration of $\varphi$ on $\alpha$}. By convention, $\varphi^{0}$ acts as the identity map, i.e., $\varphi^{0}(\alpha) = \alpha$ for every point $\alpha\in {\mathbb{P}^N(K)}$. As before, the everyday philosopher may want to know (quoting here Devaney \cite{Dev}): \say{\textit{Where do points $\alpha, \varphi(\alpha), \varphi^2(\alpha), \ \cdots\ ,\varphi^n(\alpha)$ go as $n$ becomes large, and what do they do when they get there?}} Now for any given integer $n\geq 0$ and any given point $\alpha\in {\mathbb{P}^N(K)}$, we then call the set consisting of all the iterates $\varphi^n(\alpha)$ the \textit{(forward) orbit of $\alpha$}; and which in dynamical systems we usually denote it by $\mathcal{O}^{+}(\alpha)$.

As mentioned in \cite{BK2} that one of the main 
goals in arithmetic dynamics (a newly emerging area of mathematics concerned with studying number-theoretic properties of discrete dynamical systems) is to classify all the points $\alpha\in\mathbb{P}^N(K)$ according to the behavior of their forward orbits $\mathcal{O}^{+}(\alpha)$. In this direction, we recall that any point $\alpha\in {\mathbb{P}^N(K)}$ is called a \textit{periodic point of $\varphi$}, whenever $\varphi^n (\alpha) = \alpha$ for some integer $n\in \mathbb{Z}_{\geq 0}$. In this case, we recall that any integer $n\geq 0$ such that the iterate $\varphi^n (\alpha) = \alpha$, is called \textit{period of $\alpha$}; and the smallest such positive integer $n\geq 1$ is called the \textit{exact period of $\alpha$}. We recall Per$(\varphi, {\mathbb{P}^N(K)})$ to denote set of all periodic points of $\varphi$; and also recall that for any given point $\alpha\in$Per$(\varphi, {\mathbb{P}^N(K)})$ the set of all iterates of $\varphi$ on $\alpha$ is called \textit{periodic orbit of $\alpha$}. In their 1994 paper \cite{Russo} and in his 1998 paper \cite{Poonen} respectively, Walde-Russo and Poonen give independently interesting examples of rational periodic points of any $\varphi_{2,c}$ defined over the field $\mathbb{Q}$; and so the interested reader may wish to revisit \cite{Russo, Poonen} to gain familiarity with the notion of periodicity of points. 

Previously in article \cite{BK3} we (inspired by work of Bhargava-Shankar-Tsimerman (BST) in arithmetic statistics (a branch of number theory concerned with counting and distribution of arithmetic objects) and also by Conjecture \ref{per} of Morton-Silverman in arithmetic dynamics) proved that the number of distinct integral fixed points of any $\varphi_{p^{\ell},c}$ modulo prime $p\mathcal{O}_{K}$ (for every $\ell \in \{1,p\}$) is $p$ or zero; from which it then followed that the average number of distinct integral fixed points of any $\varphi_{p^{\ell},c}$ modulo $p\mathcal{O}_{K}$ is unbounded or zero as $c\to \infty$. Moreover, we then also observed in \cite{BK3} that the expected total number of distinct integral fixed points in the whole family of maps $\varphi_{p^{\ell},c}$ modulo $p\mathcal{O}_{K}$ (for every $\ell \in \{1,p\}$) is equal to $p+0=p$; which may grow to infinity when degree $p^{\ell}\to \infty$. Later in article \cite{BK11} we (inspired by Mazur \cite{Maz} and (BST) on $m$-torsion point-counting in arithmetic statistics, along with Conjecture \ref{per} in arithmetic dynamics) proved that the number of distinct $m$-periodic points of any $\varphi_{p,c}$ modulo $p$ is $p$ or zero; from which it then followed that the average number of distinct $m$-periodic integral points of any $\varphi_{p,c}$ modulo $p$ is also unbounded or zero as $c\to \infty$. Moreover, we then also observed in [\cite{BK11}, Remark 2.3] that the expected total number of distinct $m$-periodic integral points in the whole family of maps $\varphi_{p,c}$ modulo $p$ is also equal to $p+0=p$ for every fixed period $m\in \mathbb{Z}_{\geq 2}$; which may also grow to infinity when $p\to \infty$. So now, motivated by Artin-Mazur \cite{AM} on periodic orbits and (BST) on $m$-torsion point-counting in arithmetic statistics, along with Conjecture \ref{per} in arithmetic dynamics, we revisit \cite{BK3, BK11} and then consider in Sect.\ref{sec2} any $\varphi_{p^{\ell},c}$ iterated on $\mathcal{O}_{K}\slash p\mathcal{O}_{K}$. In doing so, we then prove the following main theorem on $\varphi_{p,c}$, which we state later more precisely as Theorem \ref{2.2} and generalized further as Theorem \ref{2.3}; and moreover restricting on $\mathbb{Z} \subset \mathcal{O}_{K}$, we then obtain a more generalization Corollary \ref{cor2.4} of [\cite{BK11}, Thm.2.2]:

\begin{thm}\label{BB} 
Let $K\slash \mathbb{Q}$ be any number field of degree $ n \geq 2$ with the ring of integers $\mathcal{O}_{K}$, and in which any fixed prime integer $p\geq 3$ is inert. Let $m\geq 2$ be any fixed integer, and $\varphi_{p, c}$ be a map defined by $\varphi_{p, c}(z) = z^p + c$ for all $c, z\in\mathcal{O}_{K}$. Then the number of distinct $m$-periodic integral points of any $\varphi_{p,c}$ modulo $p\mathcal{O}_{K}$ is $p$ or zero. 
\end{thm}

Recall further in article \cite{BK2} we (again inspired by (BST)'s work in arithmetic statistics, and also by conjectural work \ref{conjecture 3.2.1} of Hutz along with Panraksa's work \cite{par2} in arithmetic dynamics) proved that the number of distinct integral fixed points of any $\varphi_{(p-1)^{\ell},c}$ modulo prime $p\mathcal{O}_{K}$ is equal to $1$ or $2$ or $0$; from which it then followed that the average number of distinct integral fixed points of any $\varphi_{(p-1)^{\ell},c}$ modulo $p\mathcal{O}_{K}$ is also $1$ or $2$ or $0$ as $c\to \infty$. Moreover, we then also observed in [\cite{BK2}, Remark 3.5] that the expected total number of distinct integral fixed points in the whole family of maps $\varphi_{(p-1)^{\ell},c}$ modulo $p\mathcal{O}_{K}$ is a constant equal to $1 + 2+ 0=3$ even when degree $(p-1)^{\ell}$. Later in \cite{BK11} we (motivated by Mazur's work \cite{Maz} and (BST)'s work on $n$-torsion point-counting in arithmetic statistics, along with Hutz's Conjecture \ref{conjecture 3.2.1} in arithmetic dynamics) proved that the number of distinct $n$-periodic integral points of any $\varphi_{p-1,c}$ modulo $p$ is equal to $1$ or $2$ or $0$; from which it then followed that the average number of distinct $n$-periodic points of any $\varphi_{p-1,c}$ modulo $p$ is also $1$ or $2$ or $0$ as $c\to \infty$. Moreover, we then also observed in [\cite{BK11}, Remark 3.3] that the expected total number of distinct $n$-periodic integral points in the whole family of maps $\varphi_{p-1,c}$ modulo $p$ is also a constant equal to $1 + 2+ 0=3$ for every fixed odd period $n\in \mathbb{Z}_{\geq 3}$ (or equal to $1 + 1 + 2+ 0=4$ for every fixed even period $n\in \mathbb{Z}_{\geq 2}$) even when $p-1\to \infty$. So now, motivated again by that same work of Artin-Mazur \cite{AM} on periodic orbits and again by advances on torsion point-counting in arithmetic statistics and periodic point-counting in arithmetic dynamics, we revisit the setting in Section \ref{sec2} and then consider in Section \ref{sec3} any even degree polynomial map $\varphi_{(p-1)^{\ell},c}$ iterated on the space $\mathcal{O}_{K}\slash p\mathcal{O}_{K}$. In doing so, we then also prove the following main theorem on any $\varphi_{p-1,c}$, which we state later more precisely as Theorem \ref{3.2} and generalized further as Theorem \ref{3.3}; and moreover which when we restrict on $\mathbb{Z} \subset \mathcal{O}_{K}$ , we then also obtain a further generalization Corollary \ref{cor3.4} of [\cite{BK11}, Theorem 3.2]:
\newpage
\begin{thm}\label{Binder-Brian}
Let $K\slash \mathbb{Q}$ be any number field of degree $n\geq 2$ with the ring of integers $\mathcal{O}_{K}$, and in which any fixed prime $p\geq 5$ is inert. Let $m\geq 2$ be any fixed integer, and $\varphi_{p-1, c}$ be defined by $\varphi_{p-1, c}(z) = z^{p-1} + c$ for all $c, z\in\mathcal{O}_{K}$. Then the number of distinct $m$-periodic integral points of any $\varphi_{p-1,c}$ modulo $p\mathcal{O}_{K}$ is $1$ or $2$ or zero.
\end{thm}

\noindent Notice that the count obtained in Theorem \ref{Binder-Brian} and more precisely in Theorem \ref{3.2} on the number of distinct $m$-periodic integral points of any $\varphi_{p-1,c}$ modulo $p\mathcal{O}_{K}$ is independent of $p$ (and hence independent of  deg$(\varphi_{p-1,c})$) and degree $n=[K:\mathbb{Q}]$ in each of the possibilities. Moreover, we may also observe that the expected total count (namely, $1 + 2 + 0 =3$ for every fixed odd period $m\in \mathbb{Z}_{\geq 3}$ or $1 + 1 + 2 + 0 =4$ for every fixed even period $m\in \mathbb{Z}_{\geq 2}$) in Theorem \ref{3.2} (and hence in Theorem \ref{Binder-Brian}) on the number of distinct $m$-periodic integral points in the whole family of polynomial maps $\varphi_{p-1,c}$ modulo $p\mathcal{O}_{K}$ is not only also independent of $p$ (and hence independent of deg$(\varphi_{p-1,c})$) and $n$, but is also a constant equal to $3$ or $4$ even when degree $p-1\to \infty$ or $n\to \infty$. On the other hand, we may also notice that the count obtained in Theorem \ref{BB} on the number of distinct $m$-periodic integral points of any $\varphi_{p,c}$ modulo $p\mathcal{O}_{K}$ may depend on $p$ (and hence depend on deg$(\varphi_{p,c}))$, however, not on degree $n$ in one of the two possibilities; or the count obtained in Theorem \ref{BB} may neither depend on $p$ nor $n$ in the other possibility. Consequently, the expected total count (namely, $p+0 =p$ for every fixed period $m\in \mathbb{Z}_{\geq 2}$) in Theorem \ref{BB} on the number of distinct $m$-periodic integral points in the whole family of polynomial maps $\varphi_{p,c}$ modulo $p\mathcal{O}_{K}$ may not only depend on $p$, but also may grow to infinity when $p\to \infty$.

Inspired by work of Adam-Fares \cite{Ada} in arithmetic dynamics and by a \say{counting-application} philosophy in arithmetic statistics, we then revisit in a forthcoming paper \cite{BK333} the setting in Section \ref{sec2} and \ref{sec3} where we consider a polynomial map iterating on the space $\mathbb{Z}_{p}\slash p\mathbb{Z}_{p}$; and again with the sole purpose of investigating further the aforementioned relationship. Somewhat interestingly, we prove in \cite{BK333} a counting and asymptotics that's analogous to the one already done here in Section \ref{sec2} and \ref{sec3}. Motivated further by $\mathbb{F}_{p}(t)$-periodic point-counting theorem of Benedetto \ref{main}, we then also in \cite{BK333} revisit the setting in Section \ref{sec2} and \ref{sec3} where we consider a polynomial map iterating on the space $\mathbb{F}_{p}[t]\slash (\pi)$, where $\pi\in \mathbb{F}_{p}[t]$ is any fixed irreducible monic polynomial. In doing so, we also prove in \cite{BK333} a counting and asymptotics that's analogous to the one done in Sect. \ref{sec2} and \ref{sec3}. 

In addition, to the notion of a periodic point and a periodic orbit, we also recall that a point $\alpha\in {\mathbb{P}^N(K)}$ is called a \textit{preperiodic point of $\varphi$}, whenever the iterate $\varphi^{m+n}(\alpha) = \varphi^{m}(\alpha)$ for some integers $m\geq 0$ and $n\geq 1$. In this case, we recall that the smallest integers $m\geq 0$ and $n\geq 1$ such that $\varphi^{m+n}(\alpha) = \varphi^{m}(\alpha)$ happens, are called the \textit{preperiod} and \textit{eventual period of $\alpha$}, respectively. Again, we denote the set of preperiodic points of $\varphi$ by PrePer$(\varphi, {\mathbb{P}^N(K)})$. For any given preperiodic point $\alpha$ of $\varphi$, we then call the set of all iterates of $\varphi$ on $\alpha$, \textit{the preperiodic orbit of $\alpha$}.
Now observe for $m=0$, we have $\varphi^{n}(\alpha) = \alpha $ and so $\alpha$ is a periodic point of period $n$. Thus, the set  Per$(\varphi, {\mathbb{P}^N(K)}) \subseteq$ PrePer$(\varphi, {\mathbb{P}^N(K)})$; however, it need not be PrePer$(\varphi, {\mathbb{P}^N(K)})\subseteq$ Per$(\varphi, {\mathbb{P}^N(K)})$. In their 2014 paper \cite{Doyle}, Doyle-Faber-Krumm give nice examples (which also recover examples in Poonen's paper \cite{Poonen}) of preperiodic points of any quadratic map $\varphi$ (where $\varphi$ is not necessarily the somewhat mostly studied $\varphi_{2,c}$ in arithmetic dynamics) defined over quadratic fields; and so the interested reader may wish to revisit \cite{Poonen, Doyle}.

In the year 1950, Northcott \cite{North} used the theory of height functions to show that not only is the set PrePer$(\varphi, {\mathbb{P}^N(K)})$ always finite, but also for a given morphism $\varphi$ the set PrePer$(\varphi, {\mathbb{P}^N(K)})$ can be computed effectively. Forty-five years later, in the year 1995, Morton and Silverman conjectured that PrePer$(\varphi, \mathbb{P}^N(K))$ can be bounded in terms of degree $d$ of $\varphi$, degree $D$ of $K$, and dimension $N$ of the space ${\mathbb{P}^N(K)}$. This celebrated conjecture is called the \textit{Uniform Boundedness Conjecture}; which we then restate here as the following conjecture:

\begin{conj} \label{silver-morton}[\cite{Morton}]
Fix integers $D \geq 1$, $N \geq 1$, and $d \geq 2$. There exists a constant $C'= C'(D, N, d)$ such that for all number fields $K/{\mathbb{Q}}$ of degree at most $D$, and all morphisms $\varphi: {\mathbb{P}^N}(K) \rightarrow {\mathbb{P}^N}(K)$ of degree $d$ defined over $K$, the total number of preperiodic points of a morphism $\varphi$ is at most $C'$, i.e., \#PrePer$(\varphi, \mathbb{P}^N(K)) \leq C'$.
\end{conj}
\noindent A special case of Conjecture \ref{silver-morton} is when the degree $D$ of a number field $K$ is $D = 1$, dimension $N$ of a space $\mathbb{P}^N(K)$ is $N = 1$, and degree $d$ of a morphism $\varphi$ is $d = 2$. In this case, if $\varphi$ is a polynomial morphism, then it is a quadratic map defined over the field $\mathbb{Q}$. Moreover, in this very special case, in the year 1995, Flynn and Poonen and Schaefer conjectured that a quadratic map has no points $z\in\mathbb{Q}$ with exact period more than 3. This conjecture of Flynn-Poonen-Schaefer \cite{Flynn} (which has been resolved for cases $n = 4$, $5$ in \cite{mor, Flynn} respectively and conditionally for $n=6$ in \cite{Stoll} is, however, still open for all integers $n\geq 7$ and moreover, which also Hutz-Ingram \cite{Ingram} gave strong computational evidence supporting it) is restated here formally as the following conjecture. Note that in this same special case, rational points of exact period $n\in \{1, 2, 3\}$ were first found in the year 1994 by Russo-Walde \cite{Russo} and also found in the year 1995 by Poonen \cite{Poonen} using a different set of techniques. We now restate the anticipated conjecture of Flynn-Poonen-Schaefer as the following conjecture:
 
\begin{conj} \label{conj:2.4.1}[\cite{Flynn}, Conjecture 2]
If $n \geq 4$, then there is no quadratic polynomial $\varphi_{2,c }(z) = z^2 + c\in \mathbb{Q}[z]$ with a rational point of exact period $n$.
\end{conj}
Now by assuming Conjecture \ref{conj:2.4.1} and also establishing interesting results on rational preperiodic points, in the year 1998, Poonen \cite{Poonen} then concluded that the total number of rational preperiodic points of any quadratic polynomial $\varphi_{2, c}(z)=z^2 + c$ is at most nine. We restate here formally Poonen's result as the following corollary:
\begin{cor}\label{cor2}[\cite{Poonen}, Corollary 1]
If Conjecture \ref{conj:2.4.1} holds, then $\#$PrePer$(\varphi_{2,c}, \mathbb{Q}) \leq 9$,  for all quadratic maps $\varphi_{2, c}$ defined by $\varphi_{2, c}(z) = z^2 + c$ for all points $c, z\in\mathbb{Q}$.
\end{cor}

On still the same note of exact periods and pre(periodic) points, the next natural question that one could ask is whether the aforementioned phenomenon on exact periods and pre(periodic) points has been investigated in some other cases, namely, when $D\geq 2$, $N\geq 1$ and $d\geq 2$. In the case $D = d = 2$ and $N = 1$, then again if $\varphi$ is a polynomial map, then $\varphi$ is a quadratic map defined over a quadratic field $K = \mathbb{Q}(\sqrt{D'})$. In this case, in the years 1900, 1998 and 2006, Netto \cite{Netto}, Morton-Silverman \cite{Morton} and Erkama \cite{Erkama} resp., found independently a parametrization of a point $c$ in the field $\mathbb{C}$ of all complex points which guarantees $\varphi_{2,c}$ to have periodic points of period $M=4$. And moreover when $c\in \mathbb{Q}$, Panraksa \cite{par1} showed that one gets \textit{all} orbits of length $M = 4$ defined over $\mathbb{Q}(\sqrt{D'})$. For $M=5$, Flynn-Poonen-Schaefer \cite{Flynn} found a parametrization of a point $c\in \mathbb{C}$ that yields points of period 5; however, these periodic points are not in $K$, but rather in some other extension of $\mathbb{Q}$. In the same case $D = d = 2$ and $N = 1$, Hutz-Ingram \cite{Ingram} and Doyle-Faber-Krumm \cite{Doyle} did not find in their computational investigations points $c\in K$ for which $\varphi_{2,c}$ defined over $K$ has $K$-rational points of exact period $M = 5$. Note that to say that the above authors didn't find points $c\in K$ for which $\varphi_{2,c}$ has $K$-rational points of exact period $M=5$, is not the same as saying that such points do not exist; since it's possible that the techniques which the authors employed in their computational investigations may have been far from enabling them to decide concretely whether such points exist or not. In fact, as of the present article, we do not know whether $\varphi_{2,c}$ has $K$-rational points of exact period $5$ or not, but surprisingly from \cite{Flynn, Stoll, Ingram, Doyle} we know that for $c=-\frac{71}{48}$ and $D'=33$ the map $\varphi_{2,c}$ defined over $K = \mathbb{Q}(\sqrt{33})$ has $K$-rational points of exact period $M = 6$; and mind you, this is the only example of $K$-rational points of exact period $M=6$ that is currently known of in the whole literature of arithmetic dynamics. For $M>6$, in 2013, Hutz-Ingram [\cite{Ingram}, Prop. 2 and 3] gave strong computational evidence which showed that for any absolute discriminant $D'$ at most 4000 and any $c\in K$ with a certain logarithmic height, the map $\varphi_{2,c}$ defined over any $K$ has no $K$-rational points of exact period greater than 6. Moreover, the same authors \cite{Ingram} also showed that the smallest upper bound on the size of PrePer$(\varphi_{2,c}, K)$ is 15. A year later, in 2014, Doyle-Faber-Krumm \cite{Doyle} also gave computational evidence on 250000 pairs $(K, \varphi_{2,c})$ which not only established the same claim [\cite{Doyle}, Thm. 1.2] as that of Hutz-Ingram \cite{Ingram} on the upper bound of the size of PrePer$(\varphi_{2,c}, K)$, but also covered Poonen's claims in \cite{Poonen} on $\varphi_{2,c}$ over $\mathbb{Q}$. Three years later, in 2018, Doyle \cite{Doy} adjusted the computations in his aforementioned cited work with Faber and Krumm; and after from which he made the following conjecture on any quadratic map over any $K = \mathbb{Q}(\sqrt{D'})$:

\begin{conj}\label{do}[\cite{Doy}, Conjecture 1.4]
Let $K\slash \mathbb{Q}$ be a quadratic field and let $f\in K[z]$ be a quadratic polynomial. \newline Then, $\#$PrePer$(f, K)\leq 15$.
\end{conj}

Recall in \cite{BK1,BK3, BK11} we attempted to understand (on the level of the ring of integers $\mathcal{O}_{K}$) the possibility and validity of \ref{per} of Morton-Silverman's Conjecture \ref{silver-morton}. In this article, we again wish to continue with this attempt of hoping to understand (again on the level of the ring of $\mathcal{O}_{K}$) the possibility and validity of \ref{per}. That is, in Section \ref{sec2} and \ref{sec3} we consider polynomial maps of any odd prime power degree $d\geq 3$ over $\mathcal{O}_{K}$ and also consider polynomial maps of any even degree $d\geq 4$ over $\mathcal{O}_{K}$, where $K$ is any number field of degree $n\geq 2$, resp.; all of this again done in the attempt of understanding the possibility and validity of the following version:

\begin{conj} \label{silver-morton 1}($(D,1)$-version of Conjecture \ref{silver-morton})\label{per}
Fix integers $D \geq 1$ and $d \geq 2$. There exists a constant $C'= C'(D, d)$ such that for all number fields $K/{\mathbb{Q}}$ of degree at most $D$, and all morphisms $\varphi: {\mathbb{P}}^1(K) \rightarrow {\mathbb{P}}^1(K)$ of degree $d$ over $K$, the total number of periodic points of a morphism $\varphi$ is at most $C'$, i.e., \#Per$(\varphi, \mathbb{P}^1(K)) \leq C'$.
\end{conj}

\subsection*{History on the Connection Between the Size of Per$(\varphi_{d, c}, K)$ and the Coefficient $c$}

In the year 1994, Walde and Russo not only proved [\cite{Russo}, Corollary 4] that for a quadratic map $\varphi_{2,c}$ defined over $\mathbb{Q}$ with a periodic point, the denominator of a rational point $c$, denoted as den$(c)$, is a square but they also proved that den$(c)$ is even, whenever $\varphi_{2,c}$ admits a rational cycle of length $\ell \geq 3$. Moreover, Walde-Russo also proved [\cite{Russo}, Cor. 6, Thm. 8 and Cor. 7] that the size \#Per$(\varphi_{2, c}, \mathbb{Q})\leq 2$, whenever den$(c)$ is an odd integer. 

Three years later, in the year 1997, Call-Goldstine \cite{Call} proved that the size of PrePer$(\varphi_{2,c},\mathbb{Q})$ can be bounded above in terms of the number of distinct odd primes dividing den$(c)$. We restate formally this result of Call-Goldstine as the following theorem, in which $GCD(a, e)$ refers to the greatest common divisor of $a$, $e \in \mathbb{Z}$:

\begin{thm}\label{2.3.1}[\cite{Call}, Theorem 6.9]
Let $e>0$ be an integer and let $s$ be the number of distinct odd prime factors of e. Define $\varepsilon  = 0$, $1$, $2$, if $4\nmid e$, if $4\mid e$ and $8 \nmid e$, if $8 \mid e$, respectively. Let $c = a/e^2$, where $a\in \mathbb{Z}$ and $GCD(a, e) = 1$. If $c \neq -2$, then the total number of $\mathbb{Q}$-preperiodic points of $\varphi_{2, c}$ is at most $2^{s + 2 + \varepsilon} + 1$. Moreover, a quadratic map $\varphi_{2, -2}$ has exactly six rational preperiodic points.
\end{thm}

Eight years later, after the work of Call-Goldstine, in the year 2005, Benedetto \cite{detto} studied polynomial maps $\varphi$ of arbitrary degree $d\geq 2$ defined over an arbitrary global field $K$, and then established the following result on the relationship between the size of the set PrePre$(\varphi, K)$ and the number of bad primes of $\varphi$ in $K$:

\begin{thm}\label{main} [\cite{detto}, Main Theorem]
Let $K$ be a global field, $\varphi\in K[z]$ be a polynomial of degree $d\geq 2$ and $s$ be the number of bad primes of $\varphi$ in $K$. The number of preperiodic points of $\varphi$ in $\mathbb{P}^N(K)$ is at most $O(\text{s log s})$. 
\end{thm}
 
\noindent Since Benedetto's Theorem \ref{main} applies to any polynomial $\varphi$ of arbitrary degree $d\geq 2$ defined over any number field $K$, it then follows that one can immediately apply Benedetto's Theorem \ref{main} to any polynomial $\varphi$ of arbitrary odd or even degree $d> 2$ defined over any number field $K$ and then obtain the upper bound in Theorem \ref{main}. 

Five years after the work of Benedetto, in the year 2010, Narkiewicz's \cite{Narkie1} proved that any polynomial map $\varphi_{d,c}$ of odd prime-power degree $d = p^{\ell}$ with $\ell\geq 1$ defined over any totally complex extension $K\slash \mathbb{Q}$ of degree $n$ where $K$ does not contain $p$-\text{th} roots of unity, the length of $K$-cycles of a polynomial map $\varphi_{d,c}$ are bounded by a certain constant $B = B(K, p)$ depending only on $K$ and $p$; and moreover if $p$ exceeds $2^n$, then the bound depends only on degree $n$ of $K$. We restate here more formally Narkiewicz's result as the following:

\begin{thm} \label{theorem 3.2.1}[\cite{Narkie1}, Theorem]
Let $K$ be a totally complex extension of $\mathbb{Q}$ of degree $n>1$, denote by $R$ its ring of integers and $D$ be the maximal order of a primitive root of unity contained in $K$. Let $p$ be a prime not dividing $D$, and put $F(X) = X^n + c\in K[X]$ with $n = p^k$ with $k\geq 1$ and $c\neq 0$. Then the lengths of cycles of $F$ in $K$ are bounded by a constant $B = B(K, p)$. If $p>2^n$, then this constant can be taken to be $n2^{n+1}(2^n-1)$.
\end{thm}\noindent Now recall in arithmetic dynamics, and more generally in classical dynamical systems that we can always identify any $K$-orbit of any map, say $\varphi_{p,c}$, with any $K$-cycle of the same map. So then, if we were working under the assumptions in Theorem \ref{theorem 3.2.1}, then by Theorem \ref{theorem 3.2.1} the total number of distinct points in any $K$-orbit is bounded by a constant $B$ depending on only $K$ and $p$; and moreover $B = n2^{n+1}(2^n -1)$ whenever $p>2^n$.

Three years after \cite{Narkie}, in 2015, Hutz \cite{Hutz} developed an algorithm determining effectively all $\mathbb{Q}$-preperiodic points of a morphism defined over a given number field $K$; from which he then made the following conjecture: 

\begin{conj} \label{conjecture 3.2.1}[\cite{Hutz}, Conjecture 1a]
For any integer $n > 2$, there is no even degree $d > 2$ and no point $c \in \mathbb{Q}$ such that the polynomial map $\varphi_{d, c}$ has rational points of exact period $n$.
Moreover, \#PrePer$(\varphi_{d, c}, \mathbb{Q}) \leq 4$. 
\end{conj}

\begin{rem}
If Conjecture \ref{conjecture 3.2.1} held, then this would also mean that for any integer $n\geq 3$ and for any even integer $d\geq 4$, the monic polynomial $\varphi_{d, c}^n(z)-z \neq 0$ for all rational integral points $c, z\in \mathbb{Q}$. This would then also mean that the total number of $n$-periodic integral points of any $\varphi_{d, c}$ of even degree $d\geq 4$ is equal to zero. Moreover, since the monic polynomial $\varphi_{d,c}(x)\in \mathbb{Z}[x]$ has good reduction modulo $p$ and so the reduced polynomial $\varphi_{d,c}(x)$ modulo $p$ also has even degree $d$, it would then also follow that the total number of $n$-periodic integral points of any $\varphi_{d,c}(x)$ modulo $p$ (and hence of any $\varphi_{d,c}$ modulo $p$) is equal to zero. But of course now (as in \cite{BK11}) the issue here is that we unfortunately don't know (as to the author's knowledge) whether Conjecture \ref{conjecture 3.2.1} holds or not. On the note whether any theoretical progress has yet been made on Conjecture \ref{conjecture 3.2.1}, more recently, Panraksa \cite{par2} proved among many other results that the quartic polynomial $\varphi_{4,c}(z)\in\mathbb{Q}[z]$ has rational points of exact period $n = 2$. Moreover, he also proved that $\varphi_{d,c}(z)\in\mathbb{Q}[z]$ has no rational points of exact period $n = 2$ for any $c \in \mathbb{Q}$ with $c \neq -1$ and $d = 6$, $2k$ with $3 \mid 2k-1$. The interested reader may find these mentioned results of Panraksa in his unconditional Thms. 2.1, 2.4 and Thm. 1.7 conditioned on the abc-conjecture in \cite{par2}.
\end{rem}

Twenty-eight years later, after the work of Walde-Russo, in the year 2022, Eliahou-Fares proved [\cite{Shalom2}, Theorem 2.12] that the denominator of a rational point $-c$, denoted as den$(-c)$ is divisible by 16, whenever $\varphi_{2,-c}$ defined by $\varphi_{2, -c}(z) = z^2 - c$ for all $c, z\in \mathbb{Q}$ admits a rational cycle of length $\ell \geq 3$. Moreover, they also proved [\cite{Shalom2}, Proposition 2.8] that the size \#Per$(\varphi_{2, -c}, \mathbb{Q})\leq 2$, whenever den$(-c)$ is an odd integer. Motivated by \cite{Call}, Eliahou-Fares \cite{Shalom2} also proved that the size of Per$(\varphi_{2, -c}, \mathbb{Q})$ can be bounded above by using information on den$(-c)$, namely, information in terms of the number of distinct primes dividing den$(-c)$. Moreover, they in \cite{Shalom1} also showed that the upper bound is four, whenever $c\in \mathbb{Q^*} = \mathbb{Q}\setminus\{0\}$. We restate here their results as:

\begin{cor}\label{sha}[\cite{Shalom2, Shalom1}, Cor. 3.11 and Cor. 4.4, respectively]
Let $c\in \mathbb{Q}$ such that den$(c) = d^2$ with $d\in 4 \mathbb{N}$. Let $s$ be the number of distinct primes dividing $d$. Then, the total number of $\mathbb{Q}$-periodic points of $\varphi_{2, -c}$ is at most $2^s + 2$. Moreover, for $c\in \mathbb{Q^*}$ such that the den$(c)$ is a power of a prime number. Then, $\#$Per$(\varphi_{2, c}, \mathbb{Q}) \leq 4$.
\end{cor}

\noindent Once again, the purpose of this article is to investigate further the aforementioned connection, however, carrying it out in the case of polynomial maps $\varphi_{p^{\ell}, c}$ defined over any number field $K\slash \mathbb{Q}$ of degree $n\geq 2$ for any given prime integer $p\geq 3$, and also in the case of polynomial maps $\varphi_{(p-1)^{\ell}, c}$ defined over any number field $K\slash \mathbb{Q}$ of degree $n\geq 2$ for any given prime integer $p\geq 5$ and for any integer $\ell\geq 1$; and doing all of this from a spirit that is greatly inspired and guided by some of the many striking developments in the area of arithmetic statistics.

\section{The Number of $m$-Periodic Integral Points of any Family of Polynomial Maps $\varphi_{p^{\ell},c}$}\label{sec2}

In this section, we count the number of distinct $m$-periodic integral points of any $\varphi_{p^{\ell},c}$ modulo prime $p\mathcal{O}_{K}$, where $p\geq 3$ is any prime, $\ell \geq 1$ and $m\geq 2$ are any fixed integers. To do so, we let $p\geq 3$ be any prime, $c\in \mathcal{O}_{K}$ be any point, $\ell \geq 1$ and $m\geq 2$ be any fixed integers, and then define $m$-periodic point-counting function 
\begin{equation}\label{N_{c}}
N_{c}^{(m)}(p) := \# \Biggl\{ z\in \mathcal{O}_{K}\slash p\mathcal{O}_{K}   : \begin{aligned} \varphi_{p^{\ell},c}^{m-1}(z) -z \not \equiv 0 \ \text{(mod $p\mathcal{O}_{K}$)} \\ \ \varphi_{p^{\ell},c}^{m}(z) - z \equiv 0 \ \text{(mod $p\mathcal{O}_{K}$)} \end{aligned} \Biggr\}.
\end{equation}

\noindent Setting $\ell =1$, and so the map $\varphi_{p^{\ell}, c} = \varphi_{p,c}$, we then first prove the following theorem and its generalization \ref{2.2}:

\begin{thm} \label{2.1}
Let $K\slash \mathbb{Q}$ be any number field of degree $n \geq 2$ with the ring of integers $\mathcal{O}_{K}$, and in which $3$ is inert. Let  $\varphi_{3, c}$ be a cubic map defined by $\varphi_{3, c}(z) = z^3 + c$ for all $c, z\in\mathcal{O}_{K}$, and  $N_{c}^{(m)}(3)$ be the number defined as in \textnormal{(\ref{N_{c}})}. Then $N_{c}^{(m)}(3)=3$ for every coefficient $c\in 3\mathcal{O}_{K}$; otherwise $N_{c}^{(m)}(3) = 0$ for any coefficient $c \not \in 3\mathcal{O}_{K}$.
\end{thm}

\begin{proof}
Let $f(z)= \varphi_{3,c}^m(z)-z = \varphi_{3,c}(\varphi_{3,c}^{m-1}(z)) - z = (\varphi_{3,c}^{m-1}(z))^3 - z + c$, and so $f(z)= (\varphi_{3,c}^{m-1}(z))^3 - z + c$. Now note that applying the multinomial theorem repeatedly on $(\varphi_{3,c}^{m-1}(z))^3$ after applying the binomial theorem on $(z^3 + c)^3$, we then obtain $(\varphi_{3,c}^{m-1}(z))^3$ is a monic polynomial in $z$ of degree $3^m$ with integral coefficients in multiples of $c$. Hence, we may then write $(\varphi_{3,c}^{m-1}(z))^3 = z^{3^{m}} + h(z)$, where $h(z)$ is a non-constant polynomial in $z$ of deg$(h)<3^m$ with integral coefficients in multiples of $c$; and so $f(z)= z^{3^{m}} + h(z) - z + c$. Now for every coefficient $c\in 3\mathcal{O}_{K}$, reducing $f(z)$ modulo prime ideal $3\mathcal{O}_{K}$, it then follows $f(z)\equiv z^{3^m} - z$ (mod $3\mathcal{O}_{K}$), since also $h(z)\in c\mathcal{O}_{K}[z]$ and so $h(z)\equiv 0$ (mod $3\mathcal{O}_{K}$); and so now $f(z)$ modulo $3\mathcal{O}_{K}$ is a polynomial defined over a finite field $\mathcal{O}_{K}\slash 3\mathcal{O}_{K}$ of order $3^{[K:\mathbb{Q}]} = 3^n$. Now since (from a well-known fact) every subfield of a finite field $\mathcal{O}_{K}\slash 3\mathcal{O}_{K}$ is of order $3^r$ for some positive integer $r\mid n$, we then obtain the inclusion $\mathbb{F}_{3}\hookrightarrow \mathcal{O}_{K}\slash 3\mathcal{O}_{K}$ of fields, where $\mathbb{F}_{3}$ is a field of order $3$; and moreover (from a well-known fact) the cubic monic  polynomial $h(x)=x^3-x$ vanishes at every $z\in \mathbb{F}_{3}\subset \mathcal{O}_{K}\slash 3\mathcal{O}_{K}$ and so $z^3 = z$ for every  $z\in \mathbb{F}_{3}$. But then observe $z^{3^m}= (z^3)^{3^{m-1}} = (z^3)^{3^{m-2}} = z^{3^{m-2}}$ for every $z\in \mathbb{F}_{3}$. Since $m\geq 2$ and so $m-2\geq 0$, then if $m-2 = 0$ and so $z^{3^{m-2}} = z$, it then follows $z^{3^m} = z$ for every $z\in \mathbb{F}_{3}$; and so the reduced polynomial $f(z)\equiv 0$ (mod $3\mathcal{O}_{K}$) for every point $z\in \mathbb{F}_{3}\subset \mathcal{O}_{K}\slash 3\mathcal{O}_{K}$. Otherwise, if $m-2 > 0$, then since $m$ is fixed, we may continue performing the above procedure of decreasing the exponent $m-2$ of $z^{3^{m-2}}=z^{3^m}$ for every $z\in \mathbb{F}_{3}$ until $m-2$ is equal to zero; and from which we then again obtain that $f(z)\equiv 0$ (mod $3\mathcal{O}_{K}$) for every point $z\in \mathbb{F}_{3}\subset \mathcal{O}_{K}\slash 3\mathcal{O}_{K}$. But now we then conclude $N_{c}^{(m)}(3) = 3$.

Finally, we now show $N_{c}^{(m)}(3) = 0$ for every coefficient $c\not \equiv 0$ (mod $3\mathcal{O}_{K}$) and for every fixed $m\in \mathbb{Z}_{\geq 2}$. To see this, let's for the sake of a contradiction, suppose $f(z)\equiv 0$ (mod $3\mathcal{O}_{K}$) for some $z\in \mathcal{O}_{K}\slash 3\mathcal{O}_{K}$ and for every coefficient $c\not \equiv 0$ (mod $3\mathcal{O}_{K}$) and for every fixed $m\in \mathbb{Z}_{\geq 2}$. Now recall from earlier $f(z)= z^{3^{m}} + h(z) - z + c$ where $h(z)\in c\mathcal{O}_{K}[z]$, it then also follows $z^{3^{m}} + h(z) - z + c \equiv 0$ (mod $3\mathcal{O}_{K}$) for some $z\in \mathcal{O}_{K}\slash 3\mathcal{O}_{K}$ and for every $c\not \equiv 0$ (mod $3\mathcal{O}_{K}$) and every fixed $m$. So now, recall from earlier $z^{3^m} = z$ for every $z\in \mathbb{F}_{3}$ and fixed $m$, we may then rewrite $(z^{3^{m}} - z) + (h(z) + c) \equiv 0$ (mod $3\mathcal{O}_{K}$) for some $z\in \mathcal{O}_{K}\slash3\mathcal{O}_{K}$ and every $c\not \equiv 0$ (mod $3\mathcal{O}_{K}$) to obtain $h(z) + c \equiv 0$ (mod $3\mathcal{O}_{K}$) for some $z\in \mathbb{F}_{3}\subset \mathcal{O}_{K}\slash3\mathcal{O}_{K}$ and every $c\not \equiv 0$ (mod $3\mathcal{O}_{K}$). Now looking at the multinomial expansion of $(\varphi_{3,c}^{m-1}(z))^3$, it then follows $h(z)\equiv \sum_{i=1}^{m-1}c^{3^{m-i}}$ (mod $3\mathcal{O}_{K}$); and so $h(z) + c \equiv \sum_{i=1}^{m-1}c^{3^{m-i}} + c $ (mod $3\mathcal{O}_{K}$) and so $\sum_{i=1}^{m-1}c^{3^{m-i}} + c \equiv 0$ (mod $3\mathcal{O}_{K}$). But now we note that the congruence $\sum_{i=1}^{m-1}c^{3^{m-i}} + c \equiv 0$ (mod $3\mathcal{O}_{K}$) can also happen if $\sum_{i=1}^{m-1}c^{3^{m-i}}  \equiv 0$ (mod $3\mathcal{O}_{K}$) and also $ c \equiv 0$ (mod $3\mathcal{O}_{K}$); and so yielding a contradiction. Otherwise, suppose $f(\alpha)\equiv 0$ (mod $3\mathcal{O}_{K}$) and so $\alpha^{3^{m}} + h(\alpha) - \alpha + c \equiv 0$ (mod $3\mathcal{O}_{K}$) for some point $\alpha \in \mathcal{O}_{K}\slash 3\mathcal{O}_{K} \setminus \mathbb{F}_{3}$ and for every $c\not \equiv 0$ (mod $3\mathcal{O}_{K}$) and fixed $m$. So then, since $z^{3^{m}}=z$ for every $z\in \mathbb{F}_{3^{m}}\subset \mathcal{O}_{K}\slash 3\mathcal{O}_{K}$ and for every fixed $m\mid n$, then if a root $\alpha \in \mathbb{F}_{3^{m}}\subset \mathcal{O}_{K}\slash 3\mathcal{O}_{K}\setminus \mathbb{F}_{3}$ and so $\alpha^{3^{m}}=\alpha$, it then follows $h(\alpha)  + c \equiv 0$ (mod $3\mathcal{O}_{K}$); and since $h(\alpha)\equiv \sum_{i=1}^{m-1}c^{3^{m-i}}$ (mod $3\mathcal{O}_{K}$), then $\sum_{i=1}^{m-1}c^{3^{m-i}} + c \equiv 0$ (mod $3\mathcal{O}_{K}$); and so also yielding a contradiction. Otherwise, if $\alpha \not \in \mathbb{F}_{3^{m}}$ and since $h(\alpha)\equiv \sum_{i=1}^{m-1}c^{3^{m-i}}$ (mod $3\mathcal{O}_{K}$), then $(\alpha^{3^{m}} - \alpha )+(\sum_{i=1}^{m-1}c^{3^{m-i}}  + c) \equiv 0$ (mod $3\mathcal{O}_{K}$). But then we note $(\alpha^{3^{m}} - \alpha) +(\sum_{i=1}^{m-1}c^{3^{m-i}}  + c) \equiv 0$ (mod $3\mathcal{O}_{K}$) can also occur if $\alpha^{3^{m}} - \alpha \equiv 0$ (mod $3\mathcal{O}_{K}$) and also $\sum_{i=1}^{m-1}c^{3^{m-i}}  + c \equiv 0$ (mod $3\mathcal{O}_{K}$); and so a contradiction. This overall means that $f(x)= \varphi_{3,c}^m(x)-x$ has no roots in $ \mathcal{O}_{K}\slash 3\mathcal{O}_{K}$ for every coefficient $c\not \in 3\mathcal{O}_{K}$ and for every fixed $m\in \mathbb{Z}_{\geq 2}$; and so we conclude $N_{c}^{(m)}(3) = 0$. This completes the whole proof, as required.
\end{proof} 
We now wish to generalize Theorem \ref{2.1} to any $\varphi_{p,c}$ for any prime $p\geq 3$. More precisely, we prove that the number of distinct $m$-periodic integral points of any polynomial map $\varphi_{p, c}$ modulo $p\mathcal{O}_{K}$ is either $ p$ or zero:

\begin{thm} \label{2.2}
Let $K\slash \mathbb{Q}$ be any number field of degree $n\geq 2$ with the ring of integers $\mathcal{O}_{K}$, and in which any fixed prime $p\geq 3$ is inert. Let $\varphi_{p, c}$ be defined by $\varphi_{p, c}(z) = z^p + c$ for all $c, z\in\mathcal{O}_{K}$, and $N_{c}^{(m)}(p)$ be defined as in \textnormal{(\ref{N_{c}})}. Then $N_{c}^{(m)}(p)=p$ for every coefficient $c\in p\mathcal{O}_{K}$; otherwise $N_{c}^{(m)}(p) = 0$ for every coefficient $c \not \in p\mathcal{O}_{K}$. 
\end{thm}
\begin{proof}
By applying a similar argument as in the Proof of Theorem \ref{2.1}, we then obtain the count as desired. That is, let $f(z)= \varphi_{p,c}^m(z)-z = \varphi_{p,c}(\varphi_{p,c}^{m-1}(z)) - z = (\varphi_{p,c}^{m-1}(z))^p - z + c$, and so $f(z)= (\varphi_{p,c}^{m-1}(z))^p - z + c$. So now, we note that applying the multinomial theorem repeatedly on $(\varphi_{p,c}^{m-1}(z))^p$ after applying the binomial theorem on $(z^p + c)^p$, it then follows $(\varphi_{p,c}^{m-1}(z))^p$ is a monic polynomial in $z$ of degree $p^m$ with integral coefficients in multiples of $c$. Thus, we may then write $(\varphi_{p,c}^{m-1}(z))^p = z^{p^{m}} + h(z)$, where $h(z)$ is a non-constant polynomial in $z$ of degree deg$(h)<p^m$ with integral coefficients in multiples of $c$; and so $f(z)= z^{p^{m}} + h(z) - z + c$. Now for every coefficient $c\in p\mathcal{O}_{K}$, reducing $f(z)$ modulo prime ideal $p\mathcal{O}_{K}$, it then follows $f(z)\equiv z^{p^m} - z$ (mod $p\mathcal{O}_{K}$), since also $h(z)\in c\mathcal{O}_{K}[z]$ and so $h(z)\equiv 0$ (mod $p\mathcal{O}_{K}$); and so now $f(z)$ modulo $p\mathcal{O}_{K}$ is a polynomial defined over a finite field $\mathcal{O}_{K}\slash p\mathcal{O}_{K}$ of order $p^{[K:\mathbb{Q}]} = p^n$. So now, since (as a fact) every subfield of a finite field $\mathcal{O}_{K}\slash p\mathcal{O}_{K}$ is of order $p^r$ for some positive integer $r\mid n$, we then obtain the inclusion $\mathbb{F}_{p}\hookrightarrow \mathcal{O}_{K}\slash p\mathcal{O}_{K}$ of fields, where $\mathbb{F}_{p}$ is a field of order $p$; and moreover recall (as a fact) $z^p = z$ for every element $z\in \mathbb{F}_{p}\subset \mathcal{O}_{K}\slash p\mathcal{O}_{K}$. But then notice $z^{p^m}= (z^p)^{p^{m-1}} = (z^p)^{p^{m-2}} = z^{p^{m-2}}$ for every element $z\in \mathbb{F}_{p}$. So now, since $m\geq 2$ and so $m-2\geq 0$, then if $m-2 = 0$ and so $z^{p^{m-2}} = z$, it then follows $z^{p^m} = z$  for every element $z\in \mathbb{F}_{p}$; and so the reduced polynomial $f(z)\equiv 0$ (mod $p\mathcal{O}_{K}$) for every point $z\in \mathbb{F}_{p}\subset \mathcal{O}_{K}\slash p\mathcal{O}_{K}$. Otherwise, if $m-2 > 0$, then since $m$ is fixed, we may then continue performing the above procedure of decreasing the exponent $m-2$ of $z^{p^{m-2}}=z^{p^m}$ for every $z\in \mathbb{F}_{p}$ until $m-2$ is equal to zero; from which we then again obtain that the reduced polynomial $f(z)\equiv 0$ (mod $p\mathcal{O}_{K}$) for every point $z\in \mathbb{F}_{p}\subset \mathcal{O}_{K}\slash p\mathcal{O}_{K}$. But now we then conclude  $N_{c}^{(m)}(p) = p$.

Finally, we now show $N_{c}^{(m)}(p) = 0$ for every coefficient $c\not \equiv 0$ (mod $p\mathcal{O}_{K}$) and for every fixed $m\in \mathbb{Z}_{\geq 2}$. Let's for the sake of a contradiction, suppose $f(z)\equiv 0$ (mod $p\mathcal{O}_{K}$) for some point $z\in \mathcal{O}_{K}\slash p\mathcal{O}_{K}$ and for every coefficient $c\not \equiv 0$ (mod $p\mathcal{O}_{K}$) and for every fixed $m\in \mathbb{Z}_{\geq 2}$. So then, since from earlier $f(z)= z^{p^{m}} + h(z) - z + c$ where $h(z)\in c\mathcal{O}_{K}[z]$, it then also follows $z^{p^{m}} + h(z) - z + c \equiv 0$ (mod $p\mathcal{O}_{K}$) for some $z\in \mathcal{O}_{K}\slash p\mathcal{O}_{K}$ and for every $c\not \equiv 0$ (mod $p\mathcal{O}_{K}$) and every fixed $m$. Now since from earlier $z^{p^m} = z$ for every $z\in \mathbb{F}_{p}$ and fixed $m$, we may rewrite $(z^{p^{m}} - z)+ h(z) + c \equiv 0$ (mod $p\mathcal{O}_{K}$) for some $z\in \mathcal{O}_{K}\slash p\mathcal{O}_{K}$ and for every $c\not \equiv 0$ (mod $p\mathcal{O}_{K}$) to then obtain $h(z) + c \equiv 0$ (mod $p\mathcal{O}_{K}$) for some $z\in \mathbb{F}_{p}\subset \mathcal{O}_{K}\slash p\mathcal{O}_{K}$ and for every $c\not \equiv 0$ (mod $p\mathcal{O}_{K}$). Moreover, looking at the multinomial expansion of $(\varphi_{p,c}^{m-1}(z))^p$, it then follows $h(z)\equiv \sum_{i=1}^{m-1}c^{p^{m-i}}$ (mod $p\mathcal{O}_{K}$); and so $h(z) + c \equiv \sum_{i=1}^{m-1}c^{p^{m-i}} + c $ (mod $p\mathcal{O}_{K}$) and so $\sum_{i=1}^{m-1}c^{p^{m-i}} + c  \equiv 0$ (mod $p\mathcal{O}_{K}$). But now we note $\sum_{i=1}^{m-1}c^{p^{m-i}} + c \equiv 0$ (mod $p\mathcal{O}_{K}$) can also occur if $\sum_{i=1}^{m-1}c^{p^{m-i}}  \equiv 0$ (mod $p\mathcal{O}_{K}$) and also $c \equiv 0$ (mod $p\mathcal{O}_{K}$); and so follows a contradiction. Otherwise, suppose $f(\alpha)\equiv 0$ (mod $p\mathcal{O}_{K}$) and so $\alpha^{p^{m}} + h(\alpha) - \alpha + c \equiv 0$ (mod $p\mathcal{O}_{K}$) for some $\alpha \in \mathcal{O}_{K}\slash p\mathcal{O}_{K} \setminus \mathbb{F}_{p}$ and for every $c\not \equiv 0$ (mod $p\mathcal{O}_{K}$) and every fixed $m$. So then, since $z^{p^{m}}=z$ for every $z\in \mathbb{F}_{p^{m}}\subset \mathcal{O}_{K}\slash p\mathcal{O}_{K}$ and for every fixed $m\mid n$, then if a root $\alpha \in \mathbb{F}_{p^{m}}\subset \mathcal{O}_{K}\slash p\mathcal{O}_{K}\setminus \mathbb{F}_{p}$ and so $\alpha^{p^{m}}=\alpha$, it then follows $h(\alpha)  + c \equiv 0$ (mod $p\mathcal{O}_{K}$); and moreover since $h(\alpha)\equiv \sum_{i=1}^{m-1}c^{p^{m-i}}$ (mod $p\mathcal{O}_{K}$), it then follows $\sum_{i=1}^{m-1}c^{p^{m-i}} + c \equiv 0$ (mod $p\mathcal{O}_{K}$); and so also follows a contradiction. Otherwise, if $\alpha \not \in \mathbb{F}_{p^{m}}$ and since $h(\alpha)\equiv \sum_{i=1}^{m-1}c^{p^{m-i}}$ (mod $p\mathcal{O}_{K}$), it then follows $(\alpha^{p^{m}} - \alpha) +\sum_{i=1}^{m-1}c^{p^{m-i}}  + c \equiv 0$ (mod $p\mathcal{O}_{K}$). But then we note $(\alpha^{p^{m}} - \alpha) +(\sum_{i=1}^{m-1}c^{p^{m-i}}  + c) \equiv 0$ (mod $p\mathcal{O}_{K}$) can also occur if $\alpha^{p^{m}} - \alpha \equiv 0$ (mod $p\mathcal{O}_{K}$) and also $\sum_{i=1}^{m-1}c^{p^{m-i}}  + c \equiv 0$ (mod $p\mathcal{O}_{K}$); and so a contradiction. This overall means $f(x)=\varphi_{p,c}^m(x)-x$ has no roots in $ \mathcal{O}_{K}\slash p\mathcal{O}_{K}$ for every $c\not \in p\mathcal{O}_{K}$ and fixed $m$; and so $N_{c}^{(m)}(p) = 0$. This completes the whole proof, as desired.
\end{proof}

Finally, we wish to generalize Theorem \ref{2.2} further to any $\varphi_{p^{\ell}, c}$ for any prime $p\geq 3$ and any $\ell \in \mathbb{Z}_{\geq 1}$. That is, we prove the number of distinct $m$-periodic integral points of any $\varphi_{p^{\ell}, c}$ modulo $p\mathcal{O}_{K}$ is also $p$ or zero:

\begin{thm} \label{2.3}
Let $K\slash \mathbb{Q}$ be any number field of degree $ n \geq 2$ with ring $\mathcal{O}_{K}$, and in which any fixed prime $p\geq 3$ is inert. Let $\ell \geq 1$ be any fixed integer, and  $\varphi_{p^{\ell}, c}$ be a map defined by $\varphi_{p^{\ell}, c}(z) = z^{p^{\ell}} + c$ for all $c, z\in\mathcal{O}_{K}$. Let $N_{c}^{(m)}(p)$ be defined as in \textnormal{(\ref{N_{c}})}. Then $N_{c}^{(m)}(p) = p$  for every $c\in p\mathcal{O}_{K}$; otherwise $N_{c}^{(m)}(p) = 0$ for every $c \not \in p\mathcal{O}_{K}$. 
\end{thm}

\begin{proof}
By applying a similar argument as in the Proof of Theorem \ref{2.2}, we then obtain the count as desired. That is, as before let $f(z)= \varphi_{p^{\ell},c}^m(z)-z = (\varphi_{p^{\ell},c}^{m-1}(z))^{p^{\ell}} - z + c$, and so $f(z)= (\varphi_{p^{\ell},c}^{m-1}(z))^{p^{\ell}} - z + c$. Now applying the multinomial theorem repeatedly on the term $(\varphi_{p^{\ell},c}^{m-1}(z))^{p^{\ell}}$ after applying the binomial theorem on the term $(z^{p^{\ell}} + c)^{p^{\ell}}$, it then follows $(\varphi_{p^{\ell},c}^{m-1}(z))^{p^{\ell}}$ is a monic polynomial in $z$ of degree $p^{m\ell}$ with integral coefficients in multiples of $c$. Hence, we may then write $(\varphi_{p^{\ell},c}^{m-1}(z))^{p^{\ell}} = z^{p^{m\ell}} + h(z)$, where $h(z)$ is a non-constant polynomial in $z$ of deg$(h)<p^{m\ell}$ with integral coefficients in multiples of $c$; and so $f(z)= z^{p^{m\ell}} + h(z) - z + c$. Now for every coefficient $c\in p\mathcal{O}_{K}$, reducing $f(z)$ modulo prime ideal $p\mathcal{O}_{K}$, it then follows $f(z)\equiv z^{p^{m\ell}} - z$ (mod $p\mathcal{O}_{K}$), since also $h(z)\in c\mathcal{O}_{K}[z]$ and so $h(z)\equiv 0$ (mod $p\mathcal{O}_{K}$); and so now the reduced polynomial $f(z)$ modulo $p\mathcal{O}_{K}$ is a polynomial defined over a finite field $\mathcal{O}_{K}\slash p\mathcal{O}_{K}$. So now, recall the inclusion $\mathbb{F}_{p}\hookrightarrow \mathcal{O}_{K}\slash p\mathcal{O}_{K}$ of fields and moreover since (as a fact) $z^p = z$ for every element $z\in \mathbb{F}_{p}$, we then note that $z^{p^{m\ell}}= (z^p)^{p^{m\ell-1}} = (z^p)^{p^{m\ell-2}} = z^{p^{m\ell-2}}$ for every element $z\in \mathbb{F}_{p}\subset \mathcal{O}_{K}\slash p\mathcal{O}_{K}$. Now since $m\ell\geq 2$ and so $m\ell-2\geq 0$, then if $m\ell-2 = 0$ and so $z^{p^{m\ell-2}} = z$, then this yields $z^{p^{m\ell}} = z$  for every element $z\in \mathbb{F}_{p}$; and so $f(z)\equiv 0$ (mod $p\mathcal{O}_{K}$) for every point $z\in \mathbb{F}_{p}\subset \mathcal{O}_{K}\slash p\mathcal{O}_{K}$. Otherwise, if $m\ell-2 > 0$, then since $m\ell$ is fixed, we may continue performing the above procedure of decreasing the exponent $m\ell-2$ of $z^{p^{m\ell-2}} = z^{p^{m\ell}}$ for every $z\in \mathbb{F}_{p}$ until $m\ell-2$ is equal to zero; from which we then again obtain $f(z)\equiv 0$ (mod $p\mathcal{O}_{K}$) for every $z\in \mathbb{F}_{p}\subset \mathcal{O}_{K}\slash p\mathcal{O}_{K}$. But now, as before we then conclude $N_{c}^{(m)}(p) = p$.

Finally, we now show $N_{c}^{(m)}(p) = 0$ for every coefficient $c \not \in p\mathcal{O}_{K}$ and for every fixed $m\geq 2$ and $\ell\geq 1$. As before, we may for the sake of a contradiction, suppose that $f(z)\equiv 0$ (mod $p\mathcal{O}_{K}$) for some point $z\in \mathcal{O}_{K}\slash p\mathcal{O}_{K}$ and for every $c\not \equiv 0$ (mod $p\mathcal{O}_{K}$) and for fixed $m$ and $\ell$. Now since from earlier $f(z)= z^{p^{m\ell}} + h(z) - z + c$ where $h(z)\in c\mathcal{O}_{K}[z]$, it then also follows $z^{p^{m\ell}} + h(z) - z + c \equiv 0$ (mod $p\mathcal{O}_{K}$) for some $z\in \mathcal{O}_{K}\slash p\mathcal{O}_{K}$ and for every $c\not \equiv 0$ (mod $p\mathcal{O}_{K}$) and every fixed $m$ and $\ell$. Moreover, recall from earlier that $z^{p^{m\ell}} = z$ for every $z\in \mathbb{F}_{p}$ and for every fixed $m$ and $\ell$, we may then rewrite $z^{p^{m\ell}} - z + h(z) + c \equiv 0$ (mod $p\mathcal{O}_{K}$) for some $z\in \mathcal{O}_{K}\slash p\mathcal{O}_{K}$ and for every $c\not \equiv 0$ (mod $p\mathcal{O}_{K}$) to obtain $h(z) + c \equiv 0$ (mod $p\mathcal{O}_{K}$) for some $z\in \mathbb{F}_{p}\subset \mathcal{O}_{K}\slash p\mathcal{O}_{K}$ and every $c\not \equiv 0$ (mod $p\mathcal{O}_{K}$). So now, looking at the multinomial expansion of $(\varphi_{p^{\ell},c}^{m-1}(z))^{p^{\ell}}$, it then follows $h(z)\equiv \sum_{i=1}^{m-1}c^{p^{m\ell-i}}$ (mod $p\mathcal{O}_{K}$); and so $h(z) + c \equiv \sum_{i=1}^{m-1}c^{p^{m\ell-i}} + c $ (mod $p\mathcal{O}_{K}$) and so  $\sum_{i=1}^{m-1}c^{p^{m\ell-i}} + c  \equiv 0$ (mod $p\mathcal{O}_{K}$). But now, as before we note that the congruence $\sum_{i=1}^{m-1}c^{p^{m\ell-i}} + c \equiv 0$ (mod $p\mathcal{O}_{K}$) can also happen if $\sum_{i=1}^{m-1}c^{p^{m\ell-i}}  \equiv 0$ (mod $p\mathcal{O}_{K}$) and also $c \equiv 0$ (mod $p\mathcal{O}_{K}$); and from which we then obtain a contradiction. Otherwise, suppose $f(\alpha)\equiv 0$ (mod $p\mathcal{O}_{K}$) and so $\alpha^{p^{m\ell}} + h(\alpha) - \alpha + c \equiv 0$ (mod $p\mathcal{O}_{K}$) for some $\alpha \in \mathcal{O}_{K}\slash p\mathcal{O}_{K} \setminus \mathbb{F}_{p}$ and for every $c\not \equiv 0$ (mod $p\mathcal{O}_{K}$) and fixed $m$ and $\ell$. So then, since $z^{p^{m\ell}}=z$ for every $z\in \mathbb{F}_{p^{m\ell}}\subset \mathcal{O}_{K}\slash p\mathcal{O}_{K}$ and for every fixed $m\ell\mid n$, then if $\alpha \in \mathbb{F}_{p^{m\ell}}\subset \mathcal{O}_{K}\slash p\mathcal{O}_{K}\setminus \mathbb{F}_{p}$ and so $\alpha^{p^{m\ell}}=\alpha$, it then follows $h(\alpha)  + c \equiv 0$ (mod $p\mathcal{O}_{K}$); and since also $h(\alpha)=\sum_{i=1}^{m-1}c^{p^{m\ell-i}}$ (mod $p\mathcal{O}_{K}$), it then also follows $\sum_{i=1}^{m-1}c^{p^{m\ell-i}} + c \equiv 0$ (mod $p\mathcal{O}_{K}$); and also follows a contradiction. Otherwise, if $\alpha \not \in \mathbb{F}_{p^{m\ell}}$ and since also $h(\alpha)\equiv \sum_{i=1}^{m-1}c^{p^{m\ell-i}}$ (mod $p\mathcal{O}_{K}$), it then follows $\alpha^{p^{m}} - \alpha +\sum_{i=1}^{m-1}c^{p^{m\ell-i}}  + c \equiv 0$ (mod $p\mathcal{O}_{K}$). But now we note $(\alpha^{p^{m\ell}} - \alpha) +(\sum_{i=1}^{m-1}c^{p^{m\ell-i}}  + c) \equiv 0$ (mod $p\mathcal{O}_{K}$) can also occur if $\alpha^{p^{m}} - \alpha \equiv 0$ (mod $p\mathcal{O}_{K}$) and also $\sum_{i=1}^{m-1}c^{p^{m\ell-i}}  + c \equiv 0$ (mod $p\mathcal{O}_{K}$); and so follows also a contradiction. This then overall means $f(x)=\varphi_{p^{\ell},c}^m(x)-x$ has no roots in $ \mathcal{O}_{K}\slash p\mathcal{O}_{K}$ for any coefficient  $c\not \in p\mathcal{O}_{K}$ and any fixed $m$ and $\ell$; and so we conclude $N_{c}^{(m)}(p) = 0$. This then completes the whole proof, as desired.
\end{proof}

Restricting on the subring $\mathbb{Z}\subset \mathcal{O}_{K}$ of integers, we then also obtain the following consequence of Theorem \ref{2.3} on the number of $m$-periodic integral points of every $\varphi_{p^{\ell},c}$ modulo $p$ for any prime $p\geq 3$ and any $\ell \in \mathbb{Z}_{\geq 1}$: 

\begin{cor} \label{cor2.4}
Let $p\geq 3$ be any fixed prime integer, and $\ell \geq 1$ be any fixed integer. Let $\varphi_{p^{\ell}, c}$ be a polynomial map defined by $\varphi_{p^{\ell}, c}(z) = z^{p^{\ell}} + c$ for all $c, z\in\mathbb{Z}$, and $N_{c}^{(m)}(p)$ be defined as in \textnormal{(\ref{N_{c}})} with $\mathcal{O}_{K} / p\mathcal{O}_{K}$ replaced with $\mathbb{Z}\slash p\mathbb{Z}$. Then $N_{c}^{(m)}(p) = p$  for every coefficient $c=pt$; otherwise the number $N_{c}^{(m)}(p) = 0$ for every point $c\neq pt$.
\end{cor}

\begin{proof}
By applying a similar argument as in the Proof of Theorem \ref{2.3}, we then obtain the count as desired. That is, let $f(z)= \varphi_{p^{\ell},c}^m(z)-z = (\varphi_{p^{\ell},c}^{m-1}(z))^{p^{\ell}} - z + c$, and so $f(z)= (\varphi_{p^{\ell},c}^{m-1}(z))^{p^{\ell}} - z + c$. So now, as before applying the multinomial theorem repeatedly on $(\varphi_{p^{\ell},c}^{m-1}(z))^{p^{\ell}}$ after applying the binomial theorem on $(z^{p^{\ell}} + c)^{p^{\ell}}$, it then follows $(\varphi_{p^{\ell},c}^{m-1}(z))^{p^{\ell}}$ is a monic polynomial in $z$ of degree $p^{m\ell}$ with integer coefficients in multiples of $c$. Thus, we may then write $(\varphi_{p^{\ell},c}^{m-1}(z))^{p^{\ell}} = z^{p^{m\ell}} + h(z)$, where $h(z)$ is a non-constant polynomial in $z$ of degree $<p^{m\ell}$ with integer coefficients in multiples of $c$; and so $f(z)= z^{p^{m\ell}} + h(z) - z + c$. So now, for every coefficient $c=pt$, reducing $f(z)$ modulo $p$, it then follows $f(z)\equiv z^{p^{m\ell}} - z$ (mod $p$), since also $h(z)\in c\mathbb{Z}[z]$ and so $h(z)\equiv 0$ (mod $p$); and so now the reduced polynomial $f(z)$ modulo $p$ is a polynomial defined over a finite field $\mathbb{Z}\slash p\mathbb{Z}$ of $p$ distinct elements. Now recall by Fermat's Little Theorem (FLT) that $z^p\equiv z$ (mod $p$) for any $z\in \mathbb{Z}$ (equivalently $z^p = z$ for every element $z\in \mathbb{Z}\slash p\mathbb{Z}$), it then follows $z^{p^{m\ell}}= (z^p)^{p^{m\ell-1}} = (z^p)^{p^{m\ell-2}} = z^{p^{m\ell-2}}$ for every element $z\in \mathbb{Z}\slash p\mathbb{Z}$. So now, since $m\ell\geq 2$ and so $m\ell-2\geq 0$, then if $m\ell-2 = 0$ and so $z^{p^{m\ell-2}} = z$, it then follows $z^{p^{m\ell}} = z$ for every element $z\in \mathbb{Z}\slash p\mathbb{Z}$; and so the reduced polynomial $f(z)\equiv 0$ (mod $p$) for every point $z\in \mathbb{Z}\slash p\mathbb{Z}$. Otherwise, if $m\ell-2 > 0$, then since $m\ell$ is fixed, we may continue performing the above procedure of decreasing the exponent $m\ell-2$ of $z^{p^{m\ell-2}}=z^{p^{m\ell}}$ for every $z\in \mathbb{Z}\slash p\mathbb{Z}$ until $m\ell-2$ is equal to zero; and from which we then again obtain $f(z)\equiv 0$ (mod $p$) for every point $z\in \mathbb{Z}\slash p\mathbb{Z}$. Hence, we now conclude $N_{c}^{(m)}(p) = p$. 

Finally, we now show $N_{c}^{(m)}(p) = 0$ for every coefficient $c \neq pt$ and for every fixed $m\geq 2$ and $\ell\geq 1$. As before, let's for the sake of a contradiction, suppose $f(z)\equiv 0$ (mod $p$) for some $z\in \mathbb{Z}\slash p\mathbb{Z}$ and for every $c\not \equiv 0$ (mod $p$) and fixed $m\geq 2$ and $\ell\geq 1$. So then, since from earlier $f(z)= z^{p^{m\ell}} + h(z) - z + c$ where $h(z)\in c\mathbb{Z}[z]$, it then also follows $z^{p^{m\ell}} + h(z) - z + c \equiv 0$ (mod $p$) for some point $z\in \mathbb{Z}\slash p\mathbb{Z}$ and for every $c\not \equiv 0$ (mod $p$) and every fixed $m$ and $\ell$. Moreover, recalling from earlier that $z^{p^{m\ell}} = z$ for every $z\in \mathbb{Z}\slash p\mathbb{Z}$ and for every fixed $m\geq 2$ and $\ell\geq 1$, we may then rewrite $(z^{p^{m\ell}} -z) + (h(z) + c) \equiv 0$ (mod $p$) for some $z\in \mathbb{Z}\slash p\mathbb{Z}$ and for every $c\not \equiv 0$ (mod $p$) to obtain $h(z) + c \equiv 0$ (mod $p$) for some $z\in \mathbb{Z}\slash p\mathbb{Z}$ and for every $c\not \equiv 0$ (mod $p$). Now looking at the multinomial expansion of $(\varphi_{p^{\ell},c}^{m-1}(z))^{p^{\ell}}$, it then follows $h(z)\equiv \sum_{i=1}^{m-1}c^{p^{m\ell-i}}$ (mod $p$); 
and so $h(z) + c \equiv \sum_{i=1}^{m-1}c^{p^{m\ell-i}} + c $ (mod $p$) and so $\sum_{i=1}^{m-1}c^{p^{m\ell-i}} + c  \equiv 0$ (mod $p$). But now we note that the congruence $\sum_{i=1}^{m-1}c^{p^{m\ell-i}} + c \equiv 0$ (mod $p$) can also happen whenever $\sum_{i=1}^{m-1}c^{p^{m\ell-i}}  \equiv 0$ (mod $p$) and also $ c \equiv 0$ (mod $p$); from which we then obtain a contradiction. This then means $f(x)=\varphi_{p^{\ell},c}^m(x)-x$ has no roots in $\mathbb{Z}\slash p\mathbb{Z}$ for every coefficient $c\neq pt$ and for every fixed $m$ and $\ell$; and hence we conclude $N_{c}^{(m)}(p) = 0$. This then completes the whole proof, as desired.
\end{proof}

\begin{rem}\label{re2.3}
With now Theorem \ref{2.3}, we may then associate to each distinct $m$-periodic point of $\varphi_{p^{\ell},c}$ an $m$-periodic orbit. In doing so, we then obtain a dynamical translation of Theorem \ref{2.3}, namely, the claim that the number of distinct $m$-periodic integral orbits that any $\varphi_{p^{\ell},c}$ has when iterated on the space $\mathcal{O}_{K} / p\mathcal{O}_{K}$ is either equal to $p$ or zero. As mentioned in Intro.\ref{sec1} that the count in Theorem \ref{2.3} on the number of distinct $m$-periodic integral points of any $\varphi_{p^{\ell},c}$ modulo $p\mathcal{O}_{K}$ may depend on $p$ (and hence depend on deg$(\varphi_{p^{\ell},c}))$, however, not on $n=[K:\mathbb{Q}]$ in one of the two possibilities; or the count in Theorem \ref{2.3} may neither depend on $p$ nor on $n$ in the other possibility. Consequently, the expected total count (namely, $p+0 =p$ for every fixed period $m\in \mathbb{Z}_{\geq 2}$) in Theorem \ref{2.3} on the number of distinct $m$-periodic integral points in the whole family of polynomial maps $\varphi_{p^{\ell},c}$ modulo $p\mathcal{O}_{K}$ may not only depend on $p$ and never on the degree $n$, but also the expected total count may grow to infinity when $p^{\ell}\to \infty$; a somewhat interesting phenomenon differing significantly from the phenomenon in Remark \ref{re3.5}, however, coinciding somewhat surprising with the phenomenon observed in [\cite{BK3,BK11}, Remark 3.3]. 
\end{rem}

\begin{rem}
Recall in \cite{BK3} we proved that fixed point-counting function $N_{c}(p) = p$ (for every $\ell \in \{1,p\}$) or $0$ for every fixed inert prime $p\geq 3$ and for every coefficient $c\in \mathcal{O}_{K}$ divisible or indivisible by $p$. Moreover, in Theorem \ref{2.3} we proved that for every fixed (period) $m\in \mathbb{Z}_{\geq 2}$, the $m$-periodic point-counting function $N_{c}^{(m)}(p) = p$ or $0$ for every fixed inert prime $p\geq 3$ and every coefficient $c\in \mathcal{O}_{K}$ divisible or indivisible by $p$. But now for every fixed (period) $m\geq 2$, we then note that the function $N_{c}^{(m)}(p) = N_{c}(p) = p$ (for every $\ell \in \{1,p\}$) or $0$ for every fixed inert prime $p$ and for every coefficient $c\in \mathcal{O}_{K}$ divisible or indivisible by $p$. This then also means that for every fixed (period) $m\in \mathbb{Z}_{\geq 1}$, every $m$-periodic integral orbit of any $\varphi_{p^{\ell},c}$ modulo $p\mathcal{O}_{K}$ is a fixed integral orbit, and moreover every $\varphi_{p^{\ell},c}$ modulo $p\mathcal{O}_{K}$ has $p$ distinct fixed integral orbits or zero; a somewhat interesting precise arithmetic-geometric insight on all the $m$-periodic integral orbits of any  $\varphi_{p^{\ell},c}$ modulo $p\mathcal{O}_{K}$. 
\end{rem}

\section{On Number of $m$-Periodic Integral Points of any Family of Polynomial Maps $\varphi_{(p-1)^{\ell},c}$}\label{sec3}

As in Section \ref{sec2}, we in section also  wish to count the number of distinct $m$-periodic integral points of any polynomial map $\varphi_{(p-1)^{\ell},c}$ modulo prime $p\mathcal{O}_{K}$, where $p\geq 5$ is any given prime, $\ell \geq 1$ is any integer and $m\geq 2$ is any fixed integers. As before, we let $p\geq 5$ be any prime,  $\ell \geq 1$ be any integer, $c\in \mathcal{O}_{K}$ be any point and $m\geq 2$ be any fixed integer, and then define analogously the following $m$-periodic point-counting function 
\begin{equation}\label{M_{c}}
M_{c}^{(m)}(p) := \# \Biggl\{ z\in \mathcal{O}_{K}\slash p\mathcal{O}_{K}   : \begin{aligned} \varphi_{(p-1)^{\ell},c}^{m-1}(z) -z \not \equiv 0 \ \text{(mod $p\mathcal{O}_{K}$)} \\ \ \varphi_{(p-1)^{\ell},c}^{m}(z) - z \equiv 0 \ \text{(mod $p\mathcal{O}_{K}$)} \end{aligned} \Biggr\}.
\end{equation}
\noindent Again, setting $\ell =1$ and so $\varphi_{(p-1)^{\ell}, c} = \varphi_{p-1,c}$, we first prove the following theorem and its generalization \ref{3.2}:

\begin{thm} \label{3.1}
Let $K\slash \mathbb{Q}$ be any number field of degree $n\geq 2$ with the ring of integers $\mathcal{O}_{K}$, and in which $5$ is inert. Let $\varphi_{4, c}$ be a quartic map defined by $\varphi_{4, c}(z) = z^4 + c$ for all $c, z\in\mathcal{O}_{K}$, and $M_{c}^{(m)}(5)$ be as in \textnormal{(\ref{M_{c}})}. Then $M_{c}^{(m)}(5) = 1$ for every $c\equiv \pm 1 \ (mod \ 5\mathcal{O}_{K})$ and fixed (even) $m$ or $M_{c}^{(m)}(5) = 2$ for every $c\in 5\mathcal{O}_{K}$; otherwise $M_{c}^{(m)}(5) = 0$ for every point $c \equiv -1\ (mod \ 5\mathcal{O}_{K})$ and fixed odd $m$ or $c\not \equiv \pm 1, 0 \ (mod \ 5\mathcal{O}_{K})$ and fixed (even) $m$.
\end{thm}

\begin{proof}
Let $g(z)= \varphi_{4,c}^m(z)-z = \varphi_{4,c}(\varphi_{4,c}^{m-1}(z)) - z = (\varphi_{4,c}^{m-1}(z))^4 - z + c$, and so $g(z)= (\varphi_{4,c}^{m-1}(z))^4 - z + c$. Now applying the multinomial theorem repeatedly on $(\varphi_{4,c}^{m-1}(z))^4$ right after applying the binomial theorem on $(z^4 + c)^4$, it then follows $(\varphi_{4,c}^{m-1}(z))^4$ is a monic polynomial in $z$ of degree $4^m$ with integral coefficients in multiples of $c$. Hence, we may then write $(\varphi_{4,c}^{m-1}(z))^4 = z^{4^{m}} + h(z)$, where $h(z)$ is a non-constant polynomial in $z$ of deg$(h)<4^m$ with integral coefficients in multiples of $c$; and so $g(z)= z^{4^{m}} + h(z) - z + c$. Now for every coefficient $c\in 5\mathcal{O}_{K}$, reducing $g(z)$ modulo prime ideal $5\mathcal{O}_{K}$, it then follows $g(z)\equiv z^{4^m} - z$ (mod $5\mathcal{O}_{K}$), since also $h(z)\in c\mathcal{O}_{K}[z]$ and so $h(z)\equiv 0$ (mod $5\mathcal{O}_{K}$); and thus now $g(z)$ modulo $5\mathcal{O}_{K}$ is a polynomial defined over a finite field $\mathcal{O}_{K}\slash 5\mathcal{O}_{K}$ of order $5^{[K:\mathbb{Q}]}=5^n$. So now, since (from a well-known fact) $\mathbb{F}_{5}\hookrightarrow \mathcal{O}_{K}\slash 5\mathcal{O}_{K}$ is an inclusion of fields and since $z^4 = 1$ for every $z\in \mathbb{F}_{5}^{\times} =\mathbb{F}_{5}\setminus \{0\}$, it then follows $z^{4^m}= (z^4)^{4^{m-1}} = 1$ for every $z\in \mathbb{F}_{5}^{\times}$ and for every $m\in \mathbb{Z}_{\geq 2}$. But now $g(z)\equiv 1 - z$ (mod $5\mathcal{O}_{K}$) for every point $z\in \mathbb{F}_{5}^{\times}\subset \mathcal{O}_{K}\slash 5\mathcal{O}_{K}$; and so $g(z)$ has a root in $\mathcal{O}_{K}\slash 5\mathcal{O}_{K}$, namely, $z\equiv 1$ (mod $5\mathcal{O}_{K}$). Moreover, since $z$ is a linear factor of $g(z)\equiv z(z^{4^m-1} - 1)$ (mod $5\mathcal{O}_{K}$), it then follows $z\equiv 0$ (mod $5\mathcal{O}_{K}$) is also root of $g(z)$ modulo $5\mathcal{O}_{K}$ in $\mathcal{O}_{K}\slash 5\mathcal{O}_{K}$. But now we then conclude $M_{c}^{(m)}(5) = 2$. To see $M_{c}^{(m)}(5) = 1$ for every coefficient $c\equiv 1$ (mod $5\mathcal{O}_{K}$) and for every fixed $m\in \mathbb{Z}_{\geq 2}$, we first note that writing $\varphi_{4,c}^{m-1}(z) = \underbrace{((((z^4 + c)^4 + c)^4 + c)^4 + \cdots + c)^4 + c}_\text{$(m-1)$ times}$ and then reducing $\varphi_{4,c}^{m-1}(z)$ modulo $5\mathcal{O}_{K}$ along with $c\equiv 1$ (mod $5\mathcal{O}_{K}$) and $z^4 = 1$ for every element $z\in \mathbb{F}_{5}^{\times}$, it then follows $\varphi_{4,c}^{m-1}(z)\equiv 2$ (mod $5\mathcal{O}_{K}$) and $(\varphi_{4,c}^{m-1}(z))^4\equiv 1$ (mod $5\mathcal{O}_{K}$) for every fixed $m$, since also $2^4 = 1$ in $\mathbb{F}_{5}$. But then $g(z)=(\varphi_{4,c}^{m-1}(z))^4 - z + c\equiv 2-z$ (mod $5\mathcal{O}_{K}$) for every point $z\in \mathbb{F}_{5}^{\times}\subset \mathcal{O}_{K}\slash5\mathcal{O}_{K}$ and so $g(z)$ modulo $5\mathcal{O}_{K}$ has a root in $\mathcal{O}_{K}\slash5\mathcal{O}_{K}$, namely, $z\equiv 2$ (mod $5\mathcal{O}_{K}$); and so we conclude $M_{c}^{(m)}(5) = 1$. We now show $M_{c}^{(m)}(5) = 1$ for every coefficient $c\equiv -1$ (mod $5\mathcal{O}_{K}$) and every fixed even integer $m\in \mathbb{Z}_{\geq 2}$. As before, since $c\equiv -1$ (mod $5\mathcal{O}_{K}$) and $z^4 = 1$ for every $z\in \mathbb{F}_{5}^{\times}$, then reducing $\varphi_{4,c}^m(z)$ modulo $5\mathcal{O}_{K}$, it then follows $\varphi_{4,c}^{m}(z)\equiv -1$ (mod $5\mathcal{O}_{K}$) for every fixed even $m\in \mathbb{Z}_{\geq 2}$. But then $g(z)= \varphi_{4,c}^m(z)-z \equiv -(1+z)$ (mod $5\mathcal{O}_{K}$) for every $z\in \mathbb{F}_{5}^{\times}\subset \mathcal{O}_{K}\slash5\mathcal{O}_{K}$ and every fixed even $m$ and so $g(z)$ modulo $5\mathcal{O}_{K}$ has a root in $\mathcal{O}_{K}\slash5\mathcal{O}_{K}$, namely, $z\equiv -1$ (mod $5\mathcal{O}_{K}$); and so we then conclude $M_{c}^{(m)}(5) = 1$. 

Finally, we now show $M_{c}^{(m)}(5) = 0$ for every coefficient $c \equiv -1$ (mod $5\mathcal{O}_{K}$) and every fixed odd integer $m\in \mathbb{Z}_{\geq 3}$ or for every coefficient $c\not \equiv \pm1, 0$ (mod $5\mathcal{O}_{K}$) and every fixed (even) integer $m\in \mathbb{Z}_{\geq 2}$. Let's for the sake of a contradiction, suppose $g(z) = \varphi_{4,c}^m(z)-z\equiv 0$ (mod (mod $5\mathcal{O}_{K}$) for some point $z\in \mathcal{O}_{K}\slash 5\mathcal{O}_{K}$ and for every $c \equiv -1$ (mod $5\mathcal{O}_{K}$) and every fixed odd $m\in \mathbb{Z}_{\geq 3}$. So then, since $c\equiv -1$ (mod $5\mathcal{O}_{K}$) and also since (as a fact) $z^4 = 1$ for every $z\in \mathbb{F}_{5}^{\times}$, reducing $\varphi_{4,c}^m(z)$ modulo $5\mathcal{O}_{K}$, we then obtain $\varphi_{4,c}^{m}(z)\equiv 0$ (mod $5\mathcal{O}_{K}$) for every fixed odd $m$; and so $g(z)= \varphi_{4,c}^m(z)-z \equiv -z$ (mod $5\mathcal{O}_{K}$) for every point $z\in \mathbb{F}_{5}^{\times}\subset \mathcal{O}_{K}\slash5\mathcal{O}_{K}$. But now we note $z\equiv 0$ (mod $5\mathcal{O}_{K}$) is a root of $g(z)$ modulo $5\mathcal{O}_{K}$ for every coefficient $c\equiv -1$ (mod $5\mathcal{O}_{K}$) and every fixed odd $m$, and also $z\equiv 0$ (mod $5\mathcal{O}_{K}$) is root of $g(z)$ modulo $5\mathcal{O}_{K}$ for every coefficient $c\equiv 0$ (mod $5\mathcal{O}_{K}$) and every odd fixed $m$, as seen earlier; and so follows a contradiction that $1 \equiv 0$ (mod $5\mathcal{O}_{K}$). Otherwise, suppose $g(z)\equiv 0$ (mod $5\mathcal{O}_{K}$) and so (from earlier) $z^{4^m} + h(z)-z+c\equiv 0$ (mod $5\mathcal{O}_{K}$) for some point $z\in \mathcal{O}_{K}\slash 5\mathcal{O}_{K}\setminus \{0\}$ and for every $c\not \equiv \pm1, 0$ (mod $5\mathcal{O}_{K}$) and every fixed (even) $m\in \mathbb{Z}_{\geq 2}$. So then, using that $z^{4^m}=1$ for every $z\in \mathbb{F}_{5}^{\times}$ and every fixed (even) $m\in \mathbb{Z}_{\geq 2}$, we then rewrite  $z^{4^m} - z + h(z)+c\equiv 0$ (mod $5\mathcal{O}_{K}$) to obtain $(1-z) + (h(z)+c)\equiv 0$ (mod $5\mathcal{O}_{K}$). But now we note $(1-z) + (h(z)+c)\equiv 0$ (mod $5\mathcal{O}_{K}$) can also happen if $1-z\equiv 0$ (mod $5\mathcal{O}_{K}$) and also $h(z)+c\equiv 0$ (mod $5\mathcal{O}_{K}$). Moreover, recall from earlier $1-z\equiv 0$ (mod $5\mathcal{O}_{K}$) when $c\equiv 0$ (mod $5\mathcal{O}_{K}$); which then contradicts $c\not \equiv \pm1, 0$ (mod $5\mathcal{O}_{K}$). Otherwise, suppose $g(z) \equiv 0$ (mod $5\mathcal{O}_{K}$) for some $\mathcal{O}_{K} / 5\mathcal{O}_{K}\setminus \mathbb{F}_{5}$ and for every $c \not \equiv \pm 1, 0$ (mod $5\mathcal{O}_{K}$) and every fixed $m\in \mathbb{Z}_{\geq 2}$. So then, since from earlier $g(z)=z^{4^m} + h(z)-z+c$ where $h(z)\in c\mathcal{O}_{K}[z]$, it then also follows $z^{4^{m}} -z + h(z) + c\equiv 0$ (mod $5\mathcal{O}_{K}$) for some $\mathcal{O}_{K} / 5\mathcal{O}_{K}\setminus \mathbb{F}_{5}$ and for every $c \not \equiv \pm 1, 0$ (mod $5\mathcal{O}_{K}$) and every fixed $m$. But now we note $(z^{4^{m}} -z) + (h(z) + c)\equiv 0$ (mod $5\mathcal{O}_{K}$) can also occur if $z^{4^{m}} -z \equiv 0$ (mod $5\mathcal{O}_{K}$) and also $h(z) + c\equiv 0$ (mod $5\mathcal{O}_{K}$). Moreover, we also note $z^{4^{m}} -z \equiv 0$ (mod $5\mathcal{O}_{K}$) for every point $z\equiv 0$ (mod $5\mathcal{O}_{K}$), which also occurred earlier when $c\equiv 0$ (mod $5\mathcal{O}_{K}$); and so also follows a contradiction. This then overall means that $g(x)=\varphi_{4,c}^m(x)-x$ has no roots in $\mathcal{O}_{K} / 5\mathcal{O}_{K}$ for every coefficient $c \equiv -1$ (mod $5\mathcal{O}_{K}$) and every fixed odd $m\in \mathbb{Z}_{\geq 3}$ or for every coefficient $c\not \equiv \pm1, 0$ (mod $5\mathcal{O}_{K}$) and every fixed (even) $m\in \mathbb{Z}_{\geq 2}$; and so we then conclude $M_{c}^{(m)}(5) = 0$. This then completes the whole proof, as needed.
\end{proof} 
We now wish to generalize Theorem \ref{3.1} to any map $\varphi_{p-1, c}$ for any given prime $p\geq 5$. More precisely, we prove 
that the number of distinct $m$-periodic integral points of any map $\varphi_{p-1, c}$ modulo $p\mathcal{O}_{K}$ is $1$ or $2$ or $0$:

\begin{thm} \label{3.2}
Let $K\slash \mathbb{Q}$ be any number field of degree $n\geq 2$ with ring $\mathcal{O}_{K}$, and in which any fixed prime $p\geq 5$ is inert. Let $\varphi_{p-1, c}$ be defined by $\varphi_{p-1, c}(z) = z^{p-1} + c$ for all $c, z\in\mathcal{O}_{K}$, and $M_{c}^{(m)}(p)$ be as in \textnormal{(\ref{M_{c}})}. Then $M_{c}^{(m)}(p) = 1$ for every $c\equiv \pm 1 \ (mod \ p\mathcal{O}_{K})$ and fixed (even) $m$ or $M_{c}^{(m)}(p) = 2$ for every $c\in p\mathcal{O}_{K}$; otherwise $M_{c}^{(m)}(p) = 0$ for every point $c \equiv -1\ (mod \ p\mathcal{O}_{K})$ and fixed odd $m$ or $c\not \equiv \pm 1, 0 \ (mod \ p\mathcal{O}_{K})$ and fixed (even) $m$.
\end{thm}
\begin{proof}
By applying a similar argument as in the Proof of Theorem \ref{3.1}, we then obtain the count as desired. That is, let $g(z)= \varphi_{p-1,c}^m(z)-z = \varphi_{p-1,c}(\varphi_{p-1,c}^{m-1}(z)) - z = (\varphi_{p-1,c}^{m-1}(z))^{p-1} - z + c$, and so $g(z)= (\varphi_{p-1,c}^{m-1}(z))^{p-1} - z + c$. So now, applying the multinomial theorem repeatedly on $(\varphi_{p-1,c}^{m-1}(z))^{p-1}$ right after applying the binomial theorem on $(z^{p-1} + c)^{p-1}$, it then follows $(\varphi_{p-1,c}^{m-1}(z))^{p-1}$ is a monic polynomial in $z$ of degree $(p-1)^m$ with integral coefficients in multiples of $c$. Thus, we may then write $(\varphi_{p-1,c}^{m-1}(z))^{p-1} = z^{(p-1)^{m}} + h(z)$, where $h(z)$ is a non-constant polynomial in $z$ of deg$(h)<(p-1)^m$ with integral coefficients in multiples of $c$; and so $g(z)= z^{(p-1)^{m}} + h(z) - z + c$. Now for every coefficient $c\in p\mathcal{O}_{K}$, reducing $g(z)$ modulo prime ideal $p\mathcal{O}_{K}$, it then follows $g(z)\equiv z^{(p-1)^m} - z$ (mod $p\mathcal{O}_{K}$), since also $h(z)\in c\mathcal{O}_{K}[z]$ and so $h(z)\equiv 0$ (mod $p\mathcal{O}_{K}$); and so now $g(z)$ modulo $p\mathcal{O}_{K}$ is a polynomial defined over a finite field $\mathcal{O}_{K}\slash p\mathcal{O}_{K}$ of order $p^{[K:\mathbb{Q}]}=p^n$. Now recall the inclusion  $\mathbb{F}_{p}\hookrightarrow \mathcal{O}_{K}\slash p\mathcal{O}_{K}$ of fields, and also recall (as a fact) $z^{p-1} = 1$ for every element $z\in \mathbb{F}_{p}^{\times}=\mathbb{F}_{p}\setminus \{0\}$, it then follows $z^{(p-1)^m}= (z^{p-1})^{(p-1)^{m-1}} = 1$ for every element $z\in \mathbb{F}_{p}^{\times}$ and for every fixed $m\in \mathbb{Z}_{\geq 2}$ and so $g(z)\equiv 1 - z$ (mod $p\mathcal{O}_{K}$) for every point $z\in \mathbb{F}_{p}^{\times}\subset \mathcal{O}_{K}\slash p\mathcal{O}_{K}$; and so $g(z)$ has a root in $\mathcal{O}_{K}\slash p\mathcal{O}_{K}$, namely, $z\equiv 1$ (mod $p\mathcal{O}_{K}$). Moreover, since $z$ is a linear factor of $g(z)\equiv z(z^{(p-1)^m-1} - 1)$ (mod $p\mathcal{O}_{K}$), it then follows $z\equiv 0$ (mod $p\mathcal{O}_{K}$) is also root of $g(z)$ modulo $p\mathcal{O}_{K}$. But now we then conclude $M_{c}^{(m)}(p) = 2$. To see $M_{c}^{(m)}(p) = 1$ for every coefficient $c\equiv 1$ (mod $p\mathcal{O}_{K}$) and for every fixed $m\in \mathbb{Z}_{\geq 2}$, we first note that writing $\varphi_{p-1,c}^{m-1}(z) = \underbrace{((((z^{p-1} + c)^{p-1} + c)^{p-1} + c)^{p-1} + \cdots + c)^{p-1} + c}_\text{$(m-1)$ times}$ and then reducing $\varphi_{p-1,c}^{m-1}(z)$ modulo $p\mathcal{O}_{K}$ along with $c\equiv 1$ (mod $p\mathcal{O}_{K}$) and also $z^{p-1} = 1$ for every element $z\in \mathbb{F}_{p}^{\times}$, it then follows $\varphi_{p-1,c}^{m-1}(z)\equiv 2$ (mod $p\mathcal{O}_{K}$) and $(\varphi_{p-1,c}^{m-1}(z))^{p-1}\equiv 1$ (mod $p\mathcal{O}_{K}$) for every fixed $m$, since also $2^{p-1} = 1$ in $\mathbb{F}_{p}^{\times}$. But then $g(z)=(\varphi_{p-1,c}^{m-1}(z))^{p-1} - z + c\equiv 2-z$ (mod $p\mathcal{O}_{K}$) for every point $z\in \mathbb{F}_{p}^{\times}\subset \mathcal{O}_{K}\slash p\mathcal{O}_{K}$ and so $g(z)$ modulo $p\mathcal{O}_{K}$ has a root in $\mathcal{O}_{K}\slash p\mathcal{O}_{K}$, namely, $z\equiv 2$ (mod $p\mathcal{O}_{K}$); and so we then conclude $M_{c}^{(m)}(p) = 1$. We now show $M_{c}^{(m)}(p) = 1$ for every coefficient $c\equiv -1$ (mod $p\mathcal{O}_{K}$) and every fixed even integer $m\in \mathbb{Z}_{\geq 2}$. As before, we first note that since $c\equiv -1$ (mod $p\mathcal{O}_{K}$) and also since $z^{p-1} = 1$ for every element $z\in \mathbb{F}_{p}^{\times}$, reducing $\varphi_{p-1,c}^m(z)$ modulo $p\mathcal{O}_{K}$, it then follows $\varphi_{p-1,c}^{m}(z)\equiv -1$ (mod $p\mathcal{O}_{K}$) for every fixed even $m\in \mathbb{Z}_{\geq 2}$. But then $g(z)= \varphi_{p-1,c}^m(z)-z \equiv -(1+z)$ (mod $p\mathcal{O}_{K}$) for every point $z\in \mathbb{F}_{p}^{\times}\subset \mathcal{O}_{K}\slash p\mathcal{O}_{K}$ and every fixed even $m$ and so $g(z)$ modulo $p\mathcal{O}_{K}$ has a root in $\mathcal{O}_{K}\slash p\mathcal{O}_{K}$, namely, $z\equiv -1$ (mod $p\mathcal{O}_{K}$); and so we then conclude $M_{c}^{(m)}(p) = 1$. 

Finally, we now show $M_{c}^{(m)}(p) = 0$ for every coefficient $c \equiv -1$ (mod $p\mathcal{O}_{K}$) and every fixed odd integer $m\in \mathbb{Z}_{\geq 3}$ or for every coefficient $c\not \equiv \pm1, 0$ (mod $p\mathcal{O}_{K}$) and every fixed (even) integer $m\in \mathbb{Z}_{\geq 2}$. To see this, let's for the sake of a contradiction, suppose $g(z) = \varphi_{p-1,c}^m(z)-z\equiv 0$ (mod (mod $p\mathcal{O}_{K}$) for some $z\in \mathcal{O}_{K}\slash p\mathcal{O}_{K}$ and for every $c \equiv -1$ (mod $p\mathcal{O}_{K}$) and every fixed odd $m\in \mathbb{Z}_{\geq 3}$. So then, since $c\equiv -1$ (mod $p\mathcal{O}_{K}$) and also since (as a fact) $z^{p-1} = 1$ for every $z\in \mathbb{F}_{p}^{\times}$, reducing $\varphi_{p-1,c}^m(z)$ modulo $p\mathcal{O}_{K}$, it then follows $\varphi_{p-1,c}^{m}(z)\equiv 0$ (mod $p\mathcal{O}_{K}$) for every fixed odd $m$; and so $g(z)= \varphi_{p-1,c}^m(z)-z \equiv -z$ (mod $p\mathcal{O}_{K}$) for every point $z\in \mathbb{F}_{p}^{\times}\subset \mathcal{O}_{K}\slash p\mathcal{O}_{K}$. But now we note $z\equiv 0$ (mod $p\mathcal{O}_{K}$) is a root of $g(z)$ modulo $p\mathcal{O}_{K}$ for every coefficient $c\equiv -1$ (mod $p\mathcal{O}_{K}$) and every fixed odd $m$, and also $z\equiv 0$ (mod $p\mathcal{O}_{K}$) is a root of $g(z)$ modulo $p\mathcal{O}_{K}$ for every coefficient $c\equiv 0$ (mod $p\mathcal{O}_{K}$) and every fixed odd $m$, as seen earlier; and so follows a contradiction that $1 \equiv 0$ (mod $p\mathcal{O}_{K}$). Otherwise, suppose $g(z)\equiv 0$ (mod $p\mathcal{O}_{K}$) and so (from earlier) $z^{(p-1)^m} + h(z)-z+c\equiv 0$ (mod $p\mathcal{O}_{K}$) for some $z\in \mathcal{O}_{K}\slash p\mathcal{O}_{K}\setminus \{0\}$ and for every $c\not \equiv \pm1, 0$ (mod $p\mathcal{O}_{K}$) and for every fixed (even) $m\in \mathbb{Z}_{\geq 2}$. So then, using that $z^{(p-1)^m}=1$ for every $z\in \mathbb{F}_{p}^{\times}$, we then write $z^{(p-1)^m} - z + h(z)+c\equiv 0$ (mod $p\mathcal{O}_{K}$) to obtain $(1-z) + (h(z)+c)\equiv 0$ (mod $p\mathcal{O}_{K}$). But now we note $(1-z) + (h(z)+c)\equiv 0$ (mod $p\mathcal{O}_{K}$) can also happen if $1-z\equiv 0$ (mod $p\mathcal{O}_{K}$) and also $h(z)+c\equiv 0$ (mod $p\mathcal{O}_{K}$). Moreover, recall from earlier $1-z\equiv 0$ (mod $p\mathcal{O}_{K}$) when $c\equiv 0$ (mod $p\mathcal{O}_{K}$); which then contradicts $c\not \equiv \pm1, 0$ (mod $p\mathcal{O}_{K}$). Otherwise, suppose $g(z) \equiv 0$ (mod $p\mathcal{O}_{K}$) for some point $\mathcal{O}_{K} / p\mathcal{O}_{K}\setminus \mathbb{F}_{p}$ and for every $c \not \equiv \pm 1, 0$ (mod $p\mathcal{O}_{K}$) and for every fixed $m\in \mathbb{Z}_{\geq 2}$. Now since from earlier $g(z)=z^{(p-1)^m} + h(z)-z+c$ where $h(z)\in c\mathcal{O}_{K}[z]$, it then follows $z^{(p-1)^{m}} -z + h(z) + c\equiv 0$ (mod $p\mathcal{O}_{K}$) for some $\mathcal{O}_{K} / p\mathcal{O}_{K}\setminus \mathbb{F}_{p}$ and for every $c \not \equiv \pm 1, 0$ (mod $p\mathcal{O}_{K}$) and every fixed $m$. But now, as before we note $(z^{(p-1)^{m}} -z) + (h(z) + c)\equiv 0$ (mod $p\mathcal{O}_{K}$) can also occur if $z^{(p-1)^{m}} -z \equiv 0$ (mod $p\mathcal{O}_{K}$) and also $h(z) + c\equiv 0$ (mod $p\mathcal{O}_{K}$). Moreover, we also note $z^{(p-1)^{m}} -z \equiv 0$ (mod $p\mathcal{O}_{K}$) for every $z\equiv 0$ (mod $p\mathcal{O}_{K}$), which also occurred earlier when $c\equiv 0$ (mod $p\mathcal{O}_{K}$); and so also follows a contradiction. This then overall means $g(x)=\varphi_{p-1,c}^m(x)-x$ has no roots in $\mathcal{O}_{K} / p\mathcal{O}_{K}$ for every coefficient $c \equiv -1$ (mod $p\mathcal{O}_{K}$) and fixed odd $m\in \mathbb{Z}_{\geq 3}$ or for every coefficient $c\not \equiv \pm1, 0$ (mod $p\mathcal{O}_{K}$) and fixed (even) $m\in \mathbb{Z}_{\geq 2}$; and so we conclude $M_{c}^{(m)}(p) = 0$. This then completes the whole proof, as needed. 
\end{proof}

Finally, we wish to generalize Theorem \ref{3.2} further to any $\varphi_{(p-1)^{\ell}, c}$ for any prime $p\geq 5$ and any $\ell\in \mathbb{Z}_{\geq 1}$. That is, we prove the number of distinct $m$-periodic points of any $\varphi_{(p-1)^{\ell}, c}$ modulo $p\mathcal{O}_{K}$ is also $1$ or $2$ or zero:

\begin{thm} \label{3.3}
Let $K\slash \mathbb{Q}$ be any number field of degree $n\geq 2$ with ring $\mathcal{O}_{K}$, and in which any fixed prime $p\geq 5$ is inert. Let $\varphi_{(p-1)^{\ell}, c}(z) = z^{(p-1)^{\ell}} + c$ for all $c, z\in\mathcal{O}_{K}$ and $\ell\in \mathbb{Z}_{\geq 1}$. Let $M_{c}^{(m)}(p)$ be as in \textnormal{(\ref{M_{c}})}. Then $M_{c}^{(m)}(p) = 1$ for every $c\equiv \pm 1 \ (mod \ p\mathcal{O}_{K})$ and fixed (even) $m$ or $M_{c}^{(m)}(p) = 2$ for every $c\in p\mathcal{O}_{K}$; otherwise $M_{c}^{(m)}(p) = 0$ for every point $c \equiv -1\ (mod \ p\mathcal{O}_{K})$ and fixed odd $m$ or $c\not \equiv \pm 1, 0 \ (mod \ p\mathcal{O}_{K})$ and fixed (even) $m$.
\end{thm}
\begin{proof}
By applying a similar argument as in the Proof of Theorem \ref{3.2}, we then obtain the count as desired. That is, let $g(z)= \varphi_{(p-1)^{\ell}, c}^m(z)-z = \varphi_{(p-1)^{\ell},c}(\varphi_{(p-1)^{\ell},c}^{m-1}(z))-z = (\varphi_{(p-1)^{\ell},c}^{m-1}(z))^{(p-1)^{\ell}}-z + c$, and so $g(z) = (\varphi_{(p-1)^{\ell},c}^{m-1}(z))^{(p-1)^{\ell}}-z + c$. Now applying the multinomial theorem on $(\varphi_{p-1,c}^{m-1}(z))^{(p-1)^{\ell}}$ right after applying the binomial theorem on $(z^{(p-1)^{\ell}}+c)^{(p-1)^{\ell}}$, it then follows $(\varphi_{(p-1)^{\ell},c}^{m-1}(z))^{(p-1)^{\ell}}$ is a monic polynomial in $z$ of degree $(p-1)^{m\ell}$ with integral coefficients in multiples of $c$. Hence, we may then write $(\varphi_{(p-1)^{\ell},c}^{m-1}(z))^{(p-1)^{\ell}} = z^{(p-1)^{m\ell}} + h(z)$, where $h(z)$ is a non-constant polynomial in $z$ of deg$(h)<(p-1)^{m\ell}$ with integral coefficients in multiples of $c$; and so $g(z)= z^{(p-1)^{m\ell}} + h(z) - z + c$. Now for every coefficient $c \in p\mathcal{O}_{K}$, reducing $g(z)$ modulo prime ideal $p\mathcal{O}_{K}$, we then obtain $g(z)\equiv z^{(p-1)^{m\ell}} - z$ (mod $p\mathcal{O}_{K}$), since also $h(z)\in c\mathcal{O}_{K}[z]$ and so $h(z)\equiv 0$ (mod $p\mathcal{O}_{K}$); and so now $g(z)$ modulo $p\mathcal{O}_{K}$ is a polynomial defined over a finite field $\mathcal{O}_{K}\slash p\mathcal{O}_{K}$. So now, recall the inclusion    $\mathbb{F}_{p}\hookrightarrow\mathcal{O}_{K}\slash p\mathcal{O}_{K}$ of fields and also note (from a well-known fact) $z^{(p-1)^m} = 1$ for every $z\in \mathbb{F}_{p}^{\times}$ and for every fixed $m\in \mathbb{Z}_{\geq 2}$, it then follows $z^{(p-1)^{m\ell}} =1$ for every element $z\in \mathbb{F}_{p}^{\times}$ and for every fixed $\ell \in \mathbb{Z}_{\geq 1}$. But then $g(z)\equiv 1 - z$ (mod $p\mathcal{O}_{K}$) for every point $z\in \mathbb{F}_{p}^{\times}\subset\mathcal{O}_{K}\slash p\mathcal{O}_{K}$; and so $g(z)$ has a root in $\mathcal{O}_{K}\slash p\mathcal{O}_{K}$, namely, $z\equiv 1$ (mod $p\mathcal{O}_{K}$). Moreover, since $z$ is also a linear factor of $g(z)\equiv z(z^{(p-1)^{m\ell}-1} - 1)$ (mod $p\mathcal{O}_{K}$), it then follows $z\equiv 0$ (mod $p\mathcal{O}_{K}$) is also a root of $g(z)$ modulo $p\mathcal{O}_{K}$ in $\mathcal{O}_{K}\slash p\mathcal{O}_{K}$. But now we then conclude $M_{c}^{(m)}(p) = 2$. To see $M_{c}^{(m)}(p) = 1$ for every coefficient $c\equiv 1$ (mod $p\mathcal{O}_{K}$) and for every fixed $\ell \in \mathbb{Z}_{\geq 1}$ and $m\in \mathbb{Z}_{\geq 2}$, we note that writing $\varphi_{(p-1)^{\ell},c}^{m-1}(z) = \underbrace{((((z^{(p-1)^{\ell}} + c)^{(p-1)^{\ell}} + c)^{(p-1)^{\ell}} + c)^{(p-1)^{\ell}} + \cdots + c)^{(p-1)^{\ell}} + c}_\text{$(m-1)$ times}$ and reducing $\varphi_{(p-1)^{\ell},c}^{m-1}(z)$ modulo $p\mathcal{O}_{K}$ along with $c\equiv 1$ (mod $p\mathcal{O}_{K}$) and also using $z^{(p-1)^{\ell}} = 1$ for every $z\in \mathbb{F}_{p}^{\times}$, it then follows $\varphi_{(p-1)^{\ell},c}^{m-1}(z)\equiv 2$ (mod $p\mathcal{O}_{K}$) and $(\varphi_{(p-1)^{\ell},c}^{m-1}(z))^{(p-1)^{\ell}}\equiv 1$ (mod $p\mathcal{O}_{K}$) for every fixed $m$, since also $2^{(p-1)^{\ell}} = 1$ in $\mathbb{F}_{p}^{\times}$. But then $g(z)=(\varphi_{(p-1)^{\ell},c}^{m-1}(z))^{(p-1)^{\ell}} - z + c\equiv 2-z$ (mod $p\mathcal{O}_{K}$) for every point $z\in \mathbb{F}_{p}^{\times}\subset \mathcal{O}_{K}\slash p\mathcal{O}_{K}$ and so $g(z)$ modulo $p\mathcal{O}_{K}$ has a root in $\mathcal{O}_{K}\slash p\mathcal{O}_{K}$, namely, $z\equiv 2$ (mod $p\mathcal{O}_{K}$); and so we conclude $M_{c}^{(m)}(p) = 1$. We now show $M_{c}^{(m)}(p) = 1$ for every coefficient $c\equiv -1$ (mod $p\mathcal{O}_{K}$) and for every fixed $\ell \in \mathbb{Z}_{\geq 1}$ and every fixed even integer $m\in \mathbb{Z}_{\geq 2}$. As before, since $c\equiv -1$ (mod $p\mathcal{O}_{K}$) and also since $z^{(p-1)^{\ell}} = 1$ for every $z\in \mathbb{F}_{p}^{\times}$, reducing $\varphi_{(p-1)^{\ell},c}^m(z)$ modulo $p\mathcal{O}_{K}$, it then follows $\varphi_{(p-1)^{\ell},c}^{m}(z)\equiv -1$ (mod $p\mathcal{O}_{K}$) for every fixed even $m\in \mathbb{Z}_{\geq 2}$. But then $g(z)= \varphi_{(p-1)^{\ell},c}^m(z)-z \equiv -(1+z)$ (mod $p\mathcal{O}_{K}$) for every point $z\in \mathbb{F}_{p}^{\times}\subset \mathcal{O}_{K}\slash p\mathcal{O}_{K}$ and for every fixed even $m$; and so $g(z)$ modulo $p\mathcal{O}_{K}$ has a root in $\mathcal{O}_{K}\slash p\mathcal{O}_{K}$; and so we conclude $M_{c}^{(m)}(p) = 1$. 

Finally, we now show $M_{c}^{(m)}(p) = 0$ for every coefficient $c \equiv -1$ (mod $p\mathcal{O}_{K}$) and every fixed $\ell\in \mathbb{Z}_{\geq 1}$ and fixed odd $m\in \mathbb{Z}_{\geq 3}$ or for every coefficient $c\not \equiv \pm1, 0$ (mod $p\mathcal{O}_{K}$) and every fixed $\ell\in \mathbb{Z}_{\geq 1}$ and fixed (even) $m\in \mathbb{Z}_{\geq 2}$. As before, let's for the sake of a contradiction, suppose $g(z) = \varphi_{(p-1)^{\ell},c}^m(z)-z\equiv 0$ (mod (mod $p\mathcal{O}_{K}$) for some $z\in \mathcal{O}_{K}\slash p\mathcal{O}_{K}$ and for every $c \equiv -1$ (mod $p\mathcal{O}_{K}$) and every $\ell$ and every fixed odd $m\in \mathbb{Z}_{\geq 3}$. So then, since $c\equiv -1$ (mod $p\mathcal{O}_{K}$) and also since $z^{(p-1)^{\ell}} = 1$ for every $z\in \mathbb{F}_{p}^{\times}$, reducing $\varphi_{(p-1)^{\ell},c}^m(z)$ modulo $p\mathcal{O}_{K}$, it then follows $\varphi_{(p-1)^{\ell},c}^{m}(z)\equiv 0$ (mod $p\mathcal{O}_{K}$) for every fixed odd $m$; and so $g(z)= \varphi_{(p-1)^{\ell},c}^m(z)-z \equiv -z$ (mod $p\mathcal{O}_{K}$) for every point $z\in \mathbb{F}_{p}^{\times}\subset \mathcal{O}_{K}\slash p\mathcal{O}_{K}$. But now we note $z\equiv 0$ (mod $p\mathcal{O}_{K}$) is a root of $g(z)$ modulo $p\mathcal{O}_{K}$ for every coefficient $c\equiv -1$ (mod $p\mathcal{O}_{K}$) and every fixed odd $m$, and also $z\equiv 0$ (mod $p\mathcal{O}_{K}$) is a root of $g(z)$ modulo $p\mathcal{O}_{K}$ for every coefficient $c\equiv 0$ (mod $p\mathcal{O}_{K}$) and every fixed odd $m$, as seen earlier; and so follows a contradiction that $1 \equiv 0$ (mod $p\mathcal{O}_{K}$). Otherwise, suppose $g(z)\equiv 0$ (mod $p\mathcal{O}_{K}$) and so (from earlier) $z^{(p-1)^{m\ell}} + h(z)-z+c\equiv 0$ (mod $p\mathcal{O}_{K}$) for some $z\in \mathcal{O}_{K}\slash p\mathcal{O}_{K}\setminus \{0\}$ and for every $c\not \equiv \pm1, 0$ (mod $p\mathcal{O}_{K}$) and every fixed $\ell \in \mathbb{Z}_{\geq 1}$ and (even) $m\in \mathbb{Z}_{\geq 2}$. So now, using that $z^{(p-1)^{m\ell}}=1$ for every $z\in \mathbb{F}_{p}^{\times}$, we then write $(z^{(p-1)^{m\ell}} -z) + (h(z) +c)\equiv 0$ (mod $p\mathcal{O}_{K}$) to obtain $(1-z) + (h(z)+c)\equiv 0$ (mod $p\mathcal{O}_{K}$). But now we note $(1-z) + (h(z)+c)\equiv 0$ (mod $p\mathcal{O}_{K}$) can also happen if $1-z\equiv 0$ (mod $p\mathcal{O}_{K}$) and also $h(z)+c\equiv 0$ (mod $p\mathcal{O}_{K}$). Moreover, recall from earlier $1-z\equiv 0$ (mod $p\mathcal{O}_{K}$) when $c\equiv 0$ (mod $p\mathcal{O}_{K}$); and which then contradicts $c\not \equiv \pm1, 0$ (mod $p\mathcal{O}_{K}$). Otherwise, suppose $g(z) \equiv 0$ (mod $p\mathcal{O}_{K}$) for some point $\mathcal{O}_{K} / p\mathcal{O}_{K}\setminus \mathbb{F}_{p}$ and for every $c \not \equiv \pm 1, 0$ (mod $p\mathcal{O}_{K}$) and for every fixed $m\in \mathbb{Z}_{\geq 2}$. So then, since from earlier $g(z)=z^{(p-1)^{m\ell}} + h(z)-z+c$ where $h(z)\in c\mathcal{O}_{K}[z]$, it then also follows $z^{(p-1)^{m\ell}} -z + h(z) + c\equiv 0$ (mod $p\mathcal{O}_{K}$) for some $\mathcal{O}_{K} / p\mathcal{O}_{K}\setminus \mathbb{F}_{p}$ and for every $c \not \equiv \pm 1, 0$ (mod $p\mathcal{O}_{K}$) and every fixed $m$. But now we note $(z^{(p-1)^{m\ell}} -z) + (h(z) + c)\equiv 0$ (mod $p\mathcal{O}_{K}$) can also occur whenever $z^{(p-1)^{m\ell}} -z \equiv 0$ (mod $p\mathcal{O}_{K}$) and also $h(z) + c\equiv 0$ (mod $p\mathcal{O}_{K}$). Moreover, we also note $z^{(p-1)^{m\ell}} -z \equiv 0$ (mod $p\mathcal{O}_{K}$) for every $z\equiv 0$ (mod $p\mathcal{O}_{K}$), which also occurred earlier when $c\equiv 0$ (mod $p\mathcal{O}_{K}$); and so also follows a contradiction. This then overall means $g(x)=\varphi_{(p-1)^{\ell},c}^m(x)-x$ has no roots in $\mathcal{O}_{K} / p\mathcal{O}_{K}$ for every coefficient $c \equiv -1$ (mod $p\mathcal{O}_{K}$) and fixed $\ell\in \mathbb{Z}_{\geq 1}$ and fixed odd $m\in \mathbb{Z}_{\geq 3}$ or for every coefficient $c\not \equiv \pm1, 0$ (mod $p\mathcal{O}_{K}$) and fixed $\ell\in \mathbb{Z}_{\geq 1}$ and fixed (even) $m\in \mathbb{Z}_{\geq 2}$; and thus we conclude $M_{c}^{(m)}(p) = 0$. This then completes the whole proof, as needed.
\end{proof}

Restricting on a subring $\mathbb{Z}\subset \mathcal{O}_{K}$ of integers, we obtain the following consequence of Theorem \ref{3.3} on the number of distinct $m$-periodic integral points of any $\varphi_{(p-1)^{\ell},c}$ modulo $p$ for every prime $p\geq 5$ and any $\ell \in \mathbb{Z}_{\geq 1}$:

\begin{cor} \label{cor3.4}
Let $p\geq 5$ be any fixed prime integer, and $\ell \geq 1$ be any integer. Let $\varphi_{(p-1)^{\ell}, c}$ be defined by $\varphi_{(p-1)^{\ell}, c}(z) = z^{(p-1)^{\ell}} + c$ for all $c, z\in\mathbb{Z}$, and $M_{c}^{(m)}(p)$ be defined as in \textnormal{(\ref{M_{c}})} with $\mathcal{O}_{K} / p\mathcal{O}_{K}$ replaced with $\mathbb{Z}\slash p\mathbb{Z}$. Then $M_{c}^{(m)}(p) = 1$ for any coefficient $c\equiv \pm 1 \ (mod \ p)$ and fixed (even) $m$ or $M_{c}^{(m)}(p) = 2$ for any $c= pt$; otherwise $M_{c}^{(m)}(p) = 0$ for any $c \equiv -1\ (mod \ p)$ and fixed odd $m$ or any $c\not \equiv \pm 1, 0 \ (mod \ p)$ and fixed (even) $m$.
\end{cor}

\begin{proof}
By applying a similar argument as in the Proof of Theorem \ref{3.3}, we then obtain the count as desired. That is, let $g(z)= \varphi_{(p-1)^{\ell}, c}^m(z)-z = \varphi_{(p-1)^{\ell},c}(\varphi_{(p-1)^{\ell},c}(z))-z$, and so $g(z)= (\varphi_{p-1,c}^{m-1}(z))^{p-1} - z + c$. So now, as before applying the multinomial theorem on $(\varphi_{p-1,c}^{m-1}(z))^{(p-1)^{\ell}}$ and for every coefficient $c =pt$, reducing $g(z)$ modulo $p$, it then follows $g(z)\equiv z^{(p-1)^{m\ell}} - z$ (mod $p$); and so now $g(z)$ modulo $p$ is now a polynomial defined over field $\mathbb{Z}\slash p\mathbb{Z}$ of $p$ distinct elements. Now recall by Fermat's Little Theorem (FLT) that $z^{p-1} \equiv 1$ (mod $p$)(equivalently $z^{p-1} = 1$) for every $z\in (\mathbb{Z}\slash p\mathbb{Z})^{\times} = \mathbb{Z}\slash p\mathbb{Z}\setminus \{0\}$, it then also follows $z^{(p-1)^{m\ell}} \equiv 1$ (mod $p$) for every $z\in (\mathbb{Z}\slash p\mathbb{Z})^{\times}$ and for every fixed $\ell\in \mathbb{Z}_{\geq 1}$ and $m \in \mathbb{Z}_{\geq 2}$. But then $g(z)\equiv 1 - z$ (mod $p$) for every point $z\in (\mathbb{Z}\slash p\mathbb{Z})^{\times}$; and so $g(z)$ modulo $p$ has a root in $\mathbb{Z}\slash p\mathbb{Z}$. Moreover, since $z$ is also a linear factor of $g(z)\equiv z(z^{(p-1)^{m\ell}-1} - 1)$ (mod $p$), it then follows $z\equiv 0$ (mod $p$) is also a root of $g(z)$ modulo $p$ in $\mathbb{Z}\slash p\mathbb{Z}$. But now we then conclude $M_{c}^{(m)}(p) = 2$. To see $M_{c}^{(m)}(p) = 1$ for every coefficient $c\equiv 1$ (mod $p$) and every fixed $\ell \in \mathbb{Z}_{\geq 1}$ and $m\in \mathbb{Z}_{\geq 2}$, we again note that writing $\varphi_{(p-1)^{\ell},c}^{m-1}(z) = \underbrace{((((z^{(p-1)^{\ell}} + c)^{(p-1)^{\ell}} + c)^{(p-1)^{\ell}} + c)^{(p-1)^{\ell}} + \cdots + c)^{(p-1)^{\ell}} + c}_\text{$(m-1)$ times}$ and reducing $\varphi_{(p-1)^{\ell},c}^{m-1}(z)$ modulo $p$ along with $c\equiv 1$ (mod $p$) and also since $z^{(p-1)^{\ell}} = 1$ for every $z\in (\mathbb{Z}\slash p\mathbb{Z})^{\times}$, it then follows $\varphi_{(p-1)^{\ell},c}^{m-1}(z)\equiv 2$ (mod $p$) and $(\varphi_{(p-1)^{\ell},c}^{m-1}(z))^{(p-1)^{\ell}}\equiv 1$ (mod $p$) for every fixed $\ell$ and $m$, since also $2^{(p-1)^{\ell}} \equiv 1$ (mod $p$). But then $g(z)=(\varphi_{(p-1)^{\ell},c}^{m-1}(z))^{(p-1)^{\ell}} - z + c\equiv 2-z$ (mod $p$) for every point $z\in (\mathbb{Z}\slash p\mathbb{Z})^{\times}$ and so $g(z)$ modulo $p$ has a root in $\mathbb{Z}\slash p\mathbb{Z}$; and so we then conclude $M_{c}^{(m)}(p) = 1$. We now show $M_{c}^{(m)}(p) = 1$ for every coefficient $c\equiv -1$ (mod $p$) and for every fixed $\ell \in \mathbb{Z}_{\geq 1}$ and for every fixed even integer $m\in \mathbb{Z}_{\geq 2}$. As before, we note that since $c\equiv -1$ (mod $p$) and since also $z^{(p-1)^{\ell}} = 1$ for every $z\in (\mathbb{Z}\slash p\mathbb{Z})^{\times}$, reducing $\varphi_{(p-1)^{\ell},c}^m(z)$ modulo $p$, it then follows $\varphi_{(p-1)^{\ell},c}^{m}(z)\equiv -1$ (mod $p$) for every fixed even $m$. But then $g(z)= \varphi_{(p-1)^{\ell},c}^m(z)-z \equiv -(1+z)$ (mod $p$) for every point $z\in (\mathbb{Z}\slash p\mathbb{Z})^{\times}$; and so $g(z)$ modulo $p$ has a root in $\mathbb{Z}\slash p\mathbb{Z}$ and so we then conclude $M_{c}^{(m)}(p) = 1$. 

Finally, we now show $M_{c}^{(m)}(p) = 0$ for every coefficient $c \equiv -1$ (mod $p$) and every fixed odd $m\in \mathbb{Z}_{\geq 3}$ or for every coefficient $c\not \equiv \pm1, 0$ (mod $p$) and every fixed (even) $m\in \mathbb{Z}_{\geq 2}$. As before, let's for the sake of a contradiction, suppose $g(z) = \varphi_{(p-1)^{\ell},c}^m(z)-z\equiv 0$ (mod (mod $p$) for some $z\in \mathbb{Z}\slash p\mathbb{Z}$ and for every $c \equiv -1$ (mod $p$) and for every $\ell\in \mathbb{Z}_{\geq 1}$ and every fixed odd $m\in \mathbb{Z}_{\geq 3}$. So then, since $c\equiv -1$ (mod $p$) and also since $z^{(p-1)^{\ell}} = 1$ for every $z\in (\mathbb{Z}\slash p\mathbb{Z})^{\times}$, reducing $\varphi_{(p-1)^{\ell},c}^m(z)$ modulo $p$, it then follows $\varphi_{(p-1)^{\ell},c}^{m}(z)\equiv 0$ (mod $p$) for every fixed odd $m$; and so $g(z)= \varphi_{(p-1)^{\ell},c}^m(z)-z \equiv -z$ (mod $p$) for every point $z\in (\mathbb{Z}\slash p\mathbb{Z})^{\times}$. But now we note $z\equiv 0$ (mod $p$) is a root of $g(z)$ modulo $p$ for every $c\equiv -1$ (mod $p$) and every fixed odd $m$, and also $z\equiv 0$ (mod $p$) is a root of $g(z)$ modulo $p$ for every $c\equiv 0$ (mod $p$) and every fixed odd $m$, as seen earlier; and so follows a contradiction that $1 \equiv 0$ (mod $p$). Otherwise, suppose $g(z)\equiv 0$ (mod $p$) and so (from earlier) $z^{(p-1)^{m\ell}} + h(z)-z+c\equiv 0$ (mod $p$) for some $z\in (\mathbb{Z}\slash p\mathbb{Z})^{\times}$ and for every $c\not \equiv \pm1, 0$ (mod $p$) and every fixed (even) $m\in \mathbb{Z}_{\geq 2}$. So now, using $z^{(p-1)^{m\ell}}=1$ for every $z\in (\mathbb{Z}\slash p\mathbb{Z})^{\times}$ and every fixed $\ell\in \mathbb{Z}_{\geq 1}$ and (even) $m\in \mathbb{Z}_{\geq 2}$, we then rewrite $z^{(p-1)^{m\ell}} -z + h(z)+c\equiv 0$ (mod $p$) to obtain $(1-z) + (h(z)+c)\equiv 0$ (mod $p$). But now we note $(1-z) + (h(z)+c)\equiv 0$ (mod $p$) can also happen if $1-z\equiv 0$ (mod $p$) and also $h(z)+c\equiv 0$ (mod $p$). Moreover, recall from earlier $1-z\equiv 0$ (mod $p$) when $c\equiv 0$ (mod $p$); which then contradicts $c\not \equiv \pm1, 0$ (mod $p$). This then overall means $g(x)=\varphi_{(p-1)^{\ell},c}^m(x)-x$ has no roots in $(\mathbb{Z}\slash p\mathbb{Z})$ for every coefficient $c \equiv -1$ (mod $p$) and every fixed $\ell\in \mathbb{Z}_{\geq 1}$ and fixed odd $m\in \mathbb{Z}_{\geq 3}$ or for every coefficient $c\not \equiv \pm1, 0$ (mod $p$) and every fixed $\ell\in \mathbb{Z}_{\geq 1}$ and fixed (even) $m\in \mathbb{Z}_{\geq 2}$; and thus we conclude $M_{c}^{(m)}(p) = 0$. This then completes the whole proof, as needed.
\end{proof}

\begin{rem}\label{re3.5}
With Theorem \ref{3.3} at our disposal, we may then associate to each distinct $m$-periodic integral point of $\varphi_{(p-1)^{\ell},c}$ a $m$-periodic integral orbit. In doing so, we then obtain a dynamical translation of Theorem \ref{3.3}, namely, the claim that the number of distinct $n$-periodic orbits that any $\varphi_{(p-1)^{\ell},c}$ has when iterated on the space $\mathcal{O}_{K} / p\mathcal{O}_{K}$ is $1$ or $2$ or $0$. Again, as noted in Intro.\ref{sec1} that the count obtained in Theorem \ref{3.3} on the number of distinct $m$-periodic integral points of any $\varphi_{(p-1)^{\ell},c}$ modulo $p\mathcal{O}_{K}$ is independent of $p$ (and so independent of deg$(\varphi_{(p-1)^{\ell},c})$ for every fixed $\ell \in \mathbb{Z}_{\geq 1}$) and $n=[K: \mathbb{Q}]$ in each of the possibilities considered. Moreover, we may also observe that the expected total count of the number of distinct $m$-periodic integral points (orbits) in the whole family of maps $\varphi_{(p-1)^{\ell},c}$ modulo $p\mathcal{O}_{K}$ (namely, $1 + 2 + 0 =3$ for every fixed odd period $m\in \mathbb{Z}_{\geq 3}$ or namely, $1 + 1 + 2 + 0 =4$ for every fixed even period $m\in \mathbb{Z}_{\geq 2}$) is not only also independent of $p$ (and hence independent of deg$(\varphi_{(p-1)^{\ell},c})$ for every fixed $\ell \in \mathbb{Z}_{\geq 1}$), but is also a constant equal to $3$ or $4$ even when degree $(p-1)^{\ell}\to \infty$ or $n\to \infty$; a somewhat interesting phenomenon differing significantly from a phenomenon that we remarked in Remark \ref{re2.3}, however, coinciding somewhat surprising with a phenomenon remarked in \cite{BK2, BK11}.
\end{rem}

\begin{rem}
As before, recall in \cite{BK2} we proved that fixed point-counting function $M_{c}(p) = 1, 2$ or $0$ for every fixed $p\geq 5$ and  every $c\equiv 1, 0$ (mod $p\mathcal{O}_{K}$) or $c\equiv -1$ (mod $p\mathcal{O}_{K}$). Moreover, recall in the Proof of Theorem \ref{3.3} we proved that for every fixed odd (period) $m\in \mathbb{Z}_{\geq 3}$, the points $z\equiv 1, 0, 2$ (mod $p\mathcal{O}_{K}$) are $m$-periodic integral points of any $\varphi_{p-1,c}$ modulo $p$ (which we (in \cite{BK2}) also obtained as fixed integral points of any $\varphi_{(p-1)^{\ell},c}$ modulo $p\mathcal{O}_{K}$); from which we then concluded $M_{c}^{(m)}(p)  = 1, 2$ or $0$ for every fixed  $p\geq 5$ and every $c\equiv 1, 0$ (mod $p\mathcal{O}_{K}$) or $c\equiv -1$ (mod $p\mathcal{O}_{K}$). But now for every fixed odd (period) $m\in \mathbb{Z}_{\geq 3}$, we then note $M_{c}^{(m)}(p) = M_{c}(p) = 1, 2$ or $0$ for every fixed $p\geq 5$ and every $c\equiv 1, 0$ (mod $p\mathcal{O}_{K}$) or $c\equiv -1$ (mod $p\mathcal{O}_{K}$). Consequently, for every fixed odd (period) $m\in \mathbb{Z}_{\geq 1}$, we then note that every $m$-periodic integral orbit of $\varphi_{(p-1)^{\ell},c}$ modulo $p\mathcal{O}_{K}$ is a fixed integral orbit, and moreover every $\varphi_{(p-1)^{\ell},c}$ modulo $p\mathcal{O}_{K}$ has one or two or no fixed integral orbits; a somewhat interesting precise arithmetic-geometric insight on all odd $m$-periodic integral orbits of any $\varphi_{(p-1)^{\ell},c}$ modulo $p\mathcal{O}_{K}$. Furthermore, note that setting (period) $m=2$ in Theorem \ref{3.3}, we then have $M_{c}^{(2)}(p) = 1, 2$ or $0$ for every fixed $p$ and every $c\equiv \pm 1, 0$ (mod $p\mathcal{O}_{K}$) or $c\not \equiv \pm 1, 0$ (mod $p\mathcal{O}_{K}$). Moreover, for every fixed even (period) $m\in \mathbb{Z}_{\geq 4}$, we also found in the Proof of Theorem \ref{3.3} that the points $z\equiv 1, 0, 2, -1$ (mod $p\mathcal{O}_{K}$) are $m$-periodic integral points of any $\varphi_{(p-1)^{\ell},c}$ modulo $p\mathcal{O}_{K}$ (which also showed up as $2$-periodic integral points of any $\varphi_{(p-1)^{\ell},c}$ modulo $p\mathcal{O}_{K}$ in that same Proof of Theorem \ref{3.3}); from which we then concluded $M_{c}^{(m)}(p)  = 1, 2$ or $0$ for every fixed $p$ and every $c\equiv \pm 1, 0$ (mod $p\mathcal{O}_{K}$) or $c\not \equiv \pm 1, 0$ (mod $p\mathcal{O}_{K}$). But now for every fixed even (period) $m \in \mathbb{Z}_{\geq 4}$, we then also note $M_{c}^{(m)}(p) = M_{c}^{(2)}(p) = 1, 2$ or $0$ for every fixed $p$ and every $c\equiv \pm 1, 0$ (mod $p\mathcal{O}_{K}$) or $c\not \equiv \pm 1, 0$ (mod $p\mathcal{O}_{K}$). As before, for every fixed even (period) $m\in \mathbb{Z}_{\geq 2}$, we then also note that every $m$-periodic integral orbit of any $\varphi_{(p-1)^{\ell},c}$ modulo $p\mathcal{O}_{K}$ is a $2$-periodic integral orbit, and moreover every $\varphi_{(p-1)^{\ell},c}$ modulo $p\mathcal{O}_{K}$ has one or two or no $m$-periodic integral orbits; another somewhat interesting precise arithmetic-geometric insight on all even $m$-periodic integral orbits of every reduced map $\varphi_{(p-1)^{\ell},c}$ modulo $p\mathcal{O}_{K}$.    
\end{rem}

\section{Dynamical Complexity of Forward Periodic Orbit Structure of any $\overline{\varphi_{p^{\ell},c}}$ and $\overline{\varphi_{(p-1)^{\ell},c}}$}

Observe in Theorem \ref{2.3} that the number $N_{c}^{(m)}(p)$ is independent of (period) $m$, and moreover we may have 
\begin{center}
    $\lim\limits_{m\to \infty} N_{c}^{(m)}(p) = p$ or $0$.
\end{center}

As mentioned in \cite{BK11} that one of the main objectives in classical dynamical systems is to understand \textit{all} orbits usually via topological and analytic techniques. Moreover, in doing so, one may not only find that orbits can easily get very complicated but also important and interesting statistical questions (e.g., determining topological entropy) concerning measuring the complexity of the underlying system, may become intractable. (The interested reader may read about topological entropy in work of Adler \cite{Adl} and Bowen \cite{Ruf}). So now, inspired (as in \cite{BK11}) by such a statistic and since we may now also view $N_{c}^{(m)}(p)$ as $m$-periodic orbit-counting function, we in this section wish to investigate very mildly the complexity of a polynomial discrete dynamical system $(\mathcal{O}_{K}\slash p\mathcal{O}_{K}$, $\varphi_{p^{\ell},c}$ modulo $p\mathcal{O}_{K}$). With that in mind, we again analyze the behavior of the associated exponential growth rate $\rho(\overline{\varphi_{p^{\ell},c}})$ [\cite{Kat}, Page 196], where $\overline{\varphi_{p^{\ell},c}}:\mathcal{O}_{K}\slash p\mathcal{O}_{K} \to \mathcal{O}_{K}\slash p\mathcal{O}_{K}$ is the polynomial map $\varphi_{p^{\ell},c}$ modulo $p\mathcal{O}_{K}$; and in doing so we then obtain the following corollary showing that $m$-periodic orbit-counting function $N_{c}^{(m)}(p)$ grows by a factor $1=e^{\rho(\overline{\varphi_{p^{\ell},c}})}$ as period $m\to \infty$ and so the orbit-counting function $N_{c}^{(m)}(p)$ is constant: 
\begin{cor}\label{4.1}
Assume Theorem \ref{2.3}, and let $m\geq 2$ be any period. Then the exponential growth rate of $m$-periodic orbit-counting function $N_{c}^{(m)}(p)$ exists and is equal to zero. More precisely, we have 

\begin{center}
    $\rho(\overline{\varphi_{p^{\ell},c}}) := \limsup\limits_{m\to \infty}\frac{\textnormal{log}(\textnormal{max} \{N_{c}^{(m)}(p),1\})}{m} = 0$.
\end{center} 
\end{cor}
\begin{proof}
Since we know from Theorem \ref{2.3} that the number $N_{c}^{(m)}(p) = p\text{ or } 0$ for any fixed (period) $m\geq 2$, we then obtain $\frac{\text{log}(\text{max} \{N_{c}^{(m)}(p),1\})}{m} = \frac{\text{log }p}{m}$ or $0$. So now letting $m\to \infty$, we then obtain $\rho(\overline{\varphi_{p^{\ell},c}}) = 0$ as desired.
\end{proof}
Similarly, we may also observe in Theorem \ref{3.3} that $M_{c}^{(m)}(p)$ is independent of (period) $m$, and moreover  
\begin{center}
    $\lim\limits_{m\to \infty} M_{c}^{(m)}(p) = 1, 2 \text{ or } 0.$
\end{center}

So now, as before we may then also investigate very mildly the complexity of a polynomial discrete dynamical system $(\mathcal{O}_{K}\slash p\mathcal{O}_{K}$, $\varphi_{(p-1)^{\ell},c}$ modulo $p\mathcal{O}_{K}$) by again determining the behavior of the associated exponential growth rate $\rho(\overline{\varphi_{(p-1)^{\ell},c}})$, where $\overline{\varphi_{(p-1)^{\ell},c}}$ is the polynomial map $\varphi_{(p-1)^{\ell},c}$ modulo $p\mathcal{O}_{K}$. In doing so, we then immediately obtain the following corollary showing that $m$-periodic orbit-counting function $M_{c}^{(m)}(p)$ grows by a factor $1=e^{\rho(\overline{\varphi_{(p-1)^{\ell},c}})}$ as period $m\to \infty$ and so the orbit-counting function $M_{c}^{(m)}(p)$ is  also constant:
\begin{cor}
Assume Theorem \ref{3.3}, and let $m\geq 2$ be any period. Then the exponential growth rate of $m$-periodic orbit-counting function $M_{c}^{(m)}(p)$ exists and is equal to zero. More precisely, we have 
\begin{center}
    $\rho(\overline{\varphi_{(p-1)^{\ell},c}}) := \limsup\limits_{m\to \infty}\frac{\textnormal{log}(\textnormal{max} \{M_{c}^{(m)}(p),1\})}{m} = 0$.
\end{center}
\end{cor}
\begin{proof}
Since we know from Theorem \ref{3.3} that the number $M_{c}^{(m)}(p) = 1, 2 \text{ or } 0$ for any fixed (period) $m\geq 2$, we then obtain $\frac{\text{log}(\text{max} \{M_{c}^{(m)}(p),1\})}{m} = 0$ or $\frac{\text{log }2}{m}$. Now letting $m\to \infty$, it then follows $\rho(\overline{\varphi_{(p-1)^{\ell},c}}) = 0$ as desired.
\end{proof}

\section{The Average Number of $m$-Periodic Points of any Polynomial Map $\varphi_{p^{\ell},c}$ and $\varphi_{(p-1)^{\ell},c}$}\label{sec4}

In this section, we wish to inspect the behavior of the counting function $N_{c}^{(m)}(p)$ as $c$ tends to infinity. Specifically, we wish to determine: \say{\textit{For any fixed (period) $m\in \mathbb{Z}_{\geq 2}$, what is the average value of $N_{c}^{(m)}(p)$ as $c \to \infty$?}} The following corollary shows that the average value of the function $N_{c}^{(m)}(p)$ is zero or unbounded as $c\to \infty$:
\begin{cor}\label{c4.1}
Let $K\slash \mathbb{Q}$ be any number field of degree $n \geq 2$ with ring of integers $\mathcal{O}_{K}$, and in which any prime $p\geq 3$ is inert. Then the average value of $N_{c}^{(m)}(p)$ is zero or unbounded as $c\to\infty$. More precisely, we have
\begin{myitemize}
    \item[\textnormal{(a)}] \textnormal{Avg} $N^{(m)}_{c\neq pt}(p):= \lim\limits_{c \to\infty} \Large{\frac{\sum\limits_{3\leq p\leq c, \ p\nmid c \textnormal{ in } \mathcal{O}_{K}}N_{c}^{(m)}(p)}{\Large{\sum\limits_{3\leq p\leq c, \ p\nmid c \textnormal{ in } \mathcal{O}_{K}}1}}} =  0$. 
    
    \item[\textnormal{(b)}] \textnormal{Avg} $N^{(m)}_{c = pt}(p):= \lim\limits_{c \to\infty} \Large{\frac{\sum\limits_{3\leq p\leq c, \ p\mid c \textnormal{ in } \mathcal{O}_{K}}N_{c}^{(m)}(p)}{\Large{\sum\limits_{3\leq p\leq c, \ p\mid c \textnormal{ in } \mathcal{O}_{K}}1}}} =  \infty$.
\end{myitemize}

\end{cor}
\begin{proof}
Since we know from Theorem \ref{2.3} that the number $N_{c}^{(m)}(p) = 0$ for any inert $p\nmid c$ in $\mathcal{O}_{K}$, we then obtain $\lim\limits_{c\to\infty} \Large{\frac{\sum\limits_{3\leq p\leq c, \ p\nmid c \text{ in } \mathcal{O}_{K}}N_{c}^{(m)}(p)}{\Large{\sum\limits_{3\leq p\leq c, \ p\nmid c \text{ in } \mathcal{O}_{K}}1}}} = 0$; and so the average Avg$N_{c \neq pt}^{(m)}(p) = 0$. To see (b), we recall from Theorem \ref{2.3} that the number $N_{c}^{(m)}(p) = p$ for any inert $p\mid c$ in $\mathcal{O}_{K}$. But now we note  $\sum\limits_{3\leq p\leq c, \ p\mid c \text{ in } \mathcal{O}_{K}} N_{c}^{(m)}(p) = \sum\limits_{3\leq p\leq c, \ p\mid c \text{ in } \mathcal{O}_{K}}p =: \sigma_{1,p}(c)$ and  $\sum\limits_{3\leq p\leq c, \ p\mid c \text{ in } \mathcal{O}_{K}} 1  = \omega(c)$, where $\sigma_{1}(\ell)$ (resp. $\omega(\ell)$) is by definition the number of divisors (resp. the number of distinct prime divisors) of any $\ell\in \mathbb{Z}_{\geq 1}$; and so $\frac{\sum\limits_{3\leq p\leq c, \ p\mid c \text{ in } \mathcal{O}_{K}} N_{c}^{(m)}(p)}{\sum\limits_{3\leq p\leq c, \ p\mid c \text{ in } \mathcal{O}_{K}} 1} = \frac{\sigma_{1,p}(c)}{\omega(c)}$. So now, observe $\sigma_{1,p}(c)=\sum\limits_{3\leq p \leq c, \ p\mid c \text{ in } \mathcal{O}_{K}}p\leq \sum\limits_{3\leq p\leq c}p$ for every $c\in \mathbb{Z}_{\geq 3}$, and since we also know from \cite{Ma} that $\sum\limits_{3\leq p\leq c}p = \frac{c^2}{2}\biggl(\text{log }c + \text{log log }c - \frac{3}{2}+\frac{\text{log log }c - 5\slash 2}{\text{log }c}\biggl) + O\biggl(\frac{c^2(\text{log log }c)^2}{\text{log}^2c} \biggl)$ for every $c\in \mathbb{Z}_{\geq 3}$, it then follows that $\sigma_{1,p}(c)\leq \frac{c^2}{2}\biggl(\text{log }c + \text{log log }c - \frac{3}{2}+\frac{\text{log log }c - 5\slash 2}{\text{log }c}\biggl) + O\biggl(\frac{c^2(\text{log log }c)^2}{\text{log}^2c}\biggl)$. Now recall as a well-known fact that $\omega(c)\leq \frac{\text{log }c}{\text{log }2}$ for every $c\in \mathbb{Z}_{\geq 3}$. But then since $\frac{\text{log }c}{\text{log }2} \to \infty$ very slow as $c\to \infty$, then so is the number $\omega(c)\to \infty$ very slow as $c\to \infty$. So now, let's also observe that since $\frac{c^2}{2}\biggl(\text{log }c + \text{log log }c - \frac{3}{2}+\frac{\text{log log }c - 5\slash 2}{\text{log }c}\biggl) + O\biggl(\frac{c^2(\text{log log }c)^2}{\text{log}^2c}\biggl)\to \infty$ very fast as $c\to \infty$, it then follows that the number $\sigma_{1,p}(c) \to \infty$ very fast as $c\to \infty$. But now we overall conclude that the quotient $\frac{\sigma_{1,p}(c)}{\omega(c)} \to \infty$ very fast as $c\to \infty$ and so $\lim\limits_{c\to\infty} \Large{\frac{\sum\limits_{3\leq p\leq c, \ p\mid c \text{ in } \mathcal{O}_{K}}N_{c}^{(m)}(p)}{\Large{\sum\limits_{3\leq p\leq c, \ p\mid c \text{ in } \mathcal{O}_{K}}1}}}= \infty$; from which we then conclude the average value Avg $N_{c= pt}^{(m)}(p) = \infty$. This then completes the whole proof, as required.  
\end{proof}
\begin{rem} \label{4.2}
From arithmetic statistics to arithmetic dynamics, we note that Corollary \ref{c4.1} shows that any $\varphi_{p^{\ell},c}$ iterated on the space $\mathcal{O}_{K} / p\mathcal{O}_{K}$ has on average zero or an unbounded number of distinct $m$-periodic integral orbits as $c\to \infty$; a somewhat interesting averaging phenomenon coinciding with a phenomenon in [\cite{BK3}, Remark 7.2] on the average number of distinct fixed integral orbits of every $\varphi_{p^{\ell},c}$ (for any $\ell \in \{1, p\}$) iterated $\mathcal{O}_{K} / p \mathcal{O}_{K}$.
\end{rem}

Similarly, we also wish to determine: \say{\textit{For any fixed (period) $m\in \mathbb{Z}_{\geq 2}$, what is the average value of $M_{c}^{(m)}(p)$ as $c \to \infty$?}} The following corollary shows that the average value of $M_{c}^{(m)}(p)$ is $1$ or $2$ or $0$ as $c\to \infty$:
\begin{cor}\label{4.3}
Let $K\slash \mathbb{Q}$ be a number field of degree $n\geq 2$ in which any prime $p\geq 5$ is inert in $\mathcal{O}_{K}$. Then the average value of $M_{c}^{(m)}(p)$ exists and is equal to $1$ or $2$ or $0$ as $c\to\infty$. More precisely, we have 
\begin{myitemize}
    \item[\textnormal{(a)}] \textnormal{Avg} $M_{c\pm1 = pt}^{(m)}(p) := \lim\limits_{c\to\infty} \Large{\frac{\sum\limits_{5\leq p\leq (c\pm1), \ p\mid (c\pm1) \textnormal{ in } \mathcal{O}_{K}}M_{c}^{(m)}(p)}{\Large{\sum\limits_{5\leq p\leq (c\pm 1), \ p\mid (c\pm1) \textnormal{ in } \mathcal{O}_{K}}1}}} = 1.$ 

    \item[\textnormal{(b)}] \textnormal{Avg} $M_{c= pt}^{(m)}(p) := \lim\limits_{c\to\infty} \Large{\frac{\sum\limits_{5\leq p\leq c, \ p\mid c \textnormal{ in } \mathcal{O}_{K}}M_{c}^{(m)}(p)}{\Large{\sum\limits_{5\leq p\leq c, \ p\mid c \textnormal{ in } \mathcal{O}_{K}}1}}} = 2.$
    
     \item[\textnormal{(c)}] \textnormal{Avg} $M_{c\not \equiv\pm1, 0 \ (\textnormal {mod }p)}^{(m)}(p):= \lim\limits_{c \to\infty} \Large{\frac{\sum\limits_{5\leq p\leq c, \ c\not \equiv\pm1, 0 \ (\textnormal{mod } p\mathcal{O}_{K})}M_{c}^{(m)}(p)}{\Large{\sum\limits_{5\leq p\leq c, \ c\not \equiv\pm1, 0 \ (\textnormal{mod } p\mathcal{O}_{K})}1}}} =  0$.
\end{myitemize}

\end{cor}
\begin{proof}

By applying a similar argument as in the Proof of Corollary \ref{c4.1}, we then obtain the limits as desired.
\end{proof} 
\begin{rem} \label{4.4}
As before, we also note that from arithmetic statistics to arithmetic dynamics, Corollary \ref{4.3} shows that any $\varphi_{(p-1)^{\ell},c}$ iterated on $\mathcal{O}_{K} / p \mathcal{O}_{K}$ has on average one or two or no $m$-periodic orbits as $c\to \infty$; a somewhat interesting averaging phenomenon coinciding precisely with an averaging phenomenon remarked in [\cite{BK2}, Remark 4.4] on the average number of distinct fixed integral orbits of every $\varphi_{(p-1)^{\ell},c}$ iterated on $\mathcal{O}_{K} / p \mathcal{O}_{K}$.
\end{rem}

\section{On the Density of Integer Polynomials $\varphi_{p^{\ell},c}(x)\in \mathcal{O}_{K}[x]$ with the Number $N_{c}^{(m)}(p) = p$}\label{sec5}

As in [\cite{BK3}, Section 9] we in this and the next section, wish to restrict our counting on the subring $\mathbb{Z}\subset \mathcal{O}_{K}$ and then determine: \say{\textit{For any prime $p\geq 3$ and for any fixed $\ell\in \mathbb{Z}_{\geq 1}$ and fixed period $m\in \mathbb{Z}_{\geq 2}$, what is the density of monic integer polynomials $\varphi_{p^{\ell},c}(x) = x^{p^{\ell}} + c\in \mathcal{O}_{K}[x]$ with exactly $p$ distinct $m$-periodic integral points modulo $p$?}} The following corollary shows that for any fixed $\ell\in \mathbb{Z}_{\geq 1}$ and fixed period $m\in \mathbb{Z}_{\geq 2}$, there are very few integer polynomials $\varphi_{p^{\ell},c}(x)=x^{p^{\ell}} + c\in \mathbb{Z}[x]$ with $p$ distinct $m$-periodic integral points modulo $p$:

\begin{cor}\label{5.1}
Let $K\slash \mathbb{Q}$ be any number field of degree $n\geq 2$ with the ring of integers $\mathcal{O}_{K}$, and in which any prime $p\geq 3$ is inert. Let $\ell \geq 1$ be any fixed integer. Then the density of monic integer polynomials $\varphi_{p^{\ell},c}(x) = x^{p^{\ell}} + c\in \mathcal{O}_{K}[x]$ with $N_{c}^{(m)}(p) = p$ exists and is equal to $0 \%$ as $c\to \infty$. More precisely, we have 
\begin{center}
    $\lim\limits_{c\to\infty} \Large{\frac{\# \{\varphi_{p^{\ell},c}(x)\in \mathbb{Z}[x] \ : \ 3\leq p\leq c \ \textnormal{and} \ N_{c}^{(m)}(p) \ = \ p\}}{\Large{\# \{\varphi_{p^{\ell},c}(x) \in \mathbb{Z}[x] \ : \ 3\leq p\leq c \}}}} = \ 0.$
\end{center}
\end{cor}
\begin{proof}
Since the defining condition $N_{c}^{(m)}(p) = p$ is as we proved in Theorem \ref{2.3} and hence in Cor. \ref{cor2.4} determined whenever the coefficient $c$ is divisible by $p$, we may then count $\# \{\varphi_{p^{\ell},c}(x) \in \mathbb{Z}[x] : 3\leq p\leq c \ \text{and} \ N_{c}^{(m)}(p) \ = \ p\}$ by counting $\# \{\varphi_{p^{\ell},c}(x)\in \mathbb{Z}[x] : 3\leq p\leq c \ \text{and} \ p\mid c \ \text{for \ any \ fixed} \ c \}$. In that case, we then write the quotient
\begin{center}
$\Large{\frac{\# \{\varphi_{p^{\ell},c}(x) \in \mathbb{Z}[x] \ : \ 3\leq p\leq c \ \text{and} \ N_{c}^{(m)}(p) \ = \ p\}}{\Large{\# \{\varphi_{p^{\ell},c}(x) \in \mathbb{Z}[x] \ : \ 3\leq p\leq c \}}}} = \Large{\frac{\# \{\varphi_{p^{\ell},c}(x)\in \mathbb{Z}[x] \ : \ 3\leq p\leq c \ \text{and} \ p\mid c \ \text{for any fixed} \ c \}}{\Large{\# \{\varphi_{p^{\ell},c}(x) \in \mathbb{Z}[x] \ : \ 3\leq p\leq c \}}}}$. 
\end{center}\indent Moreover, for any fixed integer $c\geq 3$, the numerator of the foregoing quotient may be rewritten to then obtain
\begin{center}
$\# \{\varphi_{p^{\ell},c}(x) \in \mathbb{Z}[x] : 3\leq p\leq c \ \text{and} \ p\mid c \} = \# \{p : 3\leq p\leq c \text{ and } p\mid c \} = \sum_{3\leq p\leq c, \ p\mid c}1 = \omega (c)$, 
\end{center}where $\omega(\ell)$ is by definition the number of distinct prime factors of $\ell$. Writing $\# \{\varphi_{p^{\ell},c}(x) \in \mathbb{Z}[x]  : 3\leq p\leq c \} = \sum_{3\leq p\leq c} 1 = \pi(c)$, where $\pi(\ell)$ is by definition the number of primes at most $\ell$, we then note that the quotient 
\begin{center}
$\Large{\frac{\# \{\varphi_{p^{\ell},c}(x)\in \mathbb{Z}[x] \ : \ 3\leq p\leq c \ \text{and} \ p\mid c \ \text{for any fixed} \ c \}}{\Large{\# \{\varphi_{p^{\ell},c}(x)\in \mathbb{Z}[x] \ : \ 3\leq p\leq c \}}}} = \frac{\omega(c)}{\pi(c)}$.
\end{center}So now, recall (from a well-known fact) that for any $c\in \mathbb{Z}_{\geq 3}$, we have $2^{\omega(c)}\leq \sigma (c) \leq 2^{\Omega(c)}$, where $\sigma(\ell)$ is by definition the divisor function and $\Omega(\ell)$ is by definition the total number of prime factors of $\ell$, with respect to their multiplicity. Note that taking logarithms, we then obtain $\omega(c)\leq \frac{\text{log} \ \sigma(c)}{\text{log} \ 2}$; and so $\frac{\omega(c)}{\pi(c)} \leq \frac{\text{log} \ \sigma(c)}{\text{log} \ 2 \cdot \pi(c)}$. Moreover, for every $\epsilon >0$, it is well-known that $\sigma(c) = o(c^{\epsilon})$; and so log $\sigma(c) =$ log $o(c^{\epsilon})$ and so have $\frac{\omega(c)}{\pi(c)} \leq \frac{\text{log} \ o(c^{\epsilon})}{\text{log} \ 2 \cdot \pi(c)}$. Now for every fixed $\epsilon>0$, we then note $\lim\limits_{c\to\infty} \frac{\text{log} \ o(c^{\epsilon})}{\text{log} \ 2 \cdot \pi(c)} = 0$ and so $\lim\limits_{c\to\infty} \frac{\omega(c)}{\pi(c)} \leq 0$. But now we note that the limit 
\begin{center}
$\lim\limits_{c\to\infty} \Large{\frac{\# \{\varphi_{p^{\ell},c}(x)\in \mathbb{Z}[x] \ : \ 3\leq p\leq c \ \text{and} \ N_{c}^{(m)}(p) \ = \ p\}}{\Large{\# \{\varphi_{p^{\ell},c}(x) \in \mathbb{Z}[x] \ : \ 3\leq p\leq c \}}}} =\lim\limits_{c\to\infty} \frac{\omega(c)}{\pi(c)} \leq 0$.
\end{center}Moreover, we also observe that the number $\# \{\varphi_{p^{\ell},c}(x)\in \mathbb{Z}[x] : 3\leq p\leq c \ \text{and} \ N_{c}^{(m)}(p) \ = \ p\}\geq 1$, and so have 
\begin{center}
$\lim\limits_{c\to\infty}\Large{\frac{\# \{\varphi_{p^{\ell},c}(x) \in \mathbb{Z}[x] \ : \ 3\leq p\leq c \ \text{and} \ N_{c}^{(m)}(p) \ = \ p\}}{\Large{\# \{\varphi_{p^{\ell},c}(x) \in \mathbb{Z}[x] \ : \ 3\leq p\leq c \}}}}\geq \lim\limits_{c\to\infty}\frac{1}{\pi(c)} = 0$. But now, we then conclude that the limit 
\end{center}  $\lim\limits_{c\to\infty} \Large{\frac{\# \{\varphi_{p^{\ell},c}(x) \in \mathbb{Z}[x] \ : \ 3\leq p\leq c \ \text{and} \ N_{c}^{(m)}(p) \ = \ p\}}{\Large{\# \{\varphi_{p^{\ell},c}(x) \in \mathbb{Z}[x] \ : \ 3\leq p\leq c \}}}} = 0$ as needed. This completes the whole proof, as desired.
\end{proof}\noindent Note that one may also interpret Corollary \ref{5.1} as saying that for any fixed $\ell \in \mathbb{Z}_{\geq 1}$ and fixed period $m\in \mathbb{Z}_{\geq 2}$, the probability of choosing randomly $\varphi_{p^{\ell},c}(x)\in \mathbb{Z}[x]\subset \mathcal{O}_{K}[x]$ with $p$ distinct $m$-periodic integral points modulo $p$ is zero; a somewhat interesting probabilistic phenomenon coinciding with a phenomenon remarked in [\cite{BK3}, Corollary 9.1]  on the probability of choosing randomly $\varphi_{p^{\ell},c}(x)\in \mathbb{Z}[x]$ having $p$ fixed integral points modulo $p$.

\section{The Densities of Monic Integer Polynomials $\varphi_{(p-1)^{\ell},c}(x)\in \mathcal{O}_{K}[x]$ with $M_{c}^{(m)}(p) = 1$ or $2$}\label{sec6}

As in Section \ref{sec5}, we also wish to determine: \say{\textit{For any prime $p\geq 5$ and for any fixed $\ell \in \mathbb{Z}_{\geq 1}$ and period $m\in \mathbb{Z}_{\geq 2}$, what is the density of monic integer polynomials $\varphi_{(p-1)^{\ell},c}(x) = x^{(p-1)^{\ell}} + c\in \mathcal{O}_{K}[x]$ with two distinct $m$-periodic integral points modulo $p$?}} The following corollary shows that for any fixed $\ell \in \mathbb{Z}_{\geq 1}$ and $m\in \mathbb{Z}_{\geq 2}$, there are also very few monic polynomials $\varphi_{(p-1)^{\ell},c}(x)\in \mathbb{Z}[x]\subset \mathcal{O}_{K}[x]$ with two distinct $m$-periodic integral points modulo $p$:
\begin{cor}\label{6.1}
Let $K\slash \mathbb{Q}$ be any number field of degree $n\geq 2$ with the ring of integers $\mathcal{O}_{K}$, and in which any prime $p\geq 5$ is inert. Let $\ell \geq 1$ be any fixed integer. Then the density of monic integer polynomials $\varphi_{(p-1)^{\ell},c}(x) = x^{(p-1)^{\ell}} + c\in \mathcal{O}_{K}[x]$ with $M_{c}^{(m)}(p) = 2$ exists and is equal to $0 \%$ as $c\to \infty$. Specifically, we have 
\begin{center}
    $\lim\limits_{c\to\infty} \Large{\frac{\# \{\varphi_{(p-1)^{\ell},c}(x) \in \mathbb{Z}[x]\ : \ 5\leq p\leq c \ \textnormal{and} \ M_{c}^{(m)}(p) \ = \ 2\}}{\Large{\# \{\varphi_{(p-1)^{\ell},c}(x) \in \mathbb{Z}[x]\ : \ 5\leq p\leq c \}}}} = \ 0.$
\end{center}
\end{cor}
\begin{proof}
Since the condition $M_{c}^{(m)}(p) = 2$ is as we proved in Theorem \ref{3.3} and hence in Corollary \ref{cor3.4} determined whenever $c$ is divisible by $p$, we may then count $\# \{\varphi_{(p-1)^{\ell},c}(x) \in \mathbb{Z}[x] : 5\leq p\leq c \ \text{and} \ M_{c}^{(m)}(p) \ = \ 2\}$ by simply counting the number $\# \{\varphi_{(p-1)^{\ell},c}(x)\in \mathbb{Z}[x] : 5\leq p\leq c \ \text{and} \ p\mid c \ \text{for \ any \ fixed} \ c \}$. But now applying a similar argument as in the Proof of Corollary \ref{5.1}, we then obtain that the limit exits and is equal to $0$, as desired.
\end{proof} \noindent As before, we may also interpret Corollary \ref{6.1} as saying that for any fixed $\ell \in \mathbb{Z}_{\geq 1}$ and fixed period $m\in \mathbb{Z}_{\geq 2}$, the probability of choosing randomly $\varphi_{(p-1)^{\ell},c}(x)\in \mathbb{Z}[x]\subset \mathcal{O}_{K}[x]$ with two distinct $m$-periodic integral points modulo $p$ is zero; a somewhat interesting probabilistic phenomenon coinciding with [\cite{BK2}, Corollary 6.1] on the probability of choosing randomly $\varphi_{(p-1)^{\ell},c}(x)\in \mathbb{Z}[x]\subset \mathcal{O}_{K}[x]$ with two distinct fixed integral points modulo $p$.

The following corollary shows that for any fixed $\ell \in \mathbb{Z}_{\geq 1}$ and fixed period $m\in \mathbb{Z}_{\geq 2}$, the probability of choosing randomly a monic integer polynomial $\varphi_{(p-1)^{\ell},c}(x)=x^{(p-1)^{\ell}}+c\in \mathbb{Z}[x]\subset \mathcal{O}_{K}[x]$ with exactly one $m$-periodic point modulo $p$ is also zero; and thus also coinciding with [\cite{BK2}, Corollary 6.2] on the probability of choosing randomly a polynomial $\varphi_{(p-1)^{\ell},c}(x)\in \mathbb{Z}[x]\subset \mathcal{O}_{K}[x]$ having exactly one fixed integral point modulo $p$:

\begin{cor}\label{6.2}
Let $K\slash \mathbb{Q}$ be any number field of degree $n\geq 2$ with the ring of integers $\mathcal{O}_{K}$, and in which any prime $p\geq 5$ is inert. Let $\ell \geq 1$ be any fixed integer. The density of monic integer polynomials $\varphi_{(p-1)^{\ell},c}(x) = x^{(p-1)^{\ell}} + c\in \mathcal{O}_{K}[x]$ with $M_{c}^{(m)}(p) = 1$ exists and is equal to $0 \%$ as $c\to \infty$. More precisely, we have  
\begin{center}
    $\lim\limits_{c\to\infty} \Large{\frac{\# \{\varphi_{(p-1)^{\ell},c}(x) \in \mathbb{Z}[x]\ : \ 5\leq p\leq c \ \textnormal{and} \ M_{c}^{(m)}(p) \ = \ 1\}}{\Large{\# \{\varphi_{(p-1)^{\ell},c}(x) \in \mathbb{Z}[x]\ : \ 5\leq p\leq c \}}}} = \ 0.$
\end{center}
\end{cor}
\begin{proof}
As before, $M_{c}^{(m)}(p) = 1$ is as we proved in Corollary \ref{cor3.4} determined whenever the coefficient $c$ is such that $c\pm1$ is divisible by a prime $p\geq 5$; and so we may count $\# \{\varphi_{(p-1)^{\ell},c}(x) \in \mathbb{Z}[x] : 5\leq p\leq c \ \text{and} \ M_{c}^{(m)}(p) \ = \ 1\}$ by counting $\# \{\varphi_{(p-1)^{\ell},c}(x)\in \mathbb{Z}[x] : 5\leq p\leq c \ \text{and} \ p\mid (c\pm1) \ \text{for \ any \ fixed} \ c \}$. But now applying a very similar argument as in [\cite{BK11}, Proof of Corollary 6.2], we then obtain that the limit exits and is equal to $0$, as desired.  
\end{proof}

\section{The Density of Polynomials $\varphi_{p^{\ell},c}(x)$ with $N_{c}^{(m)}(p) = 0$ and $\varphi_{(p-1)^{\ell},c}(x)$ with $M_{c}^{(m)}(p) = 0$}\label{sec7}
\noindent Recall in Corollary \ref{5.1} that a density of $0\%$ of integer polynomials $\varphi_{p^{\ell},c}(x)\in\mathcal{O}_{K}[x]$ have number $N_{c}^{(m)}(p) = p$; and so for any fixed period $m\in \mathbb{Z}_{\geq 2}$, the density of integer polynomials $\varphi_{p^{\ell},c}^{m}(x)-x\in \mathcal{O}_{K}[x]$ that are reducible modulo $p$ is $0\%$. So now, we also wish to determine: \say{\textit{For any prime $p\geq 3$ and for any fixed $\ell \in \mathbb{Z}_{\geq 1}$ and $m\in \mathbb{Z}_{\geq 2}$, what is the density of integer polynomials $\varphi_{p^{\ell},c}(x)\in \mathcal{O}_{K}[x]$ with no $m$-periodic integral points modulo $p$?}} The following corollary shows that for any fixed $\ell \in \mathbb{Z}_{\geq 1}$ and $m\in \mathbb{Z}_{\geq 2}$, the probability of choosing randomly a monic integer polynomial $\varphi_{p^{\ell},c}(x)\in \mathbb{Z}[x]$ such that $\mathbb{Q}[x]\slash (\varphi^{m}_{p^{\ell}, c}(x)-x)$ is a number field of degree $p^{m\ell}$ is one: 
\begin{cor}\label{7.1}
Let $K\slash \mathbb{Q}$ be any number field of degree $n\geq 2$ with the ring of integers $\mathcal{O}_{K}$, and in which any prime $p\geq 3$ is inert. Let $\ell \geq 1$ and $m\geq 2$ be any fixed integers. Then the density of monic integer polynomials $\varphi_{p^{\ell},c}(x)=x^{p^{\ell}} + c\in \mathcal{O}_{K}[x]$ with $N_{c}^{(m)}(p) = 0$ exists and is equal to $100 \%$ as $c\to \infty$. More precisely, we have 
\begin{center}
    $\lim\limits_{c\to\infty} \Large{\frac{\# \{\varphi_{p^{\ell},c}(x)\in \mathbb{Z}[x] \ : \ 3\leq p\leq c \ \textnormal{and} \ N_{c}^{(m)}(p) \ = \ 0 \}}{\Large{\# \{\varphi_{p^{\ell},c}(x) \in \mathbb{Z}[x] \ : \ 3\leq p\leq c \}}}} = \ 1.$
\end{center}
\end{cor}
\begin{proof}
Since $N_{c}^{(m)}(p) = p$ or $0$ for any inert $p\geq 3$, fixed $\ell \in \mathbb{Z}_{\geq 1}$ and $m \in \mathbb{Z}_{\geq 2}$ and also because of the density proved in Corollary \ref{5.1}, we then obtain the density as desired (i.e., we obtain that the limit is equal to 1). 
\end{proof}

\noindent Note that the foregoing corollary also shows that for any fixed $\ell \in \mathbb{Z}_{\geq 1}$ and fixed $m\in \mathbb{Z}_{\geq 2}$, there are infinitely many polynomials $\varphi_{p^{\ell},c}(x)\in \mathbb{Z}[x]\subset \mathbb{Q}[x]$ such that for $f(x) = \varphi_{p^{\ell},c}^m(x)-x$, the quotient $\mathbb{Q}_{f} = \mathbb{Q}[x]\slash (f(x))$ induced by $f$ is an algebraic number field of odd degree $\kappa=p^{m\ell}$. Comparing the densities in Corollaries \ref{5.1} and \ref{7.1}, we may then observe that in the whole family of monic integer polynomials $\varphi_{p^{\ell},c}(x) = x^{p^{\ell}} +c$, almost all such monic integer polynomials $\varphi_{p^{\ell},c}$ have no $m$-periodic integral points modulo $p$; and from which it then also follows that almost all monic integer polynomials $f(x)$ are irreducible over $\mathbb{Q}$. Consequently, this may then imply that the average value of the number $N_{c}^{(m)}(p)$ in the whole family of monic polynomials $\varphi_{p^{\ell},c}(x)$ is zero.

Similarly, recall in Corollary \ref{6.1} or \ref{6.2} that a density of $0\%$ of monic integer polynomials $\varphi_{(p-1)^{\ell},c}(x)$ have $M_{c}^{(m)}(p) = 2$ or $1$, resp.; and so for any fixed period $m\in \mathbb{Z}_{\geq 2}$, the density of $\varphi_{(p-1)^{\ell},c}^m(x)-x\in \mathbb{Z}[x]$ that are reducible modulo $p$ is $0\%$. So now, we also wish to determine: \say{\textit{For any prime $p\geq 5$ and for any fixed $\ell \in \mathbb{Z}_{\geq 1}$ and $m\in \mathbb{Z}_{\geq 2}$, what is the density of monic polynomials $\varphi_{(p-1)^{\ell},c}(x)\in \mathbb{Z}[x]$ with no $m$-periodic integral points (mod $p$)?}} The corollary below shows that for any fixed $\ell$ and $m$, the probability of choosing randomly a polynomial $\varphi_{(p-1)^{\ell},c}(x)\in \mathbb{Z}[x]$ so that $\mathbb{Q}[x]\slash (\varphi^{m}_{(p-1)^{\ell}, c}(x)-x)$ is a number field of degree $(p-1)^{m\ell}$ is also 1:
\begin{cor} \label{7.2}
Let $K\slash \mathbb{Q}$ be any number field of degree $n\geq 2$ with the ring of integers $\mathcal{O}_{K}$, and in which any prime $p\geq 5$ is inert. Let $\ell \geq 1$ and $m\geq 2$ be any fixed integers. Then the density of monic integer polynomials $\varphi_{(p-1)^{\ell}, c}(x) = x^{(p-1)^{\ell}}+c\in \mathcal{O}_{K}[x]$ with $M_{c}^{(m)}(p) = 0$ exists and is equal to $100 \%$ as $c\to \infty$. That is, we have 
\begin{center}
    $\lim\limits_{c\to\infty} \Large{\frac{\# \{\varphi_{(p-1)^{\ell}, c}(x)\in \mathbb{Z}[x] \ : \ 5\leq p\leq c \ \textnormal{and} \ M_{c}^{(m)}(p) \ = \ 0 \}}{\Large{\# \{\varphi_{(p-1)^{\ell},c}(x) \in \mathbb{Z}[x] \ : \ 5\leq p\leq c \}}}} = \ 1.$
\end{center}
\end{cor}
\begin{proof}
Because $M_{c}^{(m)}(p) = 1$ or $2$ or $0$ for any inert $p\geq 5$, fixed $\ell \in \mathbb{Z}_{\geq 1}$ and $m\in \mathbb{Z}_{\geq 2}$ and also because of the densities proved in Cor. \ref{6.1} and \ref{6.2}, we then obtain the density (i.e., we obtain that the limit is equal to $1$).
\end{proof}
\noindent As before, Corollary \ref{7.2} also shows that for any fixed $\ell \in \mathbb{Z}_{\geq 1}$ and any fixed $m\in \mathbb{Z}_{\geq 2}$, there are infinitely many monic polynomials $\varphi_{(p-1)^{\ell},c}(x)\in \mathbb{Z}[x]\subset \mathbb{Q}[x]$ such that for $g(x) = \varphi^{m}_{(p-1)^{\ell},c}(x)-x$, the quotient ring $\mathbb{Q}_{g} = \mathbb{Q}[x]\slash (g(x))$ induced by $g$ is an algebraic number field of even degree $\upsilon=(p-1)^{m\ell}$. Again, if we compare the densities in Corollary \ref{6.1}, \ref{6.2} and \ref{7.2}, we then also observe that in the whole family of polynomials $\varphi_{(p-1)^{\ell},c}(x) = x^{(p-1)^{\ell}} +c\in \mathbb{Z}[x]$, almost all such polynomials $\varphi_{(p-1)^{\ell},c}(x)$ have no $m$-periodic integral points modulo $p$; from which it then also follows that almost all monic polynomials $g\in \mathbb{Z}[x]$ are irreducible over $\mathbb{Q}$. This may also imply the average value of $M_{c}^{(m)}(p)$ in the whole family of monics $\varphi_{(p-1)^{\ell},c}(x)\in \mathbb{Z}[x]$ is also zero.

Recall more generally from algebraic number theory that any number field $K$ is always naturally equipped with a ring $\mathcal{O}_{K}$ of integers in $K$; and which is classically known to describe the arithmetic of $K$, but usually very difficult to compute in practice. So now, it then follows $\mathbb{Q}_{f}$ has a ring of integers $\mathcal{O}_{\mathbb{Q}_{f}}$ and moreover applying [\cite{sch1}, Theorem 1.2], we then also have the following corollary showing that the probability of choosing randomly an irreducible monic integer polynomial $f\in \mathbb{Z}[x]$ arising from a polynomial discrete dynamical system in Section \ref{sec2} (and ascertained by Corollary \ref{7.1}), such that $\mathbb{Z}_{f}=\mathbb{Z}[x]\slash (f(x))$ is the ring of integers of $\mathbb{Q}_{f}$, is $\approx 60.7927\%$:

\begin{cor} \label{7.3}
Assume Corollary \ref{7.1}. When monic polynomials $f(x)\in \mathbb{Z}[x]$ are ordered by height $H(f)$ as defined in \textnormal{\cite{sch1}}, the density of polynomials $f(x)$ such that $\mathbb{Z}_{f}=\mathbb{Z}[x]\slash (f(x))$ is the ring of integers of $\mathbb{Q}_{f}$ is $\zeta(2)^{-1}$. 
\end{cor}

\begin{proof}
From Corollary \ref{7.1}, there are infinitely many irreducible monic polynomials $f(x)=\varphi_{p^{\ell},c}^m(x)-x\in \mathbb{Z}[x]$ such that $\mathbb{Q}_{f} = \mathbb{Q}[x]\slash (f(x))$ is an algebraic number field of deg$(f) = p^{m\ell}$; and moreover associated to each $\mathbb{Q}_{f}$ is the ring of integers $\mathcal{O}_{\mathbb{Q}_{f}}$. This then means that the family of monic polynomials $f\in \mathbb{Z}[x]$ such that $\mathbb{Q}_{f}$ is an algebraic number field of degree $\kappa=p^{m\ell}$ is not empty. But now applying [\cite{sch1}, Theorem 1.2] on the underlying family of such polynomials $f\in \mathbb{Z}[x]$ ordered by height $H(f)$ as defined in \cite{sch1} such that $\mathcal{O}_{\mathbb{Q}_{f}} = \mathbb{Z}[x]\slash (f(x))$, it then follows that the density of such monic polynomials $f\in \mathbb{Z}[x]$ is equal to $\zeta(2)^{-1} \approx 60.7927\%$, as required.
\end{proof}

Similarly, we note that every $\mathbb{Q}_{g}$ induced by a polynomial $g$, is also naturally equipped with the ring of integers $\mathcal{O}_{\mathbb{Q}_{g}}$, and which may also be difficult to compute in practice. So now, we note that by taking great advantage of [\cite{sch1}, Theorem 1.2], we then also obtain the following corollary showing the probability of choosing randomly an irreducible polynomial $g\in \mathbb{Z}[x]$ arising from a polynomial discrete dynamical system in Section \ref{sec3} (and ascertained by Corollary \ref{7.2}), such that $\mathbb{Z}_{g}=\mathbb{Z}[x]\slash (g(x))$ is the ring of integers of $\mathbb{Q}_{g}$ is also $\approx 60.7927\%$:

\begin{cor}
Assume Corollary \ref{7.2}. When monic polynomials $g(x)\in \mathbb{Z}[x]$ are ordered by height $H(g)$ as defined in \textnormal{\cite{sch1}}, the density of polynomials $g(x)$ such that $\mathbb{Z}_{g}=\mathbb{Z}[x]\slash (g(x))$ is the ring of integers of $\mathbb{Q}_{g}$ is $\zeta(2)^{-1}$. 
\end{cor}

\begin{proof}
Applying a similar argument as in the Proof of Corollary \ref{7.3}, we then obtain the density, as required.
\end{proof}

\section{On Artin-Mazur Zeta Function of $\overline{\varphi_{p^{\ell},c}}$ with $N_{c}^{(m)}(p) = 0$ and $\overline{\varphi_{(p-1)^{\ell},c}}$ with $M_{c}^{(m)}(p) = 0$}

Recall from Artin-Mazur [\cite{AM}, Page 84] that for any given topological space $X$ and for any given map $f:X\to X$, we can attach to $f$ a zeta function $\zeta_{f}$ defined in the following way, that for every complex number $s\in \mathbb{C}$, we set  
\begin{equation}\label{eqAM}
    \zeta_{f}(s)=\textnormal{exp}\biggl(\sum_{m=1}^{\infty} \frac{N_{m}(f)\cdot s^m}{m}\biggr)
\end{equation} where $N_{m}(f)$ is the number of isolated periodic points of $f$, of period $m$. So now, inspired by the definition and work of Artin-Mazur \cite{AM} on their zeta function $\zeta_{f}(s)$, we in this section wish to define (in a similar way) and then determine the Artin-Mazur zeta functions arising from a polynomial discrete dynamical system in Section \ref{sec2} and \ref{sec3}. With that in mind, we note that since $\mathcal{O}_{K}\slash p\mathcal{O}_{K}$ is a finite set of $p^n$ elements and so $\mathcal{O}_{K}\slash p\mathcal{O}_{K}$ can be equipped with a topology, then denoting the reduced polynomial map $\varphi_{p^{\ell},c}$ modulo $p\mathcal{O}_{K}$ by $\overline{\varphi_{p^{\ell},c}}$ and then also replacing $N_{m}(f)$ with $N_{c}^{(m)}(p)$ in Eqn (\ref{eqAM}), we then let $\zeta_{\overline{\varphi_{p^{\ell},c}}}$ be the Artin-Mazur zeta function induced by the map $\overline{\varphi_{p^{\ell},c}}: \mathcal{O}_{K}\slash p\mathcal{O}_{K} \to \mathcal{O}_{K}\slash p\mathcal{O}_{K}$. Inspired by \cite{AM} on their zeta function $\zeta_{f}$, we in our case then have the following corollary showing that the Artin-Mazur zeta functions $\zeta_{\overline{\varphi_{p^{\ell},c}}}$ associated with infinitely many polynomial maps $\overline{\varphi_{p^{\ell},c}}$ are constant functions on the whole complex plane $\mathbb{C}$ and hence algebraic functions of $s$:  

\begin{cor}\label{cAM}
Assume Corollary \ref{7.1} or second part of Theorem \ref{2.3}, and denote the polynomial map $\varphi_{p^{\ell},c}$ modulo prime $p\mathcal{O}_{K}$ by $\overline{\varphi_{p^{\ell},c}}$. Then the Artin-Mazur zeta function $\zeta_{\overline{\varphi_{p^{\ell},c}}}(s)=1$ for every complex number $s\in \mathbb{C}$.
\end{cor}

\begin{proof}
To see this, we note that from Corollary \ref{7.1} the existence of an infinite family of monic polynomials $\varphi_{p^{\ell},c}(x)\in (\mathbb{Z}\slash p\mathbb{Z})[x]$ having $N_{c}^{(m)}(p)=0$, for every fixed period $m\in \mathbb{Z}_{\geq 2}$; or note that from the second part of Theorem \ref{2.3} that the monic polynomial $\varphi_{p^{\ell},c}(x)\in (\mathcal{O}_{K}\slash p\mathcal{O}_{K})[x]$ has $N_{c}^{(m)}(p)=0$ for every coefficient $c\not \equiv 0\ (\text{mod} \ p\mathcal{O}_{K})$ and fixed period $m\in \mathbb{Z}_{\geq 2}$. But now since $N_{c}^{(m)}(p)=0$ for every fixed $m\in \mathbb{Z}_{\geq 2}$, we then note 
\begin{equation}
    \zeta_{\overline{\varphi_{p^{\ell},c}}}(s)=\textnormal{exp}\biggl(\sum_{m=2}^{\infty} \frac{N_{c}^{(m)}(p)\cdot s^m}{m}\biggr)=\textnormal{exp}(0)=1,
\end{equation} and thus $\zeta_{\overline{\varphi_{p^{\ell},c}}}(s)=1$ for every complex number $s\in \mathbb{C}$. This then completes the whole proof, as desired.
\end{proof}

Similarly, denoting the reduced polynomial map $\varphi_{(p-1)^{\ell},c}$ modulo $p\mathcal{O}_{K}$ by $\overline{\varphi_{(p-1)^{\ell},c}}$ and then also replacing $N_{m}(f)$ with $M_{c}^{(m)}(p)$ in Equation (\ref{eqAM}), we then also let $\zeta_{\overline{\varphi_{(p-1)^{\ell},c}}}$ be the Artin-Mazur zeta function induced by the map $\overline{\varphi_{(p-1)^{\ell},c}}: \mathcal{O}_{K}\slash p\mathcal{O}_{K} \to \mathcal{O}_{K}\slash p\mathcal{O}_{K}$. So now, inspired again by \cite{AM}, we then also have the following corollary showing that the Artin-Mazur zeta functions $\zeta_{\overline{\varphi_{(p-1)^{\ell},c}}}$ associated with infinitely many polynomial maps $\overline{\varphi_{(p-1)^{\ell},c}}$ are also constant functions on the whole complex plane $\mathbb{C}$ and thus also algebraic functions of $s$:

\begin{cor}
Assume Corollary \ref{7.2} or second part of Theorem \ref{3.3}, and denote the polynomial map $\varphi_{(p-1)^{\ell},c}$ modulo prime $p\mathcal{O}_{K}$ by $\overline{\varphi_{(p-1)^{\ell},c}}$. Then the Artin-Mazur zeta function $\zeta_{\overline{\varphi_{(p-1)^{\ell},c}}}(s)=1$ for every complex $s\in \mathbb{C}$.
\end{cor}

\begin{proof}
Applying a similar argument as in Proof of Corollary \ref{cAM}, we then obtain the conclusion, as desired.
\end{proof}

\section{On Number of Monic Polynomials $f\in \mathbb{Z}[x]$ and $g\in \mathbb{Z}[x]$ with Primitive Galois groups}\label{sec8a}

Recall from Corollary \ref{7.1} that there is an infinite family of irreducible monic integer polynomials $f(x) = \varphi_{p^{\ell},c}^m(x)-x$ such that the quotient $\mathbb{Q}_{f}=\mathbb{Q}[x]\slash (f(x))$ induced by $f$ is a number field of degree $\kappa=p^{m\ell}$. Moreover, to each such irreducible monic integer polynomial $f\in \mathbb{Q}[x]$, let $G_{f}$ be the Galois group of $f$ over $\mathbb{Q}$. 

So now, inspired (as in \cite{BK11}) by recent work of Bhargava \cite{gav1} on van der Waerden's Conjecture, we then also wish to determine the number of irreducible monic polynomials  $f\in \mathbb{Z}[x]$ arising from a polynomial discrete dynamical system in Section \ref{sec2}, of bounded height and such that $G_{f}$ is a primitive Galois group not equal to the full symmetric group $S_{\kappa}$. To that end, we (assuming Corollary \ref{7.1}) wish to first adhere to the setup and definition of coefficient height $h(f)$ in \cite{gav1}. That is, for every fixed $\kappa=p^{m\ell}$, we let $E_{\kappa}(H)$ be the number of monic integer polynomials $f(x) = \varphi_{p^{\ell},c}^m(x)-x\in \mathbb{Z}[x]$ of degree $\kappa$ with $h(f)\leq H$ and such that $G_{f}\neq S_{\kappa}$. But now we (as in \cite{BK11}) take great advantage of [\cite{gav1}, Theorem 1] and then obtain the following corollary on $E_{\kappa}(H)$:

\begin{cor}\label{c9.1}
Assume Corollary \ref{7.1}, and let $E_{\kappa}(H)$ be defined as before. Then we have $E_{\kappa}(H)=O(H^{\kappa-1})$.
\end{cor}

\begin{proof}
From Corollary \ref{7.1}, there are infinitely many irreducible monic polynomials $f(x)=\varphi_{p^{\ell},c}^m(x)-x\in \mathbb{Z}[x]$ of degree $\kappa=p^{m\ell}$. Now for every polynomial $f\in \mathbb{Z}[x]\subset \mathbb{Q}[x]$ of fixed degree $\kappa=p^{m\ell}$, let $G_{f}$ be the Galois group of $f$ over $\mathbb{Q}$. But now applying [\cite{gav1}, Theorem 1] on the set of monic polynomials $f(x)\in \mathbb{Z}[x]$ of degree $\kappa$ with $h(f)\leq H$ and such that $G_{f}$ is primitive and $G_{f}\neq S_{\kappa}$, we then immediately obtain the count, as desired.
\end{proof}

Similarly, we may also recall from Corollary \ref{7.2} that there is an infinite family of irreducible monic integer polynomials $g(x) = \varphi_{(p-1)^{\ell},c}^m(x)-x$ such that the quotient $\mathbb{Q}_{g}=\mathbb{Q}[x]\slash (g(x))$ induced by $g$ is a number field of degree $\upsilon=(p-1)^{m\ell}$. And moreover, to each such irreducible integer polynomial $g\in \mathbb{Q}[x]$, let $G_{g}$ be the Galois group of $g$ over $\mathbb{Q}$. So now, inspired again work of Bhargava \cite{gav1} on van der Waerden's Conjecture, we then also wish to determine the number of irreducible monic polynomials  $g\in \mathbb{Z}[x]$ arising from a polynomial discrete dynamical system in Section \ref{sec3}, of bounded height and such that $G_{g}$ is a primitive Galois group not equal to the full symmetric group $S_{\upsilon}$. To that end, we (assuming Corollary \ref{7.2}) again first import the setup and definition of coefficient height $h(g)$ in \cite{gav1}. That is, for every fixed $\upsilon=(p-1)^{m\ell}$, we let $E_{\upsilon}(H)$ be the number of monic integer degree-$\upsilon$ polynomials $g(x) = \varphi_{(p-1)^{\ell},c}^{m\ell}(x)-x$ with $h(g)\leq H$ and such that $G_{g}\neq S_{\upsilon}$. So now, by again taking great advantage of Bhargava's theorem [\cite{gav1}, Theorem 1], we then also obtain here the following corollary:

\begin{cor}
Assume Corollary \ref{7.2}, and let $E_{\upsilon}(H)$ be defined as before. Then we have $E_{\upsilon}(H)=O(H^{\upsilon-1})$.
\end{cor}

\begin{proof}
Applying a similar argument as in the Proof of Cor. \ref{c9.1}, we then obtain the count, as needed. That is, from Corollary \ref{7.2}, there are infinitely many irreducible monic polynomials $g(x)=\varphi_{(p-1)^{\ell},c}^m(x)-x\in \mathbb{Z}[x]$ of degree $\upsilon=(p-1)^{m\ell}$. So now, for every $g\in \mathbb{Z}[x]\subset \mathbb{Q}[x]$ of fixed degree $\upsilon=(p-1)^{m\ell}$, let $G_{g}$ be the Galois group of $g$ over $\mathbb{Q}$. But now applying [\cite{gav1}, Theorem 1] on the set of monic degree-$\upsilon$ polynomials $g(x)\in \mathbb{Z}[x]$ with $h(g)\leq H$ and such that $G_{g}$ is primitive and $G_{g}\neq S_{\upsilon}$, we then immediately obtain the count, as needed.
\end{proof}

\section{On the Number of Algebraic Number fields having Bounded Absolute Discriminant}\label{sec8}

\subsection*{On Fields $\mathbb{Q}_{f}$ \& $\mathbb{Q}_{g}$ with Bounded Absolute Discriminant \& Prescribed Galois group}
As in Section \ref{sec8a}, recall from Corollary \ref{7.1} that there is an infinite family of irreducible monic integer polynomials $f(x) = \varphi_{p^{\ell},c}^m(x)-x$ such that $\mathbb{Q}_{f}=\mathbb{Q}[x]\slash (f(x))$ induced by $f$ is a number field of degree $\kappa=p^{m\ell}$. Similarly, recall from Corollary \ref{7.2} that one can always find an infinite family of irreducible monic integer polynomials $g(x) = \varphi_{(p-1)^{\ell},c}^m(x)-x$ such that $\mathbb{Q}_{g}=\mathbb{Q}[x]\slash (g(x))$ induced by $g$ is a number field of degree $\upsilon=(p-1)^{m\ell}$. Moreover, we may associate to $\mathbb{Q}_{f}$ (resp., $\mathbb{Q}_{g}$) an integer Disc$(\mathbb{Q}_{f})$ (resp., Disc$(\mathbb{Q}_{g})$) called the discriminant. So now, inspired (as in \cite{BK3, BK2}) by number field-counting advances in arithmetic statistics, we in this section also wish to count the number of fields $\mathbb{Q}_{f}$ and $\mathbb{Q}_{g}$ induced by irreducible polynomials $f$ and $g$ arising from polynomial discrete dynamical systems in Sect. \ref{sec2} and \ref{sec3} (and ascertained by Corollary \ref{7.1} and \ref{7.2}). To that end, we (as in \cite{BK3, BK2}) define and then also determine the asymptotic behavior of the following counting functions  
\begin{equation}\label{N_{k}}
N_{\kappa}(X) := \# \Bigl\{\mathbb{Q}_{f}\slash \mathbb{Q} : [\mathbb{Q}_{f} : \mathbb{Q}] = \kappa \textnormal{ and } |\text{Disc}(\mathbb{Q}_{f})|\leq X \Bigr\}
\end{equation} 
\begin{equation}\label{M_{l}}
M_{\upsilon}(X) := \# \Bigl\{\mathbb{Q}_{g}\slash \mathbb{Q} : [\mathbb{Q}_{g} : \mathbb{Q}] = \upsilon \textnormal{ and} \ |\text{Disc}(\mathbb{Q}_{g})|\leq X \Bigr\}
\end{equation} as a positive real number $X\to \infty$. So now, motivated (as in \cite{BK3}) by work of Lemke Oliver-Thorne \cite{lem} and then applying here the first part of their [\cite{lem}, Theorem 1.2] on the function $N_{\kappa}(X)$, we then obtain the following:

\begin{cor} \label{8.1} Assume Corollary \ref{7.1}, and let $N_{\kappa}(X)$ be the number defined as in \textnormal{(\ref{N_{k}})}. Then we have 
\begin{equation}\label{N_{k}(x)} 
N_{\kappa}(X)\ll_{\kappa}X^{2d - \frac{d(d-1)(d+4)}{6\kappa}}\ll X^{\frac{8\sqrt{\kappa}}{3}}, \text{where d is the least integer for which } \binom{d+2}{2}\geq 2\kappa + 1.
\end{equation}
\end{cor}

\begin{proof}
To see the inequality \textnormal{(\ref{N_{k}(x)})}, we first recall from Cor. \ref{7.1} the existence of infinitely many monic integer polynomials $f(x)\in \mathbb{Q}[x]$ such that $\mathbb{Q}_{f}\slash \mathbb{Q}$ is a number field of degree $\kappa=p^{m\ell}$. Hence, the set of fields $\mathbb{Q}_{f}\slash \mathbb{Q}$ is not empty. But now applying [\cite{lem}, Theorem 1.2 (1)] on $N_{\kappa}(X)$, we then obtain inequality \textnormal{(\ref{N_{k}(x)})}, as needed.
\end{proof}

Motivated again by that same work of Lemke Oliver-Thorne \cite{lem}, we again take great advantage of the first part of [\cite{lem}, Theorem 1.2] by applying it on $M_{\upsilon}(X)$. In doing so, we then immediate obtain the following:

\begin{cor}Assume Corollary \ref{7.2}, and let $M_{\upsilon}(X)$ be the number defined as in \textnormal{(\ref{M_{l}})}. Then we have 
\begin{equation}\label{M_{l}(x)}
M_{\upsilon}(X)\ll_{\upsilon}X^{2d - \frac{d(d-1)(d+4)}{6\upsilon}}\ll X^{\frac{8\sqrt{\upsilon}}{3}}, \text{where d is the least integer for which } \binom{d+2}{2}\geq 2\upsilon + 1.
\end{equation}
\end{cor}

\begin{proof}
Applying a similar argument as in Proof of Corollary \ref{8.1}, we then obtain inequality \textnormal{(\ref{M_{l}(x)})}, as needed.
\end{proof}

We recall more generally that an algebraic number field $K$ is  \say{\textit{monogenic}} if there exists an algebraic number $\alpha \in K$ such that the ring of integers $\mathcal{O}_{K}$ is the subring $\mathbb{Z}[\alpha]$ generated by $\alpha$ over $\mathbb{Z}$, i.e., $\mathcal{O}_{K}= \mathbb{Z}[\alpha]$. So now, inspired (as in \cite{BK3, BK2}), we also wish to count the number of fields $\mathbb{Q}_{f}$ induced by irreducible polynomials $f\in \mathbb{Z}[x]$ arising from a polynomial discrete dynamical system in Sect.\ref{sec2} (and ascertained by Corollary \ref{7.1}), that are monogenic with $|\Delta(\mathbb{Q}_{f})| < X$ and Galois group Gal$(\mathbb{Q}_{f}\slash \mathbb{Q})$ equal to symmetric group $S_{p^{m\ell}}$. To do so, we (as in \cite{BK3, BK2}) take great advantage of a result due to Bhargava-Shankar-Wang [\cite{sch1}, Cor. 1.3] and then obtain:

\begin{cor}\label{8.3}
Assume Corollary \ref{7.1}. Then the number of isomorphism classes of algebraic number fields $\mathbb{Q}_{f}$ of degree $\kappa=p^{m\ell}$ and with $|\Delta(\mathbb{Q}_{f})| < X$ that are monogenic and have associated Galois group $S_{\kappa}$ is $\gg X^{\frac{1}{2} + \frac{1}{\kappa}}$.
\end{cor}

\begin{proof}
To see this, we recall from Cor. \ref{7.1} the existence of infinitely many irreducible monic polynomials $f(x)$ over $\mathbb{Z}$ (and so over $\mathbb{Q}$) such that the field $\mathbb{Q}_{f}$ induced by $f$ is an algebraic number field of degree $\kappa=p^{m\ell}$ number field. This then also means that the set of number fields $\mathbb{Q}_{f}$ is not empty. But now applying [\cite{sch1}, Corollary 1.3] on the underlying fields $\mathbb{Q}_{f}$ with $|\Delta(\mathbb{Q}_{f})| < X$ that are monogenic and have associated Galois group $S_{\kappa}$, it then follows that the number of isomorphism classes of such fields $\mathbb{Q}_{f}$ is $\gg X^{\frac{1}{2} + \frac{1}{\kappa}}$, as required.
\end{proof}

Similarly, we again take great advantage of that same result [\cite{sch1}, Corollary 1.3] to then also immediately count in the following corollary the number of number fields $\mathbb{Q}_{g}$ induced by irreducible monic integer polynomials $g$ arising from a polynomial discrete dynamical system in Section \ref{sec3} (and  ascertained by Corollary \ref{7.2}), that are monogenic with $|\Delta(\mathbb{Q}_{g})| < X$ and have associated Galois group Gal$(\mathbb{Q}_{g}\slash \mathbb{Q})$ equal to symmetric group $S_{(p-1)^{m\ell}}$:
\begin{cor}
Assume Corollary \ref{7.2}. Then the number of isomorphism classes of algebraic number fields $\mathbb{Q}_{g}$ of degree $\upsilon=(p-1)^{m\ell}$ and $|\Delta(\mathbb{Q}_{g})| < X$ that are monogenic and have associated Galois group $S_{\upsilon}$ is $\gg X^{\frac{1}{2} + \frac{1}{\upsilon}}$.
\end{cor}

\begin{proof}
Applying a similar argument as in the Proof of Corollary \ref{8.3}, we then obtain the count, as required.
\end{proof}

\subsection*{On Fields $K_{f}$ \& $L_{g}$ with Bounded Absolute Discriminant \& Prescribed Galois group}

Recall that we proved in Corollary \ref{7.1} the existence of an infinite family of irreducible monic integer polynomials $f(x) = \varphi_{p^{\ell},c}^m(x)-x\in \mathbb{Q}[x]\subset K[x]$ for every fixed $\ell$ and $m$; and more to this, we may also recall that the second part of Theorem \ref{2.3} (i.e., the part in which we proved that $N_{c}^{(m)}(p) = 0$ for every coefficient $c\not \equiv 0\ (\text{mod} \ p\mathcal{O}_{K})$) implies that $f(x) = \varphi_{p^{\ell},c}^m(x)-x \in \mathcal{O}_{K}[x]\subset K[x]$ is irreducible modulo prime ideal $p\mathcal{O}_{K}$. So now, as in Section \ref{sec7}, we may to each irreducible monic polynomial $f\in \mathcal{O}_{K}[x]$ associate a field $K_{f} = K[x]\slash (f(x))$, which is again an algebraic number field of deg$(f) = p^{m\ell}$ over $K$; and more to this, we also recall from algebraic number theory that associated to $K_{f}$ is an integer Disc$(K_{f})$ called the discriminant. Moreover, since we also obtain an inclusion $\mathbb{Q}\hookrightarrow K \hookrightarrow K_{f}$ of number fields, we then also note that the degree $t=[K_{f} : \mathbb{Q}] = [K : \mathbb{Q}] \cdot [K_{f} : K] = np^{m\ell}$, for fixed degree $n\geq 2$ of any number field $K$. So now, as before we also wish to count the number of number fields $K_{f}\slash \mathbb{Q}$ induced by irreducible polynomials $f\in \mathcal{O}_{K}[x]$ arising from a polynomial discrete dynamical system in Sect.\ref{sec2}. To do so, we as before define and then determine the asymptotic behavior of the counting function
\begin{equation}\label{N_{m}}
N_{t}(X) := \# \Bigl\{K_{f}\slash \mathbb{Q} : [K_{f} : \mathbb{Q}] = t \textnormal{ and } |\text{Disc}(K_{f})|\leq X \Bigr\}
\end{equation} as a positive real number $X\to \infty$. But now, as before motivated by that same work of Lemke Oliver-Thorne \cite{lem} and then applying here the first part of [\cite{lem}, Theorem 1.2] on the counting function $N_{t}(X)$, we then obtain:

\begin{cor} \label{8.5}Fix any number field $K\slash \mathbb{Q}$ of degree $n\geq 2$ with the ring of integers $\mathcal{O}_{K}$. Assume Corollary \ref{7.1} or second part of Theorem \ref{2.3}, and let $N_{t}(X)$ be the number defined as in \textnormal{(\ref{N_{m}})}. Then we have 
\begin{equation}\label{N_{m}(x)} 
N_{t}(X)\ll_{t}X^{2d - \frac{d(d-1)(d+4)}{6t}}\ll X^{\frac{8\sqrt{t}}{3}}, \text{where d is the least integer for which } \binom{d+2}{2}\geq 2t + 1.
\end{equation}
\end{cor}

\begin{proof}
To see the inequality \textnormal{(\ref{N_{m}(x)})}, we first recall from Corollary \ref{7.1} the existence of infinitely many monic polynomials $f(x)$ over $\mathbb{Q}\subset K_{f}$ such that $K_{f}\slash \mathbb{Q}$ is an algebraic number field of degree $t=np^{m\ell}$, or recall from the second part of Theorem \ref{2.3} the existence of monic integral polynomials $f(x) = \varphi_{p^{\ell},c}^m(x)-x\in K[x]$ that are irreducible modulo any fixed prime ideal $p\mathcal{O}_{K}$ for every coefficient $c\not \in p\mathcal{O}_{K}$ and hence induce degree-$t$ number fields $K_{f}\slash \mathbb{Q}$. This then also means that the set of odd degree-$t$ algebraic number fields $K_{f}\slash \mathbb{Q}$ is not empty. But now applying here [\cite{lem}, Theorem 1.2 (1)] on the number $N_{t}(X)$, we then obtain inequality \textnormal{(\ref{N_{m}(x)})}, as needed.
\end{proof}

Similarly, recall that we proved in Corollary \ref{7.2} the existence of an infinite family of irreducible monic integer polynomials $g(x) = \varphi_{(p-1)^{\ell},c}^m(x)-x \in \mathbb{Q}[x]\subset K[x]$ for every fixed $\ell $ and $m$; and more to this, we again recall that the second part of Theorem \ref{3.3} (i.e., the part in which we proved that $M_{c}^{(m)}(p) = 0$ for every coefficient $c\not \equiv \pm1, 0\ (\text{mod} \ p\mathcal{O}_{K})$) implies $g(x) = \varphi_{(p-1)^{\ell},c}^m(x)-x \in \mathcal{O}_{K}[x]\subset K[x]$ is irreducible modulo prime ideal  $p\mathcal{O}_{K}$. As before, we may to each irreducible monic polynomial $g\in \mathcal{O}_{K}[x]$ associate a number field $L_{g} = K[x]\slash (g(x))$ of even deg$(g) = (p-1)^{m\ell}$ over $K$; and also recall that attached to $L_{g}$ is an integer Disc$(L_{g})$. Since we also now obtain an inclusion $\mathbb{Q}\hookrightarrow K \hookrightarrow L_{g}$ of number fields, we then also note that $r=[L_{g} : \mathbb{Q}] = [K : \mathbb{Q}] \cdot [L_{g} : K] = n(p-1)^{m\ell}$ for any fixed $n=[K : \mathbb{Q}]$. As before, we also wish to count number fields $L_{g}\slash \mathbb{Q}$ induced by irreducible monic polynomials $g\in \mathcal{O}_{K}[x]$ arising from a polynomial discrete dynamical system in Section \ref{sec3}. To that end, we again define and then determine the asymptotic behavior of 
\begin{equation}\label{M_{r}}
M_{r}(X) := \# \Bigl\{L_{g}\slash \mathbb{Q} : [L_{g} : \mathbb{Q}] = r \textnormal{ and} \ |\text{Disc}(L_{g})|\leq X \Bigr\}
\end{equation} as a positive real number $X\to \infty$. By again taking great advantage of [\cite{lem}, Theorem 1.2 (1)], we then obtain:

\begin{cor} Fix any number field $K\slash \mathbb{Q}$ of degree $n\geq 2$ with the ring of integers $\mathcal{O}_{K}$. Assume Corollary \ref{7.2} or second part of Theorem \ref{3.3}, and let $M_{r}(X)$ be the number defined as in \textnormal{(\ref{M_{r}})}. Then we have 
\begin{equation}\label{M_{r}(x)}
M_{r}(X)\ll_{r}X^{2d - \frac{d(d-1)(d+4)}{6r}}\ll X^{\frac{8\sqrt{r}}{3}}, \text{where d is the least integer for which } \binom{d+2}{2}\geq 2r + 1.
\end{equation}
\end{cor}

\begin{proof}
Applying a similar argument as in the Proof of Cor. \ref{8.5}, we then obtain inequality \textnormal{(\ref{M_{r}(x)})} as required.
\end{proof}

As before, we also wish to apply again that same result due to Bhargava-Shankar-Wang [\cite{sch1}, Corollary 1.3] to the number of number fields $K_{f}\slash \mathbb{Q}$ induced by irreducible monic polynomials $f\in \mathcal{O}_{K}[x]$ arising from a polynomial discrete dynamical system in Section \ref{sec2}, that are monogenic and such that the associated Galois group Gal$(K_{f}\slash \mathbb{Q})$ is equal to the symmetric group $S_{np^{m\ell}}$. In doing so, we then obtain the following corollary:

\begin{cor}\label{8.7}
Assume Corollary \ref{7.1} or second part of Theorem \ref{2.3}. The number of isomorphism classes of number fields $K_{f}\slash\mathbb{Q}$ of degree $t$ and $|\Delta(K_{f})| < X$ that are monogenic and having Galois group $S_{t}$ is $\gg X^{\frac{1}{2} + \frac{1}{t}}$.
\end{cor}

\begin{proof}
To see this, we first recall from Corollary \ref{7.1} the existence of infinitely many monic polynomials $f(x)$ over $\mathbb{Q}\subset K_{f}$ such that $K_{f}\slash \mathbb{Q}$ is an algebraic  number field of degree $t=np^{m\ell}$, or recall from the second part of Theorem \ref{2.3} the existence of monic integral polynomials $f(x) = \varphi_{p^{\ell},c}^m(x)-x\in K[x]$ that are irreducible modulo any fixed prime ideal $p\mathcal{O}_{K}$ for every coefficient $c\not \in p\mathcal{O}_{K}$ and hence induce degree-$t$ number fields $K_{f}\slash \mathbb{Q}$. This then means that the set of algebraic number fields $K_{f}\slash \mathbb{Q}$ is not empty. So now, applying [\cite{sch1}, Corollary 1.3] on the underlying number fields $K_{f}$ with $|\Delta(K_{f})| < X$ that are monogenic and have associated Galois group $S_{t}$, we then obtain that the number of isomorphism classes of such number fields $K_{f}$ is $\gg X^{\frac{1}{2} + \frac{1}{t}}$, as required.
\end{proof}

Similarly, we also wish to apply again that same result due to Bhargava-Shankar-Wang [\cite{sch1}, Corollary 1.3] to the number of number fields $L_{g}\slash \mathbb{Q}$ induced by irreducible monic polynomials $g\in \mathcal{O}_{K}[x]$ arising from a polynomial discrete dynamical system in Section \ref{sec3}, that are monogenic and such that the associated Galois group Gal$(L_{g}\slash \mathbb{Q})$ is equal to the symmetric group $S_{n(p-1)^{m\ell}}$. In doing so, we then obtain the following corollary:
\begin{cor}
Assume Corollary \ref{7.2} or second part of Theorem \ref{3.3}. The number of isomorphism classes of number fields $L_{g}\slash \mathbb{Q}$ of degree $r$ and $|\Delta(L_{g})| < X$ that are monogenic and having Galois group $S_{r}$ is $\gg X^{\frac{1}{2} + \frac{1}{r}}$.
\end{cor}

\begin{proof}
Applying a similar argument as in the Proof of Corollary \ref{8.7}, we then obtain the count as required.
\end{proof}

\section{On the Number of Algebraic Number fields $K_{f}$ \& $L_{g}$ with Prescribed Class Number}\label{sec9}

Recall that for any number field $K$ with ring of integers $\mathcal{O}_{K}$, we have a finite abelian group called \say{\textit{ideal class group}} $\textnormal{Cl}(\mathcal{O}_{K})$ (also denoted as $\textnormal{Cl}(K)$), which is classically known to provide a way of measuring how far $\mathcal{O}_{K}$ is from being a unique factorization domain. Now even though the order (also called the \say{\textit{class number}} of $K$ (denoted as $h_{K}$)) of $\textnormal{Cl}(\mathcal{O}_{K})$ is finite, it is well known in algebraic and analytic number theory and even more so in arithmetic statistics, that computing $\textnormal{Cl}(\mathcal{O}_{K})$ in practice let alone determine precisely $h_{K}$, is a hard problem.

Now recall from Corollary \ref{7.1} that there is an infinite family of irreducible monic polynomials $f(x) = \varphi_{p^{\ell},c}^m(x)-x\in \mathbb{Z}[x]$ such that $\mathbb{Q}_{f}=\mathbb{Q}[x]\slash (f(x))$ is a number field of degree $p^{m\ell}$. Moreover, in light of the foregoing discussion, we then also have an invariant $\textnormal{Cl}(\mathbb{Q}_{f})$ associated to $\mathbb{Q}_{f}$ and moreover  $h_{\mathbb{Q}_{f}}$ is finite. So now, inspired by work on class groups of number fields in arithmetic statistics and in particular by work of Ho-Shankar-Varma \cite{ho} on odd degree number fields with odd class number, we then also wish to count here the number of fields $\mathbb{Q}_{f}$ induced by irreducible polynomials $f\in \mathbb{Z}[x]$ arising from a polynomial discrete dynamical system in Section \ref{sec2} (and ascertained by Corollary \ref{7.1}), with associated Galois group $S_{p^{m\ell}}$ and with prescribed class number. To do so, we take great advantage of Ho-Shankar-Varma's theorem [\cite{ho}, Theorem 4] and then obtain the following corollary showing there exits infinitely many $S_{p^{m\ell}}$-number fields $\mathbb{Q}_{f}$ with odd class number:

\begin{cor}\label{9.1}
Assume Corollary \ref{7.1}, and let $\kappa=p^{m\ell}$ be any fixed odd integer. Then there exist infinitely many $S_{\kappa}$-algebraic number fields $\mathbb{Q}_{f}$ of odd degree $\kappa$  having odd class number.  More precisely, we have 
\begin{center}
$\#\Bigl\{ \mathbb{Q}_{f} : |\Delta(\mathbb{Q}_{f})| < X \textnormal{ and } 2\nmid \textnormal{Cl}(\mathbb{Q}_{f})|\Bigr\}\gg X^{\frac{\kappa + 1}{2\kappa -2}}$,
\end{center} where the implied constants depend on degree $\kappa$ and on an arbitrary finite set $S$ of primes given as in \textnormal{\cite{ho}}.
\end{cor}

\begin{proof}
From Cor. \ref{7.1}, it follows that the family of number fields $\mathbb{Q}_{f}$ of degree $\kappa = p^{m\ell}$ is not empty. Now since $\kappa$ is an odd integer, we then see that the claim follows from [\cite{ho}, Thm. 4(a)] by setting $\mathbb{Q}_{f}=K$ as needed.
\end{proof}

Similarly, recall from the second part of Theorem \ref{2.3} the existence of integral polynomials $f(x) = \varphi_{p^{\ell},c}^m(x)-x\in K[x]$ that are irreducible modulo prime $p\mathcal{O}_{K}$ for every coefficient $c\not \in p\mathcal{O}_{K}$; and moreover each such $f\in \mathcal{O}_{K}[x]$ induces an algebraic number field $K_{f}\slash \mathbb{Q}$ of degree $t=np^{m\ell}$, for every fixed degree $n=[K: \mathbb{Q}]$. Now assuming that the degree $n$ is an odd integer and so is also the degree $t$, we then also obtain the following corollary on the number of fields $K_{f}\slash \mathbb{Q}$ induced by irreducible polynomials $f\in \mathcal{O}_{K}[x]$ arising from a polynomial discrete dynamical system in Section \ref{sec2}, with associated Galois group $S_{np^{m\ell}}$ and also having odd class number:  

\begin{cor}
Assume second part of Theorem \ref{2.3}, and let $t=np^{m\ell}$ be any fixed odd integer. There exist infinitely many $S_{t}$-algebraic number fields $K_{f}$ of degree $t$ with odd class number.  More precisely, we have 
\begin{center}
$\#\Bigl\{ K_{f}\slash \mathbb{Q} : |\Delta(K_{f})| < X \textnormal{ and } 2\nmid \textnormal{Cl}(K_{f})|\Bigr\}\gg X^{\frac{t + 1}{2t -2}}$,
\end{center} where the implied constants depend on degree $t$ and on an arbitrary finite set $S$ of primes given as in \textnormal{\cite{ho}}.
\end{cor}

\begin{proof}
Applying a similar argument as in the Proof of Corollary \ref{9.1}, we then obtain the count as needed. 
\end{proof}

As before, we may also recall from Corollary \ref{7.2} the existence of an infinite family of irreducible monic integer polynomials $g(x) = \varphi^{m}_{(p-1)^{\ell},c}(x)-x\in \mathbb{Z}[x]$ such that $\mathbb{Q}_{g} = \mathbb{Q}[x]\slash (g(x))$ induced by $g$ is a number field of degree $(p-1)^{m\ell}$. But now we also observe that to each such obtained number field $\mathbb{Q}_{g}$, we can associate a finite class group $\textnormal{Cl}(\mathbb{Q}_{g})$ and so $h_{\mathbb{Q}_{g}}$ is finite. So now, by taking again great advantage of work on class groups of number fields in arithmetic statistics and in particular the work of Siad \cite{Sia} on $S_{n}$-number fields $K$ of any even degree $n\geq 4$ and signature $(r_{1}, r_{2})$ where $r_{1}$ are the real embeddings of $K$  and $r_{2}$ are the pairs of conjugate complex embeddings of $K$, we then also obtain the following corollary on the number of number fields $\mathbb{Q}_{g}\slash \mathbb{Q}$ induced by irreducible monic polynomials $g\in \mathbb{Z}[x]$ arising from a polynomial discrete dynamical system in Sect.\ref{sec3} (and ascertained by Corollary \ref{7.2}), with associated Galois group $S_{(p-1)^{m\ell}}$ and having odd class number: 

\begin{cor}\label{9.3}
Assume Corollary \ref{7.2}, and let $\upsilon=(p-1)^{m\ell}$ be any even integer. There are infinitely many degree-$\upsilon$ monogenic number fields $\mathbb{Q}_{g}$ of any signature and associated Galois group $S_{\upsilon}$ having odd class number. 
\end{cor}
\begin{proof}
To see this, we note that by Cor. \ref{7.2}, it follows that the family of number fields $\mathbb{Q}_{g}$ of degree $\upsilon = (p-1)^{m\ell}$ is not empty. So now, since $\upsilon$ is even, we then see that the claim follows from [\cite{Sia}, Cor. 10] as indeed needed.
\end{proof}

Similarly, recall from the second part of Theorem \ref{3.3} the existence of integral polynomials $g(x) = \varphi_{(p-1)^{\ell},c}^m(x)-x\in K[x]$ that are irreducible modulo prime $p\mathcal{O}_{K}$ for every $c\not \in p\mathcal{O}_{K}$; and moreover each such $g\in \mathcal{O}_{K}[x]$ induces a field $L_{g}\slash \mathbb{Q}$ of degree $r=n(p-1)^{m\ell}$, for every fixed $n$. As before, we also have the following corollary on the number of fields $L_{g}\slash \mathbb{Q}$ induced by irreducible polynomials $g\in \mathcal{O}_{K}[x]$ arising from a polynomial discrete dynamical system in Section \ref{sec3}, with associated Galois group $S_{n(p-1)^{m\ell}}$ and having odd class number:

\begin{cor}
Assume second part of Theorem \ref{3.3}, and let $r=n(p-1)^{m\ell}$ be any fixed even integer. Then there are infinitely many monogenic $S_{r}$-number fields $L_{g}$ of degree $r$ and any signature having odd class number. 
\end{cor}

\begin{proof}
Applying a similar argument as in the Proof of Corollary \ref{9.3}, we then obtain the count as needed. 
\end{proof}

\section{On the Equidistribution of Families of Artin $L$-Functions induced by Fields $K_{f}$ \& $L_{g}$}

Recall that for any degree-$n$ every number field $K$ with ring of integers $\mathcal{O}_{K}$, we have a Dedekind zeta function $\zeta_{K}$ associated with $K$; and which for every complex $s\in \mathbb{C}$ with $\mathfrak{R}(s)>1$, this zeta function $\zeta_{K}$ is defined by  
\begin{equation}
    \zeta_{K}(s) = \sum_{I\subset \mathcal{O}_{K}}\frac{1}{|\mathcal{O}_{K}\slash I|^s}=\prod_{\mathfrak{p}\subset \mathcal{O}_{K}}\frac{1}{1-|\mathcal{O}_{K}\slash \mathfrak{p}|^{-s}}   
\end{equation}where the above sum (resp., the above product) is taken over all the nonzero ideals $I\subset \mathcal{O}_{K}$ (resp., over all the nonzero prime ideals $\mathfrak{p}$), and $|\mathcal{O}_{K}\slash I|$ (resp. $|\mathcal{O}_{K}\slash \mathfrak{p}|$) is the absolute norm of $I$ (resp. the absolute norm of $\mathfrak{p}$). As a generalization of the Riemann zeta function $\zeta_{\mathbb{Q}}(s)$ (whose vanishing on the line  $\mathfrak{R}(s) = \frac{1}{2}$ is intimately related to the distribution of primes $p \in \mathbb{Z}$ (as a consequence of the Riemann Hypothesis)), it is a classical theme in number theory to understand the vanishing of $\zeta_{K}(s)$ especially on the line $\mathfrak{R}(s) = \frac{1}{2}$, since such vanishing of the zeta function $\zeta_{K}(s)$ is also expected of revealing precise information about the distribution of prime ideals $\mathfrak{p}$ in $K$ (as also a consequence of the number field version of the Riemann Hypothesis). Note that from [\cite{Nico}, Page 10] the zeta function $\zeta_{K}(s)$ factors as $\zeta_{K}(s)=\zeta_{\mathbb{Q}}(s)L(s, \rho_{K}$), where $L(s, \rho_{K})$ is the Artin $L$-function corresponding to an Artin representation $\rho_{K}: \text{Gal}(\mathbb{Q})\to \text{Gal}(M\slash \mathbb{Q}) \hookrightarrow S_{n}\to \text{GL}_{n-1}(\mathbb{C})$, and $M$ is the normal closure of $K$.

So now, for every degree-$\kappa$ number field $\mathbb{Q}_{f}$ obtained from a polynomial discrete dynamical system in Section \ref{sec2} and ascertained by Corollary \ref{7.1}, we have a Dedekind zeta function $\zeta_{\mathbb{Q}_{f}}$ corresponding to $\mathbb{Q}_{f}$. Moreover, we also know from the remarkable work of Shankar-S\"{o}dergren-Templier [\cite{Nico}, Page 2] that the zeta function $\zeta_{\mathbb{Q}_{f}}(s)=\zeta(s)L(s, \rho_{\mathbb{Q}_{f}}$), where $\zeta(s)$ is the Riemann zeta function, $L(s, \rho_{\mathbb{Q}_{f}})$ is the Artin $L$-function, $\rho_{\mathbb{Q}_{f}}: \text{Gal}(M_{f}\slash \mathbb{Q}) \hookrightarrow S_{\kappa}\to \text{GL}_{\kappa-1}(\mathbb{C})$ is an Artin representation, and where $M_{f}$ is the normal closure of $\mathbb{Q}_{f}$.

Now motivated (as in \cite{BK11}) by remarkable work of Shankar-S\"{o}dergren-Templier \cite{Nico} on equidistribution of Artin $L$-functions arising from number fields induced by irreducible monic integer polynomials, we in the same spirit as in \cite{Nico} also wish to study the distribution of Artin $L$-functions $L(s, \rho_{\mathbb{Q}_{f}})$ arising from number fields $\mathbb{Q}_{f}$ induced by irreducible monic polynomials $f\in \mathbb{Z}[x]$ obtained from a polynomial discrete dynamical system in Section \ref{sec2}. To do so, we (assuming Corollary \ref{7.1}) wish to first adhere to the setup and notation in \cite{Nico}. That is, let $V(\mathbb{Z})^{\text{irr}}$ be the space consisting of irreducible monic integer polynomials  $f(x)=\varphi_{p^{\ell},c}^m(x)-x$ of fixed degree $\kappa=p^{m\ell}$,  and let $V(\mathbb{Z})^{\text{max}}\subset V(\mathbb{Z})^{\text{irr}}$ be a subset consisting of irreducible monic integer polynomials $f$ such that $R_{f}=\mathbb{Z}[x]\slash (f(x))$ is a maximal order in $\mathbb{Q}_{f}=\mathbb{Q}[x]\slash (f(x))$. Following \cite{Nico}, it also follows here that the additive group $G_{a}(\mathbb{Z})=\mathbb{Z}$ necessarily acts naturally on our space $V(\mathbb{Z})^{\text{irr}}$ via translation, namely, $(b \cdot f)(x):= f(x+b)$ for every element $b\in \mathbb{Z}$ and for every $f\in V(\mathbb{Z})^{\text{irr}}$; and moreover, this action of $G_{a}(\mathbb{Z})=\mathbb{Z}$ by translation also necessarily preserves each of the sets $V(\mathbb{Z})^{\text{irr}}$ and $V(\mathbb{Z})^{\text{max}}$. Now let $\mathfrak{F}_{1}$ be a family consisting of the $\mathbb{Z}$-orbits on $V(\mathbb{Z})^{\text{max}}$. It then follows (from \cite{Nico}) that the family $\mathfrak{F}_{1}$ necessarily parametrizes degree-$\kappa$ monogenized number fields $(\mathbb{Q}_{f}, \alpha)$ over $\mathbb{Q}$ up to isomorphism. We note that (from [\cite{Nico}, Subsection 2.3]) this same family $\mathfrak{F}_{1}$ parametrizing  degree-$k$ monogenized number fields $(\mathbb{Q}_{f}, \alpha)$ is also the family of associated $L$-functions $L(s, \rho_{\mathbb{Q}_{f}})$.

So now, by taking great advantage of a nice theorem of Shankar-S\"{o}dergren-Templier[\cite{Nico}, Theorem 1.1], we then also obtain the following corollary on the family $\mathfrak{F}_{1}$ parametrizing degree-$\kappa$ monogenized fields $(\mathbb{Q}_{f}, \alpha)$:

\begin{cor}\label{11.1}
Assume Corollary \ref{7.1}, and let $\mathfrak{F}_{1}$ be as before. Then $\mathfrak{F}_{1}$ parametrizing monogenized degree-$\kappa$ fields ordered by height $h(f)$ as defined in \textnormal{\cite{Nico}} satisfies Sato-Tate equidistribution in the sense of \textnormal{[\cite{Sar}, Conj.1]}. 
\end{cor}

\begin{proof}
Since we know from Corollary \ref{7.1} that there are infinitely many irreducible monic integer polynomials $f$ such that $\mathbb{Q}_{f}$ is a number field of degree $\kappa=p^{m\ell}$, then this also means that the family of degree-$\kappa$ number fields $\mathbb{Q}_{f}\slash \mathbb{Q}$ is not empty. Now letting $\alpha$ be the image of $x$ in $R_{f}=\mathbb{Z}[x]\slash (f(x))$ and so (by \cite{Nico}) the pair $(\mathbb{Q}_{f}, \alpha)$ is a degree-$\kappa$ monogenized field, it then follows that the family of monogenized degree-$\kappa$ number fields $(\mathbb{Q}_{f}, \alpha)$ is not empty; which also means that the family $\mathfrak{F}_{1}$ parametrizing  degree-$\kappa$ monogenized fields $(\mathbb{Q}_{f}, \alpha)$ is not empty. But now applying [\cite{Nico}, Thm. 1.1] to the underlying family $\mathfrak{F}_{1}$ ordered by height $h(f)$ as defined in [\cite{Nico}, Page 3], it then follows that $\mathfrak{F}_{1}$ satisfies Sato-Tate equidistribution in the sense of \textnormal{[\cite{Sar}, Conjecture 1]} as needed.
\end{proof}

Similarly, for every degree-$\upsilon$ field $\mathbb{Q}_{g}$ obtained from a polynomial discrete dynamical system in Section \ref{sec3} and ascertained by Corollary \ref{7.2}, we also have a Dedekind zeta function $\zeta_{\mathbb{Q}_{g}}$ corresponding to $\mathbb{Q}_{g}$. Moreover, it again follows from \cite{Nico} that the Dedekind zeta function $\zeta_{\mathbb{Q}_{g}}(s)=\zeta(s)L(s, \rho_{\mathbb{Q}_{g}}$), where $L(s, \rho_{\mathbb{Q}_{g}})$ is the Artin $L$-function,  $\rho_{\mathbb{Q}_{g}}: \text{Gal}(M_{g}\slash \mathbb{Q}) \hookrightarrow S_{\upsilon}\to \text{GL}_{\upsilon-1}(\mathbb{C})$ is an Artin representation, and $M_{g}$ the normal closure of $\mathbb{Q}_{g}$. 

So now, in again the same spirit as in \cite{Nico}, we also wish to study the distribution of Artin $L$-functions $L(s, \rho_{\mathbb{Q}_{g}})$ arising from fields $\mathbb{Q}_{g}$ induced by irreducible polynomials $g\in \mathbb{Z}[x]$ obtained from a polynomial discrete dynamical system in Section \ref{sec3}. To that end, we (also assuming Corollary \ref{7.2}) adhere again to the setup and notation in \cite{Nico}. That is, we again let $W(\mathbb{Z})^{\text{irr}}$ be the space consisting of irreducible monic integer polynomials  $g(x)=\varphi_{(p-1)^{\ell},c}^m(x)-x$ of fixed degree $\upsilon=(p-1)^{m\ell}$,  and let $W(\mathbb{Z})^{\text{max}}\subset W(\mathbb{Z})^{\text{irr}}$ be a subset consisting of irreducible polynomials $g$ such that $R_{g}=\mathbb{Z}[x]\slash (g(x))$ is a maximal order in $\mathbb{Q}_{g}=\mathbb{Q}[x]\slash (g(x))$. Following again \cite{Nico}, it also follows here that $G_{a}(\mathbb{Z})=\mathbb{Z}$ necessarily acts naturally on $W(\mathbb{Z})^{\text{irr}}$ via translation, namely, $(b \cdot g)(x):= g(x+b)$ for every $b\in \mathbb{Z}$ and for every $g\in W(\mathbb{Z})^{\text{irr}}$; and moreover, this action of $G_{a}(\mathbb{Z})=\mathbb{Z}$ by translation also necessarily preserves each of $W(\mathbb{Z})^{\text{irr}}$ and $W(\mathbb{Z})^{\text{max}}$. Now let $\mathfrak{F}_{2}$ be a family consisting of the $\mathbb{Z}$-orbits on $W(\mathbb{Z})^{\text{max}}$. It then follows (from \cite{Nico}) that the family $\mathfrak{F}_{2}$ necessarily parametrizes degree-$\upsilon$ monogenized fields $(\mathbb{Q}_{g}, \beta)$ up to isomorphism. As before, we also note that (from [\cite{Nico}, Subsect.2.3]) this same family $\mathfrak{F}_{2}$ parametrizing  degree-$\upsilon$ monogenized number fields $(\mathbb{Q}_{g}, \beta)$ is also the family of associated $L$-functions $L(s, \rho_{\mathbb{Q}_{g}})$. By again, taking great advantage of [\cite{Nico}, Thm. 1.1], we then obtain the following corollary on $\mathfrak{F}_{2}$:

\begin{cor}
Assume Corollary \ref{7.2}, and let $\mathfrak{F}_{2}$ be as before. Then $\mathfrak{F}_{2}$ parametrizing monogenized degree-$\upsilon$ fields ordered by height $h(g)$ as defined in \textnormal{\cite{Nico}} satisfies Sato-Tate equidistribution in the sense of \textnormal{[\cite{Sar}, Conj.1]}. 
\end{cor}

\begin{proof}
By applying a similar argument as in the Proof of Corollary \ref{11.1}, it then also follows immediately that the family $\mathfrak{F}_{2}$ satisfies  Sato-Tate equidistribution in the sense of \textnormal{[\cite{Sar}, Conjecture 1]} as also indeed needed.
\end{proof} 

Similarly, recall again from the second part of Theorem \ref{2.3} that the monic polynomial $f(x) = \varphi_{p^{\ell},c}^m(x)-x\in \mathcal{O}_{K}[x]$ is irreducible modulo prime $p\mathcal{O}_{K}$; and which consequently for each such $f\in \mathcal{O}_{K}[x]$ induces a number field $K_{f}\slash \mathbb{Q}$ of degree $t=np^{m\ell}$, for every fixed degree $n=[K: \mathbb{Q}]$. Moreover, it also follows from the primitive element theorem that we can write $K_{f}= \mathbb{Q}(\gamma)\simeq \mathbb{Q}[x]\slash (h_{1}(x)) =: \mathbb{Q}_{h_{1}}$, where $\gamma$ is some algebraic number in $K_{f}$ and $h_{1}\in \mathbb{Q}[x]$ is the characteristic polynomial (also the minimal polynomial) of $\gamma$. Now for every degree-$t=np^{m\ell}$ number field $\mathbb{Q}_{h_{1}}$ induced from a polynomial discrete dynamical system in second part of Theorem \ref{2.3}, we then also attach to $\mathbb{Q}_{h_{1}}$ a zeta function $\zeta_{\mathbb{Q}_{h_{1}}}$; and which also factors $\zeta_{\mathbb{Q}_{h_{1}}}(s)=\zeta(s)L(s, \rho_{\mathbb{Q}_{h_{1}}}$), where $L(s, \rho_{\mathbb{Q}_{h_{1}}})$ is the Artin $L$-function corresponding to a representation $\rho_{\mathbb{Q}_{h_{1}}}: \text{Gal}(M^{(t)}_{h_{1}}\slash \mathbb{Q}) \hookrightarrow S_{t}\to \text{GL}_{t-1}(\mathbb{C})$, and $M^{(t)}_{h_{1}}$ is the normal closure of $\mathbb{Q}_{h_{1}}$. Now since $\gamma\in K_{f}$ can be some element such that $K_{f}=\mathbb{Q}(\gamma)\simeq \mathbb{Q}_{h_{1}}$, we then also wish to study the distribution of $L$-functions $L(s, \rho_{\mathbb{Q}_{h_{1}}})$ arising from fields $\mathbb{Q}_{h_{1}}$ induced by irreducible monic polynomials $h_{1}\in \mathbb{Z}[x]$. To do so, we again let $V^{(t)}(\mathbb{Z})^{\text{irr}}$ be the space consisting of irreducible monic polynomials  $h_{1}\in \mathbb{Z}[x]$ of fixed degree $t=np^{m\ell}$,  and let $V^{(t)}(\mathbb{Z})^{\text{max}}\subset V^{(t)}(\mathbb{Z})^{\text{irr}}$ be a subset consisting of irreducible monic polynomials $h_{1}\in \mathbb{Z}[x]$ such that $R_{h_{1}}=\mathbb{Z}[x]\slash (h_{1}(x))$ is a maximal order in $\mathbb{Q}_{h_{1}}=\mathbb{Q}[x]\slash (h_{1}(x))$. As before, the additive group $G_{a}(\mathbb{Z})=\mathbb{Z}$ acts naturally on the space $V^{(t)}(\mathbb{Z})^{\text{irr}}$ via translation, namely, $(b \cdot h_{1})(x):= h_{1}(x+b)$ for every $b\in \mathbb{Z}$ and  every $h_{1}\in V^{(t)}(\mathbb{Z})^{\text{irr}}$; and moreover this action of $G_{a}(\mathbb{Z})=\mathbb{Z}$ by translation again preserves each of $V^{(t)}(\mathbb{Z})^{\text{irr}}$ and $V^{(t)}(\mathbb{Z})^{\text{max}}$. So now, let $\mathfrak{F}^{(t)}_{1}$ be a family consisting of the $\mathbb{Z}$-orbits on $V^{(t)}(\mathbb{Z})^{\text{max}}$; and which as before parametrizes degree-$t$ monogenized number fields $(\mathbb{Q}_{h_{1}}, \gamma)$ over $\mathbb{Q}$ up to isomorphism. Note that as before, we also treat  the family $\mathfrak{F}^{(t)}_{1}$ to be the family of associated $L$-functions $L(s, \rho_{\mathbb{Q}_{h_{1}}})$. But now, by again taking  great advantage of [\cite{Nico}, Theorem 1.1], we then also immediately obtain the following corollary on  $\mathfrak{F}^{(t)}_{1}$:

\begin{cor}\label{cor11.3}
Assume second part of Theorem \ref{2.3}, and let $t=np^{m\ell}$ be any fixed odd integer. Let $\mathfrak{F}^{(t)}_{1}$ be a family of $\mathbb{Z}$-orbits defined as before. Then the family $\mathfrak{F}^{(t)}_{1}$ parametrizing monogenized degree-$t$ number fields ordered by height $h(h_{1})$ as given in \textnormal{\cite{Nico}} satisfies Sato-Tate equidistribution in the sense of \textnormal{[\cite{Sar}, Conjecture 1]}. 
\end{cor} 

\begin{proof}
To see this, we note that by second part of Theorem \ref{2.3}, it then follows that the family of number fields $K_{f}\slash \mathbb{Q}$ of degree $t = np^{m\ell}$ is not empty. Moreover, it then also follows from the discussion (right before the corollary that we are proving) that the family of degree-$t$ algebraic number field $\mathbb{Q}_{h_{1}}\slash \mathbb{Q}$ is non-empty. But now, by applying again a similar argument as in the Proof of Corollary \ref{11.1}, it then also follows immediately that the family $\mathfrak{F}^{(t)}_{1}$ satisfies  Sato-Tate equidistribution in the sense of \textnormal{[\cite{Sar}, Conjecture 1]} as also indeed needed.
\end{proof}

As before, recall also from the second part of Theorem \ref{3.3} that the monic polynomial $g(x) = \varphi_{(p-1)^{\ell},c}^m(x)-x\in \mathcal{O}_{K}[x]$ is irreducible modulo prime $p\mathcal{O}_{K}$; and which consequently for each such $g\in \mathcal{O}_{K}[x]$ induces a number field $L_{g}\slash \mathbb{Q}$ of degree $r=n(p-1)^{m\ell}$, for every fixed $n=[K: \mathbb{Q}]$. Moreover, it also follows from the primitive element theorem that we can write $L_{g}= \mathbb{Q}(\nu)\simeq \mathbb{Q}[x]\slash (h_{2}(x)) =: \mathbb{Q}_{h_{2}}$, where $\nu$ is some algebraic number in $L_{g}$ and $h_{2}\in \mathbb{Q}[x]$ is the characteristic polynomial (also the minimal polynomial) of $\nu$. Now for every degree-$r=n(p-1)^{m\ell}$ number field $\mathbb{Q}_{h_{2}}$ induced from a polynomial discrete dynamical system in the second part of Theorem \ref{3.3}, we then also attach to $\mathbb{Q}_{h_{2}}$ a zeta function $\zeta_{\mathbb{Q}_{h_{2}}}$; which also factors $\zeta_{\mathbb{Q}_{h_{2}}}(s)=\zeta(s)L(s, \rho_{\mathbb{Q}_{h_{2}}}$), where $L(s, \rho_{\mathbb{Q}_{h_{2}}})$ is the Artin $L$-function corresponding to a representation $\rho_{\mathbb{Q}_{h_{2}}}: \text{Gal}(M^{(r)}_{h_{2}}\slash \mathbb{Q}) \hookrightarrow S_{r}\to \text{GL}_{r-1}(\mathbb{C})$, and $M^{(r)}_{h_{2}}$ is the normal closure of $\mathbb{Q}_{h_{2}}$. Now because $\nu\in L_{g}$ can be some element such that $L_{g}=\mathbb{Q}(\nu)\simeq \mathbb{Q}_{h_{2}}$, we then also wish to study the distribution of $L$-functions $L(s, \rho_{\mathbb{Q}_{h_{2}}})$ arising from fields $\mathbb{Q}_{h_{2}}$ induced by irreducible monic polynomials $h_{2}\in \mathbb{Z}[x]$. As before, we let $W^{(r)}(\mathbb{Z})^{\text{irr}}$ be the space consisting of irreducible monic polynomials  $h_{2}\in \mathbb{Z}[x]$ of fixed degree $r=n(p-1)^{m\ell}$,  and let $W^{(r)}(\mathbb{Z})^{\text{max}}\subset W^{(r)}(\mathbb{Z})^{\text{irr}}$ be a subset consisting of irreducible monic polynomials $h_{2}\in \mathbb{Z}[x]$ such that $R_{h_{2}}=\mathbb{Z}[x]\slash (h_{2}(x))$ is a maximal order in $\mathbb{Q}_{h_{2}}=\mathbb{Q}[x]\slash (h_{2}(x))$. Again, the additive group $G_{a}(\mathbb{Z})=\mathbb{Z}$ acts naturally on $W^{(r)}(\mathbb{Z})^{\text{irr}}$ via translation, namely, $(b \cdot h_{2})(x):= h_{2}(x+b)$ for every $b\in \mathbb{Z}$ and  every $h_{2}\in W^{(r)}(\mathbb{Z})^{\text{irr}}$; and moreover this action of $G_{a}(\mathbb{Z})=\mathbb{Z}$ by translation  preserves each of the spaces $W^{(r)}(\mathbb{Z})^{\text{irr}}$ and $W^{(r)}(\mathbb{Z})^{\text{max}}$. So now, let $\mathfrak{F}^{(r)}_{2}$ be a family consisting of the $\mathbb{Z}$-orbits on the space $W^{(r)}(\mathbb{Z})^{\text{max}}$; and also be the family parametrizing degree-$r$ monogenized fields $(\mathbb{Q}_{h_{2}}, \nu)$ over $\mathbb{Q}$ up to isomorphism. As before, we treat the family $\mathfrak{F}^{(r)}_{2}$ to be the family of associated $L$-functions $L(s, \rho_{\mathbb{Q}_{h_{2}}})$. But now by taking  great advantage of [\cite{Nico}, Theorem 1.1], we then also obtain the following corollary on the family $\mathfrak{F}^{(r)}_{2}$:

\begin{cor}
Assume second part of Theorem \ref{3.3}, and let $r=n(p-1)^{m\ell}$ be any fixed odd integer. Let $\mathfrak{F}^{(r)}_{2}$ be a family of $\mathbb{Z}$-orbits defined as before. Then the family $\mathfrak{F}^{(r)}_{2}$ parametrizing monogenized degree-$r$ number fields ordered by height $h(h_{2})$ as given in \textnormal{\cite{Nico}} satisfies Sato-Tate equidistribution in the sense of \textnormal{[\cite{Sar}, Conjecture 1]}. 
\end{cor} 

\begin{proof}
By applying a similar argument as in Proof of Corollary \ref{cor11.3}, we then obtain the conclusion as needed.
\end{proof}

\section*{\textbf{Acknowledgments}}
I’m truly very grateful and deeply indebted to Dr. Ilia Binder and Dr. Arul Shankar, and along with Dr. Jacob Tsimerman for everything. This work and my studies were hugely and wholeheartedly funded by Dr. Binder and Dr. Shankar. Any opinions expressed in this article belong solely to me, the author, Brian Kintu; and should never be taken as a reflection of the views of anyone that has been happily acknowledged by the author.

\bibliography{References}
\bibliographystyle{plain}

\noindent Dept. of Math. and Comp. Sciences (MCS), University of Toronto, Mississauga, Canada \newline
\textit{E-mail address:} \textbf{brian.kintu@mail.utoronto.ca}\newline 
\date{\small{\textit{February 22, 2026}}}

\end{document}